\documentclass[12pt, reqno]{amsart}

\usepackage{amsthm}
\usepackage{amssymb, amsmath}

\usepackage{amscd}   
\usepackage[all]{xy} 
\usepackage{colonequals}
\usepackage{tikz-cd, textcomp}

\usepackage[backref, colorlinks=true, linkcolor=blue, citecolor=blue]{hyperref}
\usepackage[alphabetic,backrefs,lite]{amsrefs}

\usepackage[]{xcolor}

\usepackage[top=1in, bottom =1in, left=1in, right = 1in]{geometry}

\usepackage{comment}
\DeclareFontEncoding{OT2}{}{} 

\hypersetup{
pdflang={en-US},
pdftitle={Curve classes on conic bundle threefolds and applications to rationality},
pdfauthor={Sarah Frei, Lena Ji, Soumya Sankar, Bianca Viray, and Isabel Vogt},
pdfsubject={Algebraic Geometry},
pdfkeywords={conic bundles, rationality, rationality obstructions, prym varieties, intermediate jacobians}
}

\makeatletter
\renewcommand*\env@matrix[1][*\c@MaxMatrixCols c]{%
  \hskip -\arraycolsep
  \let\@ifnextchar\new@ifnextchar
  \array{#1}}
\makeatother

\newtheorem{lemma}{Lemma}[section]
\newtheorem{theorem}[lemma]{Theorem}

\newtheorem{prop}[lemma]{Proposition}
\newtheorem{cor}[lemma]{Corollary}

\newtheorem{claim*}{Claim}
\newtheorem{thm}[lemma]{Theorem}
\newtheorem{defn}[lemma]{Definition}
\newtheorem{example}[lemma]{Example}

\theoremstyle{remark}
\newtheorem{remark}[lemma]{Remark}

\newtheorem{rmk}[lemma]{Remark}

\newcommand{\PP}{{\mathbb P}}
\newcommand{\C}{{\mathbb C}}
\newcommand{\F}{{\mathbb F}}
\newcommand{\Q}{{\mathbb Q}}
\newcommand{\RR}{{\mathbb R}}
\newcommand{\Z}{{\mathbb Z}}

\newcommand{\Qbar}{{\overline{\Q}}}

\newcommand{\kperf}{k^{\textup{p}}}
\newcommand{\kbar}{{\overline{k}}}

\newcommand{\Ybar}{Y_{\kbar}}

\newcommand{\Wbar}{W_{\kbar}}

\newcommand{\kk}{{\mathbf k}}

\newcommand{\phibar}{{\overline{\phi}}}

\newcommand{\calD}{{\mathcal D}}
\newcommand{\calE}{{\mathcal E}}
\newcommand{\calF}{{\mathcal F}}

\newcommand{\calL}{{\mathcal L}}

\newcommand{\calO}{{\mathcal O}}

\newcommand{\calQ}{{\mathcal Q}}

\newcommand{\OO}{{\mathcal O}}

\newcommand{\frakS}{{\mathfrak S}}

\newcommand{\mm}{{\mathfrak m}}

\usepackage[mathscr]{euscript}

\newcommand{\scrF}{{\mathscr F}}

\newcommand{\scrI}{{\mathscr I}}

\DeclareMathOperator{\HH}{H}
\DeclareMathOperator{\hh}{h}

\DeclareMathOperator{\im}{im}

\DeclareMathOperator{\Aut}{Aut}
\DeclareMathOperator{\Gal}{Gal}

\DeclareMathOperator{\Br}{Br}
\DeclareMathOperator{\Gr}{Gr}

\DeclareMathOperator{\Sym}{Sym}

\DeclareMathOperator{\Pic}{Pic}
\DeclareMathOperator{\bPic}{{\bf Pic}}

\DeclareMathOperator{\Spec}{Spec}

\DeclareMathOperator{\PGL}{PGL}

\DeclareMathOperator{\fppf}{fppf}

\DeclareMathOperator{\red}{red}

\DeclareMathOperator{\rank}{rank}

\DeclareMathOperator{\disc}{disc}
\DeclareMathOperator{\id}{id}

\DeclareMathOperator{\CH}{CH}
\DeclareMathOperator{\Prym}{Prym}
\DeclareMathOperator{\bPrym}{\mathbf{Prym}}
\DeclareMathOperator{\NS}{NS}
\DeclareMathOperator{\alg}{alg}

\DeclareMathOperator{\PPrym}{\mathbf{PPrym}}

\DeclareMathOperator{\Tor}{Tor}
\DeclareMathOperator{\SK}{SK}
\DeclareMathOperator{\K}{K}

\DeclareMathOperator{\bCH2}{\mathbf{CH}^2}

\DeclareMathOperator{\Bl}{Bl}
\DeclareMathOperator{\reduc}{red}
\newcommand{\eps}{\varepsilon}
\newcommand{\Xbar}{X_{\kbar}}
\newcommand{\ejp}{\eps_*j_*p^*}
\newcommand{\pje}{p_*j^*\eps^*}

\newcommand{\isom}{\simeq}

\newcommand{\Deltatilde}{\tilde{\Delta}}
\newcommand{\gammatilde}{\tilde{\gamma}}
\newcommand{\Ptilde}{\tilde{P}}
\newcommand{\Ptildem}[1]{\tilde{P}^{(#1)}}
\newcommand{\Pm}[1]{{P}^{(#1)}}
\newcommand{\Stildem}[1]{\tilde{S}^{(#1)}}
\newcommand{\Sm}[1]{{S}^{(#1)}}
\newcommand{\piDelta}{\varpi}
\newcommand{\piX}{\pi}

\newcommand{\DoubleCover}{Y_{\Deltatilde/\Delta}}
\newcommand{\piDoubleCover}{\tilde\pi}
\newcommand{\piQuadricSurface}{\piX_1}
\newcommand{\GammaZ}{\Gamma_Z}
\newcommand{\algequiv}{\sim_{\alg}}
\newcommand{\lattice}{L}

\DeclareMathOperator{\mult}{mult}

\numberwithin{equation}{section}
\numberwithin{table}{section}

\newcommand{\defi}[1]{\textsf{#1}} 

\newcommand{\pt}{w}

\title{Curve classes on conic bundle threefolds and applications to rationality}

\author{Sarah Frei}
\address{Department of Mathematics, Dartmouth College, 27 N. Main Street, Hanover, NH 03755}
\email{sarah.frei@dartmouth.edu}
\urladdr{http://math.dartmouth.edu/\~{}sfrei}

\author{Lena Ji}
\address{Department of Mathematics, University of Michigan, 530 Church Street, Ann Arbor, MI 48109-1043}
\email{lenaji.math@gmail.com}
\urladdr{http://www-personal.umich.edu/\~{}lenaji}

\author{Soumya Sankar}
\address{Department of Mathematics, The Ohio State University, 231 West 18th Avenue,
Columbus, OH 43210-1174}
\email{sankar.40@osu.edu}
\urladdr{https://math.osu.edu/people/sankar.40}

\author{Bianca Viray}
\address{University of Washington, Department of Mathematics, Box 354350, Seattle, WA 98195, USA}
\email{bviray@uw.edu}
\urladdr{http://math.washington.edu/\~{}bviray}

\author{Isabel Vogt}
\address{Brown University, Department of Mathematics, Box 1917, 151 Thayer Street, Providence, RI 02912, USA}
\email{ivogt.math@gmail.com}
\urladdr{https://www.math.brown.edu/ivogt/}


\begin{document}

\begin{abstract}
    We undertake a study of conic bundle threefolds \(\pi\colon X\to W\) over geometrically rational surfaces whose associated discriminant covers \(\Deltatilde\to\Delta\subset W\) are smooth and geometrically irreducible. We first show that the structure of the Galois-module \(\CH^2\Xbar\) of rational equivalence classes of curves is captured by a group scheme that is a generalization of the Prym variety of \(\Deltatilde\to\Delta\).  This generalizes Beauville's result that the algebraically trivial curve classes on \(\Xbar\) are parametrized by the Prym variety.
    
    We apply our structural result on curve classes to study the refined intermediate Jacobian torsor (IJT) obstruction to rationality introduced by Hassett--Tschinkel and Benoist--Wittenberg. The first case of interest is \(W = \PP^2\) and \(\Delta\) is a smooth plane quartic.  In this case, we show that the IJT obstruction characterizes rationality when the ground field has less arithmetic complexity (precisely, when the \(2\)-torsion in the Brauer group of the ground field is trivial).  We also show that a hypothesis of this form is necessary by constructing,
    over any \(k \subset\mathbb R\), a conic bundle threefold with \(\Delta\) a smooth quartic where the IJT obstruction vanishes, yet \(X\) is irrational over \(k\).
\end{abstract}

\maketitle

\section{Introduction}

    Let \(\piX\colon X\to W\) be a conic bundle threefold over a smooth geometrically rational surface \(W\), all defined over a field \(k\) of characteristic different from \(2\).  Despite the simple geometry of the base and fibers of this morphism, the family of all such conic bundle threefolds exhibits rich geometric variation.  Indeed, this family of conic bundles includes: rational varieties (e.g., when \(W = \PP^2\) and \(\pi\) has a rational section), geometrically rational non-rational varieties~\cite{bw-cg}*{Theorem 1.1}, unirational but non-stably rational varieties~\cite{AM-Luroth}, and stably rational but non-rational varieties \cite{BCTS-Zariski}*{Th\'eor\`eme 3}. In addition, a folklore conjecture predicts that this family even contains geometrically non-unirational varieties (see \cite{prokhorov-rationality-conic-bundles}*{Sect.~14.2}).
    
    In this paper we restrict to conic bundle threefolds whose total space is smooth. Furthermore, we assume that the singular fibers of \(\pi\) lie above a smooth, geometrically irreducible curve \(\Delta\subset W\) and that \(\pi^{-1}(\Delta)\) is geometrically irreducible (i.e., the conic bundle is \defi{geometrically ordinary} and \defi{geometrically standard}, c.f.~Section \ref{sec:Notation}). Such a conic bundle gives rise to a geometrically irreducible \'etale double cover \(\piDelta\colon \Deltatilde\to \Delta\), known as the \defi{discriminant cover} of \(X\), which
    controls much of the geometry of the conic bundle. 
    
    Indeed, over an \emph{algebraically closed} field, the discriminant cover uniquely determines the isomorphism class of the generic fiber of \(\pi\). In addition, Beauville proved that the subgroup \((\CH^2 X)^0\subset \CH^2 X\) of algebraically trivial curve classes modulo rational equivalence is isomorphic to the Prym variety \(\Prym_{\Deltatilde/\Delta}\)~\cite{beauville-ij}*{Thm. 3.1}.  By translating, we obtain \((\CH^2 X)^\gamma \simeq\Prym_{\Deltatilde/\Delta}\) for any \(\gamma\in \NS^2 X\).

    Over \emph{nonclosed} fields, however, the discriminant cover need not determine the isomorphism class of the generic fiber of \(\pi\).  Despite this, we prove that the discriminant cover characterizes curve classes, and, in fact, determines the structure of \((\CH^2 \Xbar)^\gamma\) as a torsor under \((\CH^2 \Xbar)^0\) for any \(\gamma\in (\NS^2\Xbar)^{G_k}\).
    \begin{thm}[{Follows from Theorem~\ref{thm:IJTtorsorsOfConicBundles}}]\label{thm:MainTorsors}
        Let \(k\) be a field of characteristic different from \(2\), and let \(\pi\colon X\to W\) be a smooth conic bundle over a smooth geometrically rational surface, all defined over \(k\). Assume that the discriminant cover \(\Deltatilde\to \Delta\subset W\) is smooth and geometrically irreducible.  
        Then there is a Galois-equivariant surjective group homomorphism from \(\CH^2 \Xbar\) to the \(\kbar\)-points of the \defi{\((\bPic_{W/k})\)-polarized Prym scheme},
        \[
            \PPrym_{\Deltatilde/\Delta}^{\bPic_{W/k}} \colonequals \bPic_{\Deltatilde/k}\times_{\bPic_{\Delta/k}}\bPic_{W/k}
        \]
        that restricts to an isomorphism \((\CH^2\Xbar)^0 \xrightarrow{\sim} \Prym_{\Deltatilde/\Delta}(\kbar)\).  In particular, for every
        \(\gamma\in \NS^2\Xbar\) this homomorphism induces an isomorphism between \((\CH^2\Xbar)^{\gamma}\) and the \(\kbar\)-points of a geometric connected component of \(\PPrym_{\Deltatilde/\Delta}^{\bPic_{W/k}}\) that is compatible with the actions of the identity components.
    \end{thm}
    \begin{remark}
        Note that the connected component of the identity in \(\PPrym_{\Deltatilde/\Delta}^{\bPic_{W/k}}\) is the usual Prym variety, which justifies the terminology.
    \end{remark}
    Since \(\PPrym_{\Deltatilde/\Delta}^{\bPic_{W/k}}\) is a \(k\)-scheme, these isomorphisms show that \((\CH^2\Xbar)^{0}\) (also known as the intermediate Jacobian of \(X\)) descends to an abelian variety over \(k\) and that, for \(\gamma\in (\NS^2\Xbar)^{G_k}\), the connected components \((\CH^2 \Xbar)^{\gamma}\) descend to torsors defined over \(k\).  
    In the case that \(k\subset\C\), the descent results for the intermediate Jacobian follow from work of  Achter, Casalaina-Martin, and Vial~\cite{ACMV-models}*{Thm. A}.  When \(k = \RR\), the descent results for the torsors follow from work of Hassett and Tschinkel~\cite{HT-intersection-quadrics}.  These results were extended to geometrically rational threefolds over arbitrary (not necessarily perfect) fields by Benoist and Wittenberg~\cite{bw-ij} (see~\cite{HT-cycle} for a proof over subfields of \(\C\)).  Benoist and Wittenberg proved their descent results by constructing  a codimension 2 Chow \(k\)-scheme \(\bCH2\) whose \(\kbar\)-points agree (Galois-equivariantly) with the \(\kbar\)-points of the codimension 2 Chow group \cite{bw-ij}.
    Thus in the case that \(\piX\colon X \to W\) is geometrically rational, our theorem gives the following description of the connected components of the codimension 2 Chow scheme defined by Benoist and Wittenberg.
    \begin{thm}[Follows from Theorem~\ref{thm:IJTtorsorsOfConicBundles}]\label{thm:MainTorsorScheme}
        Let \(k\) be a field of characteristic different from \(2\), and let \(\pi\colon X\to W\) be a smooth \emph{geometrically rational} conic bundle over a smooth geometrically rational surface \(W\), all defined over \(k\), with smooth and geometrically irreducible discriminant cover \(\Deltatilde\to \Delta\). Then every connected component of \(\bCH2_{X/k}\) is isomorphic to a connected component of the \((\bPic_{W/k})\)-polarized Prym scheme. 
    \end{thm}

In order to prove Theorem~\ref{thm:MainTorsorScheme}, we need an explicit morphism between \(\bCH2_{X/k}\) and the \((\bPic_{W/k})\)-polarized Prym scheme. This requires the construction of pushforward morphisms between codimension 2 Chow schemes as well as morphisms between codimension 2 Chow schemes and Picard schemes of curves. The existence of these morphisms, introduced in Section~\ref{sec:BWreview}, extend general results of \cite{bw-ij} and may be of independent interest in the study of codimension 2 Chow schemes of geometrically rational threefolds.
    
    \subsection{Applications to rationality}
    	Although a geometric rationality criterion for conic bundles over minimal rational surfaces has been known since the 1980s \cite{Shokurov-Prym}, a rationality criterion over nonclosed fields has remained stubbornly out of reach, even when the base is \(\PP^2\).  If \(W = \PP^2\), then one can verify that the rationality constructions over \(\kbar\) descend to give rational parametrizations whenever \(\deg\Delta\leq 3\) and \(X(k)\neq\emptyset\) (see Section~\ref{sec:low_degree}).  Thus, the natural case of interest is when \(W = \PP^2\) and \(\deg \Delta = 4\).  Contributing to the difficulty in this case is the fact that several of the classical rationality obstructions vanish under mild additional assumptions. Indeed, results for conic bundle surfaces show that \(X\) is unirational whenever \(X_{(\PP^2-\Delta)}(k)\neq\emptyset\) (see Section~\ref{sec:low_degree}), and the Artin--Mumford exact sequence shows that \(\Br X = \Br \PP^3\) whenever \(\Deltatilde\) is geometrically irreducible \cite{Poonen-Qpoints}*{Theorem 6.8.3, Proof of Proposition 6.9.15}; hence the Brauer group obstruction vanishes.
    	Additionally, if \(\Delta\) is smooth, then results of Bruin~\cite{bruin} combined with Theorem~\ref{thm:MainTorsors} above show that the intermediate Jacobian of \(X\) is always isomorphic to the Jacobian of a genus \(2\) curve; hence, Clemens and Griffiths's intermediate Jacobian obstruction \cite{Clemens-Griffiths-ij} vanishes.  Furthermore, the birational automorphism group of \(X\) contains the automorphism group of the generic fiber, which is an infinite orthogonal group by~\cite{Borel}*{Corollary 18.3}.
     Hence, the obstruction of Iskovskikh and Manin \cite{IskovskihManin} vanishes.

 Recently, Hassett and Tschinkel~\cite{HT-intersection-quadrics} refined the intermediate Jacobian obstruction over \(\RR\) to the \defi{intermediate Jacobian torsor} or \defi{IJT} obstruction. This was extended to arbitrary, not necessarily perfect, fields by Benoist and Wittenberg~\cite{bw-ij} (see also~\cite{HT-cycle} for \(k\subset \C\)).
 These authors showed that if \(Y\) is a \(k\)-rational threefold then \((\CH^2 \Ybar)^0\) descends, over \(k\), to an abelian variety that is isomorphic to the Jacobian of a smooth curve \(\Gamma\) (this is an extension of Clemens and Griffiths's classical intermediate Jacobian obstruction to nonclosed fields). In addition, these authors showed that, in the case that \(\Gamma\) is connected and of genus \(\geq 2\),  for all \(\gamma\in (\NS^2\Ybar)^{G_k}\) the connected component \((\CH^2\Ybar)^\gamma\) of curve classes algebraically equivalent to \(\gamma\) descends, over \(k\), to a torsor that is isomorphic to \(\bPic^i_{\Gamma/k}\) for some \(i\). 
        
        This new obstruction has proved remarkably powerful in characterizing rationality of threefolds.  
        Over \(\RR\), Hassett and Tschinkel used the IJT obstruction to characterize rationality for intersections of quadrics in \(\PP^5\)~\cite{HT-intersection-quadrics}.  Benoist and Wittenberg extended this result to intersections of two quadrics in \(\PP^5\) over arbitrary, not necessarily perfect, fields~\cite{bw-ij} (see~\cite{HT-cycle} for a proof in the case of arbitrary subfields of \(\C\)).
        Later, Kuznetsov and Prokhorov employed the IJT obstruction to dramatic effect
        to characterize rationality for all remaining Fano threefolds of geometric Picard rank \(1\) in characteristic zero~\cite{KP-Fano-3folds-rank1} (for several of these cases, the existence of a \(k\)-point already implies rationality, and the IJT obstruction is required for the remaining cases).  
        
        The explicit description we give of the torsors in Theorem~\ref{thm:MainTorsorScheme} enables us to complete an in-depth study of the IJT obstruction for conic bundle threefolds.  We focus on a  family of Fano conic bundles with \(W = \PP^2\), \(\Delta\) a smooth plane quartic, and \(X(k)\neq \emptyset\), so in particular, all classical rationality obstructions described above  vanish.  For this family we show that although the IJT obstruction can detect irrationality when other obstructions cannot, in contrast to the Picard rank \(1\) case,  the IJT obstruction alone is \emph{not} strong enough to characterize rationality.
        \begin{thm}\label{thm:NoRatlityObsSuff}
            Let \(k\) be a field of characteristic different from \(2\). Let \(Y\) be a double cover of \(\PP^1\times \PP^2\) branched over a smooth \((2,2)\) divisor.  Then \(Y\) is Fano, and the second projection endows \(Y\) with the structure of a geometrically standard, geometrically rational conic bundle over \(\PP^2\) with \(\Delta\) a plane quartic. Moreover, this conic bundle is geometrically ordinary whenever \(\Delta\) is smooth.
            In particular, if \(\Delta\) is smooth, then there is no Brauer group obstruction, intermediate Jacobian obstruction, nor birational automorphism group obstruction to rationality.  Furthermore,
            \begin{enumerate}
                \item There exists such an \(Y/\Q\) with \(Y(\Q) \neq\emptyset\) and \emph{no} IJT obstruction to rationality; however, \(Y(\RR)\) is disconnected, so \(Y\) is irrational over any subfield of \(\RR\); and\label{part:IJTnotenough} 
                \item There exists such an \(Y/\Q\) with \(Y(\Q)\neq\emptyset\) and \(Y(\RR)\) connected; however, there is an IJT obstruction over \(\RR\), and hence \(Y\) is irrational over any subfield of \(\RR\).\footnote{Furthermore, the unramified cohomology groups of \(Y_{\RR}\) are trivial \cite{bw-cg}*{Theorem 1.4}.} \label{part:RRtopnotenough}
            \end{enumerate}
        \end{thm}
        
        The failure of the IJT obstruction to characterize rationality for this family of conic bundles is intimately related to the nontriviality of \(\Br \RR\).  Morally, this failure occurs because of the dichotomy discussed towards the beginning of the introduction: the discriminant cover \(\Deltatilde \to \Delta\) completely determines the codimension \(2\) Chow scheme, but can fail to determine the isomorphism class of the generic fiber.  More precisely, given two conic bundles  \(\pi \colon X \to W\) and \(\pi' \colon X' \to W\) over a field \(k\) that have the same discriminant cover, the classes of the generic fibers \([X_{\kk(W)}]\) and \([X'_{\kk(W)}]\) in \((\Br \kk(W))[2]\) differ by a class from \((\Br k)[2]\).  When \((\Br k)[2]\) is trivial, this discrepancy disappears and we prove that the IJT obstruction \emph{does} characterize rationality in this case.
        \begin{thm}\label{thm:Br2trivial-IJT}
            Let \(k\) be a field of characteristic different from \(2\) with \((\Br k)[2] = 0\), and let \(\piX \colon X \to \PP^2\) be a geometrically standard conic bundle with \(\Delta\) a smooth plane quartic.
            Then \(X\) is \(k\)-rational if and only if the intermediate Jacobian torsor obstruction vanishes. 
        \end{thm}

If \(k\) is a finite field of odd characteristic, then Theorem~\ref{thm:Br2trivial-IJT} implies that \(X\) is necessarily \(k\)-rational, since the intermediate Jacobian torsors are automatically trivial (Corollary~\ref{cor:rationality-over-finite-fields}). In particular, the irrational examples of Theorem~\ref{thm:NoRatlityObsSuff} have rational reductions modulo all primes of good reduction.
However, over infinite fields with \((\Br k)[2]=0\), there can still exist irrational degree \(4\) conic bundles \(X\) (Example~\ref{exmp:Brk=0-IJT-example}).

The techniques involved in the proof of Theorem~\ref{thm:Br2trivial-IJT} can be leveraged to give rationality criteria even over fields with non-trivial \(2\)-torsion in the Brauer group.  Indeed, the second author and M. Ji use these techniques (and the description of the torsors given in Theorem~\ref{thm:IJ-torsors-double-cover}) to characterize \(\RR\)-rationality of the aforementioned double covers of \(\PP^1\times\PP^2\) branched over a \((2,2)\) divisor in four of the six cases of the real rigid isotopy class of the quartic curve \(\Delta\)
\cite{JJ-deg4}*{Theorem 1.2}. In these four cases, they use the real topological obstruction to control the additional data of \((\Br\RR)[2]\), showing that, in these cases, combinations of the IJT obstruction and the real topological obstruction \emph{together} characterize rationality. However, the question of whether the IJT obstruction {combined} with other obstructions characterizes rationality for \emph{all} such conic bundles remains open, as shown by the following.
\begin{example}\label{ex:NoRatlityObsOrConst}
    Let \(Y\) be the double cover of \(\PP^1_{[t_0:t_1]}\times \PP^2_{[u:v:w]}\) branched over the \((2,2)\) divisor
    \begin{gather*}
        t_0^2(-31u^2 + 12uv - 6v^2 + 4uw + 8vw + 25w^2) + 
        2t_0t_1(-25u^2 + 120uv + 30v^2 - 9uw - vw)\\ + 
        t_1^2(-8047u^2 + 1092uv -1446v^2 - 4uw - 7vw - 25w^2)=0.
    \end{gather*}
    Then \(Y\) is \(\Q\)-unirational, \(Y(\RR)\) is diffeomorphic to a 3-sphere, \(Y\) has trivial unramified cohomology groups over \(\RR\), and \(Y\) has no IJT obstruction.  However, there is no known rationality construction for \(Y\).
\end{example}

\subsection{Outline}
In Section~\ref{sec:BWreview}, we recall the construction from \cite{bw-ij} of the codimension 2 Chow scheme and then prove results about various pullback and pushforward morphisms between these schemes and the Picard scheme of a curve. These general results will be applied in Section~\ref{sec:IJ-torsors}.  A reader willing to accept these results when they are invoked in Lemma~\ref{lem:ejp-pje} can safely proceed to Section~\ref{sec:PPrymScheme} on a first reading.
In Section~\ref{sec:PPrymScheme}, we define the \(\phi\)-polarized Prym scheme for an \'etale double cover of curves (where \(\phi\) is a group scheme morphism to \(\bPic_{\Delta/k}\)), and we assemble known results about the Prym variety of \'etale double covers to deduce information about the structure of the \(\phi\)-polarized Prym scheme. Of particular interest is the case where the base curve of the \'etale double cover is a smooth plane quartic (Section~\ref{sec:UnramDoubleCoverDeg4}).
In Section~\ref{sec:IJ-torsors}, we relate the connected components of the \(\phi\)-polarized Prym scheme to cosets of \((\CH^2 \Xbar)^0\) and prove Theorems~\ref{thm:MainTorsors} and~\ref{thm:MainTorsorScheme}. In Section~\ref{sec:DoubleCover}, we specialize to conic bundles that arise as double covers of \(\PP^1\times \PP^2\) branched over a \((2,2)\) surface.  We show that these conic bundles arise naturally from \'etale double covers of smooth plane quartics, specialize our previous results to this case, and prove Theorem~\ref{thm:Br2trivial-IJT}.  In Section~\ref{sec:ProofsofMainThms}, we prove Theorem~\ref{thm:NoRatlityObsSuff} and the claims in Example~\ref{ex:NoRatlityObsOrConst}.

We end the paper with a brief section with some contextual results.  In particular, we review unirationality and rationality results for conic bundles over \(\PP^2\) with low degree discriminant curves, we recall results on the intermediate Jacobian obstruction for conic bundle threefolds, we explain how a theorem of Benoist and Wittenberg gives an example which complements Theorem~\ref{thm:Br2trivial-IJT} (Example~\ref{exmp:Brk=0-IJT-example}), and
we show that the threefolds in Theorem~\ref{thm:NoRatlityObsSuff} and Example~\ref{ex:NoRatlityObsOrConst} cannot be constructed using conic bundle structures on complete intersections of two quadrics.

\section*{Acknowledgements}

    This material is based upon work supported by the National Science Foundation under Grant No.~DMS-1439786 while the authors attended the 2020 Women in Algebraic Geometry Conference, hosted virtually at the Institute for Computational and Experimental Research in Mathematics in Providence, RI.  We thank the organizers of that conference---Melody Chan, Antonella Grassi, Julie Rana, Rohini Ramadas, and Isabel Vogt---for providing us the opportunity to work together and thank ICERM for providing the virtual tools to facilitate our collaboration.  We also thank Asher Auel, Olivier Benoist, Brendan Hassett, J\'anos Koll\'ar, Shizhang Li,  Bjorn Poonen, Vyacheslav Shokurov, and Olivier Wittenberg for helpful conversations. We especially thank the anonymous referee for a very thorough reading of this paper and helpful comments that have improved both the exposition and results. In particular, we thank the referee for suggesting Example~\ref{exmp:Brk=0-IJT-example} and for suggesting a different approach for proving Lemma~\ref{lem:ejp-pje}, which allowed us to prove the results that now appear in Section~\ref{sec:BWreview} in a more widely applicable context.
    
    During the preparation of this article, S.F.~was partially supported by NSF DMS-1745670; L.J.~was partially supported by NSF GRFP DGE-1656466, NSF DMS-1840234, and NSF MSPRF DMS-2202444; S.S.~was partially supported by NSF DMS-1928930; B.V.~was partially supported by NSF DMS-1553459, NSF DMS-2101434,  a Simons Fellowship, and the AMS Birman Fellowship; and I.V.~was partially supported by NSF MSPRF DMS-1902743 and NSF DMS-2200655. Additionally, this material is based partially upon work that was supported by National Science Foundation grant DMS-1928930 while B.V. and I.V. were in residence at the Simons Laufer Mathematical Sciences Institute in Berkeley, California, during the Spring 2023 semester.

\section{Notation and Conventions}\label{sec:Notation}

    Throughout the paper, \(k\) denotes a field and, with the exception of Section~\ref{sec:BWreview}, we assume throughout that \(k\) has characteristic different from \(2\). We fix an algebraic closure \(\kbar\) of \(k\) and write \(\kperf\) for the perfect closure of \(k\) in \(\kbar\). We will write \(G_k\) for the Galois group \(\Gal(\kbar/\kperf)\).
        
    By a variety over \(k\), we mean a reduced, equidimensional, separated scheme of finite type over \(k\).  Given a variety \(Z\) over \(k\) and a field extension \(k'/k\), we write \(Z_{k'}\) for the base change of \(Z\) to \(k'\).  If \(Z\) is integral, we denote by \(\kk(Z)\) its function field. We say \(Z\) is \defi{split} over \(k\) if it contains an open geometrically integral \(k\)-subscheme (see~\cite{CTS-Brauer-book}*{Section 10.1} for more about split varieties.).
        
    If \(Z\) is a smooth, proper, geometrically connected variety over \(k\), we write \(\CH^iZ\) for the group of codimension \(i\) cycles on \(Z\) modulo rational equivalence,  \((\CH^iZ)^0\subset\CH^iZ\) for the subgroup of algebraically trivial cycle classes, and \(\NS^i Z\) for the quotient \(\CH^iZ/(\CH^iZ)^0\), which is the group of codimension \(i\) cycles modulo algebraic equivalence.  We write 
    \(\bPic_{Z/k}\) for the Picard scheme of \(Z\), \(\Pic Z\) for the Picard group, and \(\Pic_{Z/k}\) for the absolute Picard functor. If \(D\) is a Cartier divisor on \(Z\), then \([D]\) denotes its class in \(\Pic Z\) and \(|D|\) denotes the associated complete linear system. A point \(z\in Z\) will refer to a scheme point unless otherwise specified.
    
    An \'etale double cover \(\piDelta\colon \Deltatilde\to\Delta\) of smooth proper geometrically irreducible curves arises as the quotient by a (geometrically) fixed-point-free
    involution \(\iota \colon \Deltatilde \to \Deltatilde\). The norm map \(\piDelta_*\colon \bPic_{\Deltatilde/k}\to\bPic_{{\Delta}/k}\) is defined by pushforward of divisors, or equivalently by the norms of transition functions defining a line bundle on an open cover.

    For a quadratic form \(Q\in k[x_1,\dots, x_n],\) the symmetric \(n\times n\) matrix \(M\) associated to \(Q\) is defined to be the symmetric matrix such that
    \(\vec{x} M \vec{x}^T = Q(\vec{x})\). If \(V(Q)\subset\mathbb P^{n-1}\) is smooth, then the discriminant of \(Q\) is defined to be \((-1)^{\binom{n}{2}}\det(M)\in k^{\times}/k^{\times2}\); equivalently, we can consider the \defi{discriminant extension} which is \(k(\sqrt{\disc(Q)})\).  If \(V(Q)\) is singular, then the discriminant is defined to be the discriminant of a maximal smooth plane section. If the rank of \(Q\) is even, the discriminant depends only on \(V(Q)\).
    These definitions extend to families of quadrics.
    
    We now state the conventions we use for conic bundles. Proofs of these results can be found in \cite{beauville-ij}*{Section 1}, \cite{prokhorov-rationality-conic-bundles}*{Section 3} and the references therein. (The statements in \cite{beauville-ij} are for the case of \(W=\PP^2\) but they extend to any smooth surface because the arguments are local. The results in \cite{prokhorov-rationality-conic-bundles} are stated in characteristic \(0\) but are valid in any characteristic different from \(2\).)
    
    Throughout the paper, we use \(X\) to denote a conic bundle over a smooth surface \(W\), i.e., a smooth projective variety with a flat morphism \(\piX\colon X \to W\) of relative dimension \(1\) such that \(\omega_X\) is relatively anti-ample. Then \(\mathcal E=\pi_*(\omega_X^{-1})\) is a locally free sheaf of rank \(3\) on \(W\). The line bundle \(\omega_X^{-1}\) defines an embedding \(X\hookrightarrow\PP(\mathcal E)\) whose image is defined by a section of \(\mathcal O_{\mathbb P(\mathcal E)}(2)\otimes p^*(\det\mathcal E^\vee\otimes\omega_W^{-1})\), where \(p\colon\PP(\mathcal E)\to W\) is the projection. Thus, locally over \({W}\), \(X\) is defined by a quadratic form. The generic fiber of \(\piX\) is a smooth conic, and the fiber of \(\piX\) over any \(\pt\in W\) is a conic in \(\mathbb P(\mathcal E)_{\pt}\) (possibly of rank less than \(3\)).
    
    There is a curve \(\Delta\stackrel{r}{\hookrightarrow} W\) with the property that the fiber over \(\pt\in W\) is a smooth conic if and only if \(\pt\not\in\Delta\). Let \(X_{\Delta}\) denote the preimage of \(\Delta\) in \(X\). The relative variety \(\mathcal F_1(X/W)\) of lines in the fibers of \(\piX\) factors through \(\Delta\); let \(\mathcal F_1(X/W)\to\Deltatilde\to\Delta\) be the Stein factorization. We call \(\Delta\) the \defi{discriminant locus} of \(\pi\) and \(\piDelta\colon \Deltatilde\to \Delta\) the \defi{discriminant cover} (or extension).  By definition, the curve \(\Deltatilde\) parametrizes the irreducible components of the fibers of \(\piX \colon X_{\Delta}\to \Delta\). When \(X_{\kk(\Delta)}\) is a rank \(2\) conic, then the field extension \(\kk(\Deltatilde)/\kk(\Delta)\) is the extension generated by a square root of the discriminant of the binary quadratic form defining \(X_{\kk(\Delta)}\), which justifies the terminology.
    
     A conic bundle is \defi{standard} if \(\Pic X = \pi^*\Pic W \oplus \Z\).  A conic bundle \(\pi\colon X \to W\) is \defi{ordinary} (following~\cite{BeltramettiFrancia83}) if the discriminant curve \(\Delta\) is smooth and irreducible; in this case the fibers of \(\piX\) have rank at least \(2\). Throughout this paper, we restrict to the case that \(X\) is geometrically ordinary and geometrically standard, i.e., that \(\Delta\) is smooth and geometrically irreducible and that \(\rho(X_\kbar/W_\kbar)=1\). These assumptions imply that \(\Deltatilde\) is smooth and geometrically irreducible, and that the map \(\piDelta\) is \'etale.  We call conic bundles \(X \to \PP^2\) with \(\Delta\) of degree \(d\) and satisfying these assumptions  \defi{degree \(d\) conic bundles}. 
    
    As mentioned in the introduction, a main focus of this paper is the IJT obstruction introduced by Hassett--Tschinkel and Benoist--Wittenberg.  For a geometrically rational threefold \(X\) over a field \(k\), Benoist and Wittenberg construct a \defi{codimension \(2\) Chow scheme} \(\bCH2_{X/k}\) whose connected component of the identity is called the intermediate Jacobian of \(X\).  The other (split) connected components of \(\bCH2_{X/k}\) are then naturally torsors under the intermediate Jacobian, justifying the terminology \defi{intermediate Jacobian torsors}. 
    \begin{defn}\label{defn:IJT-obstruction}
    Let \(X\) be a smooth, projective, geometrically rational threefold over a field \(k\). Assume that there exists a smooth, projective, geometrically connected curve \(\Gamma\) over \(k\) of genus \(\geq 2\) such that \((\bCH2_{X/k})^0\simeq\bPic_{\Gamma/k}^0\) as principally polarized abelian varieties. We say that the \defi{intermediate Jacobian torsor (IJT) obstruction  vanishes for \(X\)} if for every \(\gamma\in(\NS^2X_{\kbar})^{G_k}\) there exists an integer \(d\) such that \((\bCH2_{X/k})^\gamma\) and \(\bPic_{\Gamma/k}^d\) are isomorphic as  \(\bPic_{\Gamma/k}^0\)-torsors.
    \end{defn}
    \noindent By~\cite{bw-ij}*{Thm. 3.11}, if \(X\) is \(k\)-rational then the IJT obstruction must vanish.

\section{Benoist and Wittenberg's codimension \texorpdfstring{\(2\)}{2} Chow scheme}\label{sec:BWreview}

Grothendieck's definition of the Picard scheme of a \(k\)-variety \(X\) as the sheafification of the absolute Picard functor \(\Pic_{X/k}\colon T\mapsto\Pic(X_T)\) gives a group scheme whose \(\kbar\)-points agree with the usual Picard group \(\Pic(X_{\kbar})\).  When \(X\) is smooth (and so Weil divisors and Cartier divisors agree) this plays the role of a codimension \(1\) Chow scheme via the usual isomorphism \(\Pic(X_{\kbar}) \simeq \CH^1(X_\kbar) \) associating to a line bundle the vanishing divisor of a rational section. Extending this definition to codimension \(2\) cycles is extremely difficult since the naive formula ``\(T \mapsto \CH^2(X_T)\)'' is badly behaved.  Benoist--Wittenberg \cite{bw-ij} recently overcame this obstacle when \(X\) is geometrically rational by realizing a codimension \(2\) Chow functor as a quotient of \(K\)-theory.

In this section we briefly recall the notation and results from \cite{bw-ij} and references therein.  We then develop some of the functorial properties of the codimension \(2\) Chow scheme, especially regarding maps to/from the Picard scheme of a curve, that will be useful in the following sections. As mentioned in the introduction, a reader willing to accept these results when they are invoked in Lemma~\ref{lem:ejp-pje} can safely proceed to Section~\ref{sec:PPrymScheme} on a first reading.

For any quasicompact and quasiseparated scheme \(X\), let \(K_0(X)\) be the Grothendieck group of the triangulated category of perfect complexes of \(\OO_X\)-modules on \(X\).  When \(X\) admits an ample line bundle, \(K_0(X)\) is naturally isomorphic to the Grothendieck group of the exact category of vector bundles on \(X\).  If \(X\) is Noetherian and regular, \(K_0(X)\) is naturally isomorphic to the Grothendieck group of the abelian category of coherent \(\OO_X\)-modules.  
In this case, \(K_0(X)\) carries a natural filtration \(F^\bullet K_0(X)\) by codimension of support, and we write \(\Gr^\bullet K_0(X)\) for the associated graded. When it makes sense, we write \([\calE]\) for the class of a vector bundle/coherent sheaf \(\calE\) in \(\K_0(X)\).

Let \(X\) be a smooth, proper, geometrically connected threefold over \(k\).  Write \(\K_{0, X/k}\) for the absolute \(\K_0\) functor sending a \(k\)-scheme \(T\) to \(\K_0(X_T)\).
Given a morphism \(f \colon X \to Y\) of proper varieties, (derived) pullback \((f_T)^* \colon K_0(Y_T) \to K_0(X_T)\) induces a natural transformation of functors \(f^* \colon K_{0, Y/k} \to K_{0, X/k}\). If, moreover, \(f\) is perfect, then (derived) pushforward \((f_T)_* \colon K_0(X_T) \to K_0(Y_T)\) induces a natural transformation of functors \(f_* \colon K_{0, X/k} \to K_{0, Y/k}\).

Let \(\SK_{0, X/k, \fppf}\) be the fppf-sheafification of the kernel of \((\rank,\det)\colon K_{0, X/k} \to \Z_{X/k} \times \Pic_{X/k} \) \cite{bw-ij}*{Definition 2.2.1}.  When the degree map \(\CH_0(X_\Omega)\to\mathbb Z\) is an isomorphism for any algebraically closed field extension \(k\subset\Omega\) (in particular, when \(X\) is geometrically rational), there exists a unique class \(\nu_X \in \SK_{0, X/k, \fppf}(k)\) of a point on \(X\) \cite{bw-ij}*{Proposition 2.8}, with the property that for any finite extension \(k'/k\) and any coherent sheaf \(\calF\) on \(X_{k'}\) supported in dimension \(0\), 
\begin{equation}\label{eq:nu}
[\calF]= h^0(X_{k'}, \calF) \cdot \nu_X \in \SK_{0, X/k, \fppf}(k').
\end{equation}
This property implies that for any finite extension \(k'/k\) such that \(X(k')\neq\emptyset\), \([\OO_{x'}] = \nu_X\) for any \(x'\in X(k')\).  Since \(X_{k'} \to X\) is an fppf cover, this property characterizes \(\nu_X\).

Benoist and Wittenberg define \(\CH^2_{X/k,\fppf}\) to be the (presheaf) cokernel of \(\nu_{X}\colon\mathbb Z_{X/k}\to\SK_{0,X/k,\fppf}\).  When \(X\) is projective and geometrically rational, this functor is represented by a smooth group scheme over \(k\), denoted \(\bCH2_{X/k}\) \cite{bw-ij}*{Definition 2.9, Theorem 3.1(i)}.  Associating to a codimension \(2\) integral closed subvariety \([Z] \in \CH^2\Xbar\) the class \([\OO_Z] \in K_0(\Xbar)\) of its structure sheaf gives a \(G_k\)-equivariant isomorphism \(\CH^2\Xbar \simeq \bCH2_{X/k}(\kbar)\) \cite{bw-ij}*{Proposition 2.11}.

The component group \(\bCH2_{X/k}/(\bCH2_{X/k})^0\) is identified with the \(G_k\)-module \(\NS^2\Xbar\) \cite{bw-ij}*{Theorem 3.1(v)}, and we write \((\bCH2_{X/k})^\gamma\) to denote the preimage of \(\gamma \in (\NS^2\Xbar)^{G_k}\), which is a \((\bCH2_{X/k})^0\)-torsor defined over \(k\). Since the quotient map is a group homomorphism, for any \(G_k\)-invariant curve classes \(\gamma,\gamma'\), we have
\begin{equation}\label{eq:torsors_add}[(\bCH2_{X/k})^\gamma]+[(\bCH2_{X/k})^{\gamma'}]=[(\bCH2_{X/k})^{\gamma+\gamma'}]
\end{equation}
in the Weil--Ch\^atelet group of \((\bCH2_{X/k})^0\).

We next describe some situations in which a morphism of varieties induces a morphism of \(\CH^2\) functors (and hence, in the geometrically rational case, a morphism of codimension 2 Chow schemes). We will apply these morphisms to conic bundle threefolds in Section~\ref{sec:scheme_map}. All the morphisms we discuss will be between smooth proper varieties, so they will be proper and perfect \cite{stacks-project}*{\href{https://stacks.math.columbia.edu/tag/01W6}{Tag 01W6} and \href{https://stacks.math.columbia.edu/tag/068B}{Tag 068B}}.

\begin{prop}\label{prop:epsilon}
    Let \(X, X'\) be smooth proper geometrically connected threefolds over \(k\) such that the degree maps \(\CH_0(X_\Omega)\to\mathbb Z\), \(\CH_0(X'_\Omega)\to\mathbb Z\) are isomorphisms for any algebraically closed field extension \(k\subset\Omega\). Let \(\eps\colon X'\to X\) be a birational morphism. Then \(\eps\) induces natural transformations
    \[
    \eps^*\colon\CH^2_{X/k,\fppf}\to\CH^2_{X'/k,\fppf} \quad\textup{and} \quad \eps_*\colon\CH^2_{X'/k,\fppf}\to\CH^2_{X/k,\fppf}
    \]
    with the following properties:
    \begin{enumerate}\item\label{item:eps-composition-identity} The composition is the identity on \(\CH^2_{X/k,\fppf}\), and
    \item\label{item:eps-kbar-points} The induced maps on \(\kbar\)-points are \(G_k\)-equivariant and agree with the pullback map \(\eps^*\colon\CH^2 X_\kbar \to\CH^2 X'_\kbar\) and pushforward map \(\eps_*\colon\CH^2 X'_\kbar\to\CH^2 X_\kbar\) on Chow groups.
    \end{enumerate}
    If \(X, X'\) are projective and geometrically rational, this gives morphisms of group schemes \(\eps^*\colon\bCH2_{X/k}\to\bCH2_{X'/k}\) and \(\eps_*\colon\bCH2_{X'/k}\to\bCH2_{X/k}\) respecting the principal polarizations.
\end{prop}

\begin{proof}
By \cite{bw-ij}*{Section 3.2.2}, \(\eps\) induces a natural transformation \(\eps^*\colon\SK_{0,X/k,\fppf}\to\SK_{0,X'/k}\) that maps \(\nu_X\) to \(\nu_{X'}\). So taking the quotient yields the natural transformation \(\eps^*\colon\CH^2_{X/k,\fppf}\to\CH^2_{X'/k,\fppf}\). Furthermore, by \cite{bw-ij}*{Proposition 2.3(i)}, \(\eps\) induces a natural transformation \(\eps_*\colon \SK_{0,X'/k} \to \SK_{0,X/k}\). To see that this gives a natural transformation \(\eps_*\colon\CH^2_{X'/k,\fppf}\to\CH^2_{X/k,\fppf}\), it suffices to check that \(\eps_*\nu_{X'}=\nu_X\). 
Let \(k'/k\) be a finite extension such that there exists a point \(x' \in X'(k')\).  By \eqref{eq:nu}, we have \([\OO_{x'}] = \nu_{X'}\).  On the other hand, \(\eps_*[\OO_{x'}] = [\OO_{\eps(x')}] \in K_{0}(X_{k'})\) since the map \(\eps(x') \colon \Spec(k') \to X_{k'}\) is affine, and so has no higher pushforwards.
Again by \eqref{eq:nu}, we conclude that \(\nu_X = [\OO_{\eps(x')}] = \eps_*\nu_{X'}\). This shows the existence of \(\eps_*\). 

Property~\eqref{item:eps-composition-identity} follows from \cite{bw-ij}*{Proposition 2.3} and property~\eqref{item:eps-kbar-points} holds for \(\eps^*\) by \cite{bw-ij}*{Lemma 2.1}, and for \(\eps_*\) by \cite{Fulton-IsectBook}*{Example 15.1.5}. If \(X,X'\) are projective and geometrically rational, then the induced morphisms on group schemes follow from \cite{bw-ij}*{Theorem 3.1(i)} and the statement about principal polarizations follows from \cite{bw-ij}*{Proposition 3.5}.
\end{proof}

\begin{prop}\label{prop:pi}
    Let \(X\) be a smooth proper geometrically connected threefold such that the degree map \(\CH_0(X_\Omega)\to\mathbb Z\) is an isomorphism for any algebraically closed field extension \(k\subset\Omega\), and let \(\pi \colon X \to W\) be a morphism to a smooth proper geometrically connected surface.  There is a natural transformation
    \[\pi_* \colon \CH^2_{X/k, \fppf} \to \Pic_{W/k, \fppf}\]
    which on \(\kbar\)-points agrees with proper pushforward of cycles \(\CH^2(X_{\kbar}) \xrightarrow{\pi_*} \CH^1(W_{\kbar}) \simeq \Pic(W_{\kbar})\) after identifying Cartier and Weil divisors on the smooth variety \(W_\kbar\).  If \(X\) is projective and geometrically rational, this induces a morphism of group schemes 
    \[\pi_* \colon \bCH2_{X/k} \to \bPic_{W/k}.\]
\end{prop}
\begin{proof}
We have the following composition of natural transformations of functors:
    \[
    \SK_{0, X/k, \fppf} \xrightarrow{\pi_*} \K_{0, W/k, \fppf} \xrightarrow{\det} \Pic_{W/k, \fppf}.
    \]
    We claim that this composition induces a natural transformation \(\CH^2_{X/k, \fppf} \to \Pic_{W/k, \fppf}\), i.e., that \(\det \pi_*\nu_X = 0\).  Let \(k'/k\) be a finite extension such that there exists \(x \in X(k')\).  Then 
    \( \pi_*\nu_X = \pi_*[\OO_x] = [\OO_{\pi(x)}]\)
     has trivial determinant since \(W_{k'}\) is a smooth surface and \(\OO_{\pi(x)}\) is supported in codimension \(2\).  
     
     We now compute this map on \(\kbar\)-points.  
     Let \(Z \subset X_{\kbar}\) be an integral curve.  Let \(\pi_*Z\) denote the proper pushforward as a \(1\)-cycle.  It suffices to show that \(\det \pi_*\OO_Z \simeq \OO_W(\pi_*Z)\).  If the map \(\pi|_Z\) is constant (which implies that \(\pi_*Z = 0\) as 1-cycles), then \(\pi_*\OO_Z\) is supported in codimension \(2\) and so \(\det \pi_*\OO_Z \simeq \OO_W\) is trivial as well.  If \(\pi|_Z\) is nonconstant, then the map \(\pi|_Z \colon Z \to \pi(Z)\) is finite and \(\pi_*Z = \deg(\pi|_Z) \pi(Z)\).  On the other hand, \(\pi|_Z\) is affine, so higher pushforwards vanish and \(\pi_*\OO_Z\) is a vector bundle of rank \(\deg(\pi|_Z)\) supported on \(\pi(Z)\).  Since reflexive sheaves are determined away from codimension \(2\), we may apply~\cite{Fulton-IsectBook}*{Example 15.3.1} on \(W\) away from the singular locus of \(\pi(Z)\) to conclude that \(\det \pi_*\OO_Z \simeq \OO_W(\deg(\pi|_Z)\pi(Z))\).
\end{proof}

\subsection{Morphisms between the codimension \texorpdfstring{\(2\)}{2} Chow scheme of a threefold and the Picard scheme of a curve}  The intermediate Jacobian of a threefold is often related to the Jacobian of a curve.  In this section we give morphisms between the codimension \(2\) Chow scheme and Picard scheme of a curve in two scenarios that will be necessary in this paper.

Throughout this section, let \(X'\) be a smooth proper geometrically connected threefold\footnote{We will apply these results later to a blow-up of a conic bundle threefold, which is why we have elected to use $X'$ instead of $X$.} such that the degree map \(\deg\colon\CH_0(X'_\Omega)\to\mathbb Z\) is an isomorphism for any algebraically closed field extension \(k\subset\Omega\), \(S\) a smooth proper geometrically connected surface, and \(\Deltatilde\) a smooth proper geometrically connected curve. Let \(j\colon S\to X'\) be a finite morphism and \(p\colon S\to\Deltatilde\) a flat morphism.

\begin{prop}\label{prop:jp}
    Assume \(X'\) and \(S\) are projective.
     There is a natural transformation
    \[
        \Pic_{\Deltatilde/k, \fppf}\to\CH^2_{X'/k,\fppf},
    \]
    which on \(\kbar\)-points agrees with the composition of the \(G_k\)-equivariant homomorphisms 
    \[
        \Pic(\Deltatilde_{\kbar}) \simeq \CH^1(\Deltatilde_{\kbar}) \xrightarrow{p^*} \CH^1(S_{\kbar}) \xrightarrow{j_*} \CH^2(X'_{\kbar})
    \]
    on Chow groups. If \(X'\) is geometrically rational over \(k\), this induces a morphism of group schemes which we denote \(j_*p^*\colon\bPic_{\Deltatilde/k}\to\bCH2_{X'/k}\).
\end{prop}
\begin{proof}
Our eventual aim is to apply the natural transformations \(p^*\) and \(j_*\) on \(K\)-theory. 
 We must first construct a natural transformation \(\Pic_{\Deltatilde/k} \to \K_{0, \Deltatilde/k}\), which agrees with the map associating to an effective line bundle the \(K\)-theory class of the structure sheaf of the corresponding divisor.  Let \(T\) be any Noetherian local \(k\)-scheme.  Since \(T\) is affine and \(\Deltatilde\), \(S\) and \(X'\) are all projective, all \(\K_0\)-groups will be generated by classes of vector bundles.  By \cite{bw-ij}*{Proof of Lemma 2.5}, 
we have a functorial group homomorphism
\[ \Pic(\Deltatilde_T) \to \K_0(\Deltatilde_T)\]
sending a line bundle \(\calL\) to the class \([\calL] - [\calO_{\Deltatilde_T}] = [\calO_{\Deltatilde_T}] - [\calL^{-1}]\), and the image of this homomorphism is the kernel of \(\rank\colon\K_0(\Deltatilde_T)\to\mathbb Z\).
Using that \(\Pic\) and \(\K_0\) commute with directed inverse limits of quasicompact and quasiseparated schemes with affine transition maps and Noetherian approximation (see also the proof of \cite{bw-ij}*{Proposition 2.4}), this gives a natural transformation of functors
\(
\Pic_{\Deltatilde/k, \fppf} \to \K_{0, \Deltatilde/k, \fppf}.
\)
We compose this with the natural transformations \(p^*\) and \(j_*\) to obtain a natural transformation
\[ \Pic_{\Deltatilde/k, \fppf} \to \K_{0, \Deltatilde/k, \fppf} \xrightarrow{p^*} \K_{0, S/k, \fppf} \xrightarrow{j_*} \K_{0, X'/k, \fppf}.\]
We now prove that the image of this composition is contained in
\(\SK_{0, X'/k, \fppf}\), i.e., that post-composing with the map \((\rank, \det)\) is trivial. First, since we know that the image of \(\Pic_{\Deltatilde/k}\) in \(\K_{0,\Deltatilde/k,\fppf}\) is in the kernel of \(\rank\colon\K_{0,\Deltatilde/k,\fppf}\to\mathbb Z\), it follows that its image under \(j_*p^*\) is in the kernel of \(\rank\colon\K_{0, X'/k,\fppf}\to\mathbb Z\). To check that this image is also in the kernel of \(\det\colon\K_{0,X'/k,\fppf}\to\Pic_{X'/k,\fppf}\), it suffices to check on \(\kbar\)-points because \(\Pic_{\Deltatilde/k, \fppf}\) and \(\Pic_{X'/k, \fppf}\) are represented by smooth group schemes over \(k\).

Since \(\Pic(\Deltatilde_{\kbar})\) is generated by effective line bundles, we may reduce to
the case \(\calL = \OO_{\Deltatilde}(Z)\) for an effective Cartier divisor \(Z\).  Using the fundamental exact sequence 
\[0 \to \OO_{\Deltatilde}(-Z) \to \OO_{\Deltatilde} \to \OO_Z \to 0,\]
we see that the first map
\(\Pic(\Deltatilde_{\kbar}) \to K_0(\Deltatilde_{\kbar})\)
sends the line bundle \(\OO_{\Deltatilde}(Z)\) to the class of the coherent sheaf \([\OO_Z]\).  (In other words, this is the map \(\varphi^1 \colon \CH^1(\Deltatilde_{\kbar}) \to \Gr^1_F K_0(\Deltatilde_{\kbar})\) from \cite{bw-ij}*{Equation (2.4)} precomposed with the isomorphism \(\Pic(\Deltatilde_{\kbar}) \simeq \CH^1(\Deltatilde_{\kbar})\).)  Furthermore \(j_*p^*[\OO_Z] = j_*[\OO_{p^{-1}Z}]=[j_*\OO_{p^{-1}Z}] \in F^2K_0(X'_{\kbar}) = \SK_0(X'_{\kbar})\), since \(j\) is affine (and so higher direct images vanish).  

Finally, we quotient by the class \(\nu_{X'}\) to obtain the desired natural transformation 
\[\Pic_{\Deltatilde/k, \fppf}\to\CH^2_{X'/k,\fppf}.\] 
That this map on \(\bar{k}\)-points agrees with flat pullback and proper pushforward of cycles follows from the explicit description above (for pullback) together with \cite{Fulton-IsectBook}*{Example 15.1.5} (for pushforward).
\end{proof}

\begin{prop}\label{prop:pj}
    Assume \(j\colon S\to X'\) is a closed embedding. There is a natural transformation
    \[\CH^2_{X'/k,\fppf}\to \Pic_{\Deltatilde/k , \fppf},\]
    which on \(\kbar\)-points agrees with the composition of the \(G_k\)-equivariant homomorphisms 
\[\CH^2(X'_{\kbar}) \xrightarrow{j^*} \CH^2(S_{\kbar}) \xrightarrow{p_*} \CH^1(\Deltatilde_{\kbar}) \simeq \Pic(\Deltatilde_{\kbar})\]
on Chow groups. If \(X'\) is projective and geometrically rational over \(k\), this induces a morphism of group schemes which  we denote \(p_*j^*\colon\bCH2_{X'/k}\to \bPic_{\Deltatilde/k}\).
\end{prop}

\begin{proof}
    First, for every \(k\)-scheme \(T\) we get functorial homomorphisms
    \[\K_0(X'_T) \xrightarrow{j^*} \K_0(S_T) \xrightarrow{p_*} \K_0(\Deltatilde_T),\]
    which give natural transformations \(\K_{0,X'/k}\to \K_{0,S/k} \to \K_{0,\Deltatilde/k}\) by \cite{bw-ij}*{Section 2.2.2}. We post-compose these maps with \(\det \colon \K_{0,\Deltatilde/k} \to \Pic_{\Deltatilde/k}\). Restricting to \(\SK_{0,X'/k}\) gives a natural transformation
    \[\SK_{0,X'/k} \to \Pic_{\Deltatilde/k}.\]
    We take the induced map on the sheafifications, and we claim that \(\nu_{X'}\) is in the kernel. 
    Let \(k'/k\) be a finite extension such that there exists a point \(x' \in (X' - S)(k')\).  Then \([\OO_{x'}] = \nu_{X'} \in \SK_{0, X'/k, \fppf}(k')\) by \eqref{eq:nu}.  By \cite{Fulton-IsectBook}*{Example 15.1.8} (using that \(X_{k'}\) is regular and \(j_{k'}\) is a perfect closed embedding) the pullback is 
    \[j^*[\OO_{x'}] = \sum (-1)^i \Tor^{\OO_{X'}}_i(j_*\OO_S, \OO_{x'}).\]
    Using the explicit resolution \(\OO_{X'}(-S) \to \OO_{X'}\) of \(j_*\OO_S\) and the fact that \(x' \notin S\), we see that \(\Tor^{\OO_{X'}}_0(j_*\OO_S, \OO_{x'}) = \Tor^{\OO_{X'}}_1(j_*\OO_S, \OO_{x'}) = 0\), and so \(j_*\nu_{X'} = 0\).
    
    Now we compute this map on \(\kbar\)-points. Let \(Z \subset X'_{\kbar}\) be an integral curve.  If \(Z\) is not contained in \(S_{\kbar}\subset X'_{\kbar}\), then using the explicit resolution of \(j_*\OO_S\) to compute \(\Tor^{\OO_{X'}}_i(j_*\OO_S, \OO_Z)\), we find that \(j^*[\OO_Z]=\Tor^{\OO_{X'}}_0(j_*\OO_S, \OO_Z)=\OO_{S\cap Z}\).  As in the proof of Proposition~\ref{prop:pi}, we have \(\det p_*\OO_{S\cap Z}\simeq \OO_{\Deltatilde}(p_*(S\cap Z))\).  
    
    If \(Z\) is contained in \(S_{\kbar}\subset X'_{\kbar}\), then all maps are zero in the explicit resolution of \(j_*\OO_S\) restricted to \(Z\), so
    \(j^*[\OO_Z] = [\OO_Z] - [\OO_Z(-{S})]\).  This is the pushforward to \(\K_0(S_{\kbar})\) of the unique class in \(\K_0(Z)\) of rank \(0\) and determinant \(\OO_Z(S)\).  By the excess intersection formula~\cite{Fulton-IsectBook}*{Corollary 6.3}, we have \(j^*[Z] =c_1(\OO_S(S)) \cap [Z] \in \CH^2(S_{\kbar})\).  This is the pushforward to \(\CH^2(S_{\bar{k}})\) of the class in \(\CH^1(Z)\) corresponding to \(\OO_Z(S)\) under the isomorphism \(\CH^1(Z) \simeq \Pic(Z)\).  The result now follows from~\cite{Fulton-IsectBook}*{Example 15.1.5}.
\end{proof}

\section{The Prym variety and polarized Prym schemes}\label{sec:PPrymScheme}

Let \(\piDelta\colon\Deltatilde\to\Delta\) be an \'etale double cover of smooth proper geometrically connected curves over \(k\), and let \(\iota\) denote the associated (geometrically) fixed-point-free involution. Let \(\lattice\) be a $k$-group scheme with trivial identity component and let \(\phi\colon \lattice\to{\bPic_{\Delta/k}}\) be a $k$-morphism of group schemes. We will be most interested in the case when there is an embedding \(r\colon \Delta\hookrightarrow W\) into a geometrically rational surface, and the morphism \(\phi\) is \(r^*\colon\bPic_{W/k}\to\bPic_{\Delta/k}\).

We define the \defi{\(\phi\)-polarized Prym scheme} of \(\Deltatilde \to \Delta\) to be the fiber product
\[
\PPrym_{\Deltatilde/\Delta}^{\phi} \colonequals \bPic_{\Deltatilde/k}\times_{\bPic_{\Delta/k}} \lattice  \subset \bPic_{\Deltatilde/k} \times \lattice.
\]

Note that this group scheme is defined over \(k\), and its identity component is the usual Prym variety \(\bPrym_{\Deltatilde/\Delta}\), which is a principally polarized abelian variety \cite{mumford-prym}*{Sect. 2, Cor. 2}. Further, each split connected component is a torsor under \(\bPrym_{\Deltatilde/\Delta}\). In the case of an embedding \(r\colon \Delta\hookrightarrow W\) into a geometrically rational surface inducing \(r^*\colon\bPic_{W/k}\to\bPic_{\Delta/k}\), we denote the \((\bPic_{W/k})\)-polarized Prym scheme by \(\PPrym_{\Deltatilde/\Delta}^{\bPic_{W/k}}\). When \(r^*\) is injective, for example when \(W \simeq \PP^2\), we have
\[\PPrym_{\Deltatilde/\Delta}^{\bPic_{W/k}}(k) \simeq \left\{[\calO_{\Deltatilde}(D)] \in \bPic_{\Deltatilde/k}(k) : \piDelta_*D \in \im \left(\Pic\Wbar\right) \right\}.\]

For any (closed) point \(D \in \lattice\), write \(V_D\) for the fiber of \(\PPrym_{\Deltatilde/\Delta}^{\phi}\) over \(D\) via the second projection. This subscheme is naturally identified with its image in \(\bPic^{\deg \phi(D)}_{\Deltatilde_{\kbar}/\kbar}\) via the first projection.
Since \(\piDelta_*\piDelta^*D = 2D\), we have
\[
    V_{2D + D'} = \piDelta^*D + V_{D'} \quad \textup{for all }D,D'\in \lattice.
\]
Thus, it suffices to study \(V_D\) for a set of representatives \(D\) for the elements of \(\lattice/2\lattice\).  

The subscheme \(V_{\id_L}\) has two connected components, each of which is geometrically irreducible (see \cite{mumford-prym}*{Section~3}). We write
\begin{equation}\label{eq:DefnOfPAndPtilde}
    P := \bPrym_{\Deltatilde/\Delta}\quad\textup{and}\quad \Ptilde:= V_{\id_L} - P
\end{equation}
for the two components.  Since \(V_{{\id_L}} = P \sqcup \Ptilde\) is a group scheme, \(\Ptilde\) is a \(2\)-torsion \(P\)-torsor.  Furthermore, for any \(D\in \lattice(\kbar)\), \((V_D)_{\kbar}\) can be translated to \((V_{{\id_L}})_{\kbar}\) , so the scheme  \((V_D)_{\kbar}\) also has two connected components.

\subsection{Distinguishing connected components of \texorpdfstring{\(V_D\)}{VD}}
\begin{lemma}[{\cite{mumford-prym}}]\label{lem:P-Ptilde-characterization} Let \(n\geq 0\) be an integer. Then
\[
P=(\iota_* - 1)(\mathbf{Pic}^{2n}_{\Deltatilde/k})
    \simeq \bPic^0_{\Deltatilde/k}/\piDelta^*\bPic^0_{\Delta/k}
    \quad\textup{and}\quad
    \Ptilde=(\iota_* - 1)(\bPic^{2n+1}_{\Deltatilde/k}).
\]
\end{lemma}
\begin{proof}
    Note that \((\iota_* - 1)(D + \piDelta^*D_1) = (\iota_* - 1)(D)\) for any divisor classes \(D_1\in \bPic^1_{\Delta/k}\) and  \(D\in \bPic^m_{\Deltatilde/k}\).  Therefore, \((\iota_* - 1)(\mathbf{Pic}^{m+2}_{\Deltatilde/k}) = (\iota_* - 1)(\mathbf{Pic}^{m}_{\Deltatilde/k})\) for any integer \(m\), and so we may assume that \(n=0\). By~\cite{mumford-prym}*{Section~3}, \(P\) is the image of \(\bPic^0_{\Deltatilde/k}\) under \(\iota_* - 1\).  Thus, by the first isomorphism theorem, \(P=(\iota_* - 1)(\bPic^0_{\Deltatilde/k})\simeq \bPic^0_{\Deltatilde/k}/\ker (\iota_* - 1)\).
    In addition, we also have that \(\piDelta^*\bPic^0_{\Delta/k}\subset \ker\left((\iota_* - 1)|_{\bPic^0_{\Deltatilde/k}}\right)\), and that this inclusion is surjective on
    \(\kbar\)-points.  Since \(\iota_* - 1\) is smooth by ~\cite{mumford-prym}*{Section~2}, this implies that \(\piDelta^*\bPic^0_{\Delta/k}= \ker\left((\iota_* - 1)|_{\bPic^0_{\Deltatilde/k}}\right)\). 
    Finally, \((\iota_* - 1)(\bPic^1_{\Deltatilde/k})=\Ptilde\) because they are both subvarieties of \(\bPic^0_{\Deltatilde/k}\) with the same \(\kbar\)-points and \(\iota_* - 1\) is smooth.
\end{proof}

We can consider points in \(\Sym^d(\Deltatilde)(k)\) as giving \(k\)-rational effective degree \(d\) divisors on \(\Deltatilde\).  From this viewpoint, we define the intersection \(s\cap \tilde{s}\) of two points \(s, \tilde{s}\in\Sym^d(\Deltatilde)(k)\) to be the effective divisor \(E_{s,\tilde{s}}\) of largest degree such that \(E_{s,\tilde{s}}\leq s, \tilde{s}\).
\begin{cor}\label{cor:SandStildeDifferByOdd} 
    Let \(D\in \lattice(k)\) be such that \(V_D\) is split. Let  \(s,\tilde{s} \in \Sym^{\deg \phi(D)}(\Deltatilde)(k)\) be such that
    \[
    [s], [\tilde{s}] \in V_D\;\textup{and}\;
    \piDelta_*(s) = \piDelta_*(\tilde{s})\in \Sym^{\deg \phi(D)}(\Delta).
    \]
    Then \([s]\) and \([\tilde{s}]\) lie in the same connected component of \(V_D\) if and only if \(\deg(s\cap \tilde{s})\equiv\deg \phi(D)\bmod 2\).
\end{cor}
\begin{proof}
Let \(e := {\deg}(s\cap \tilde{s})\). Then since \(\piDelta_*(s) = \piDelta_*(\tilde{s})\), we may write \(s = \sum_{i=1}^n P_i\) and \(\tilde{s} = \sum_{i=1}^{n_0}P_i + \sum_{i=n_0 + 1}^n \iota(P_i)\) where \(\sum_{i = 1}^{n_0}\deg(P_i) = e\) and \(\sum_{i = n_0 + 1}^n \deg(P_i) = \sum_{i = n_0 + 1}^n \deg(\iota(P_i))= \deg\phi(D) - e.\)

Applying Lemma~\ref{lem:P-Ptilde-characterization}, we conclude that the linear equivalence class of \(\tilde{s} - {s} = \sum_{i=n_0+1}^{n} (\iota_*P_i- P_i)\) is in \(P\) if \(\sum_{i = n_0 + 1}^n \deg(P_i) = \deg \phi(D)-e\) is even and in \(\Ptilde\) if \(\deg \phi(D)-e\) is odd.  Since \(P\) is the connected component containing the identity, \([s]\) and \([\tilde{s}]\) are in the same connected component if and only if \(\deg \phi(D)-e\) is even, as desired.
\end{proof}

\subsection{The canonically polarized Prym scheme}
\label{sec:CanonicalPrymScheme}

The \defi{canonically polarized Prym scheme} of \(\Deltatilde \to \Delta\) is  \(\PPrym_{\Deltatilde/\Delta}^{\phi}\) where \(\phi\) is the embedding of the subgroup generated by the canonical divisor \(K_\Delta\).  We will denote this by \(\PPrym_{\Deltatilde/\Delta}^{\langle K _\Delta\rangle}\). For the canonically polarized Prym scheme,  Prym--Brill--Noether theory gives the following result.

\begin{prop}[{\cite{mumford-prym}*{Section 6, Equation (6.1)}}]\label{prop:P1-parity}
The parity of the dimension \(\hh^0\) of global sections differs on the two geometric connected components of \(V_{K_{\Delta}}\). 
\hfill \qed
\end{prop}
As a consequence, \(V_{K_{\Delta}}\) has two connected components (i.e., the two geometric connected components are defined over \(k\)).

\begin{defn}\label{def:Ptildem}
Let \(\Pm{1}\) denote the connected component of \(V_{K_{\Delta}}\) on which \(\hh^0\) is even, and let \(\Ptildem{1}\) denote the connected component on which \(\hh^0\) is odd.  For each positive integer \(m\) and \(e\in \{0,1\}\), define 
\[
    \Pm{2m + e} := \Pm{e} + mK_{\Deltatilde}, \quad \Ptildem{2m + e} := \Ptildem{e} + mK_{\Deltatilde},
\]
where \(\Pm{0} := P\) and \(\Ptildem{0} := \Ptilde\). (Note that as \(P\)-torsors, \(\Ptildem{m} = \Ptilde + \Pm{m}\) and \(2 \Pm{m} = 2\Ptildem{m} = \Pm{2m}\).)
Furthermore, for each positive integer \(m\), we define the subschemes \(\Sm{m}\) and \(\Stildem{m}\) of \(\Sym^{m\cdot\deg K_{\Delta}}\Deltatilde\) to be the preimages of \(\Pm{m}\) and \(\Ptildem{m}\), respectively, under the Abel--Jacobi map \(\Sym^{m\cdot\deg K_{\Delta}}\Deltatilde\to\bPic^{m\cdot\deg K_{\Delta}}_{\Deltatilde/k}\).
\end{defn}

\subsection{\'Etale double covers of smooth plane quartics}\label{sec:UnramDoubleCoverDeg4}\label{sec:PrymPlaneQuartics}

In this section, we specialize to the case where \(\Delta\) is a smooth plane quartic, which implies that the cover \(\Deltatilde\) is a nonhyperelliptic, nontrigonal genus \(5\) curve \cite{bruin}*{Lemma 3.1}.  Bruin~\cite{bruin} has studied such covers in detail.  In Theorem~\ref{thm:EqnsForCurves} below, we summarize some of Bruin's geometric results (parts~\eqref{part:Qis},~\eqref{part:Steincommutes}, and~\eqref{part:isoJacGammaToPrym}) and apply them to describe the Abel--Jacobi maps
\(\Stildem{1}\to \Ptildem{1}\) and \(\Sm{1} \to \Pm{1}\) (parts~\eqref{part:isoF2toS1} and~\eqref{part:Ptilde1}). (Bruin's results are stated in characteristic \(0\) but hold for all fields of characteristic different from \(2\); see Remark~\ref{rem:Bruin-characteristic}.)

    \begin{theorem}\label{thm:EqnsForCurves}
        Let \(k\) be a field of characteristic different from \(2\) and let \(\piDelta\colon \Deltatilde\to \Delta\) be an \'etale double cover of a smooth plane quartic curve.  
        \begin{enumerate}
            \item \cite{bruin}*{Section 3} There exist quadratic forms \(Q_1,Q_2,Q_3\in k[u,v,w]\) such that 
            \[
                \Delta = V(Q_2^2 - Q_1Q_3)\subset \PP^2\quad\textup{and}\quad
                \Deltatilde = V(Q_1- r^2, Q_2 - rs, Q_3 - s^2)\subset \PP^4.
            \]
            These quadratic forms are unique up to a \(\PGL_2\)-action, where \(\left(\begin{smallmatrix}a & b \\c & d\end{smallmatrix}\right)\) acts by
            \[
                \begin{pmatrix}
                a^2 &   2ac   & c^2 \\
                 ab & ad + bc & cd\\
                b^2 &   2bd   & d^2
                \end{pmatrix}\cdot \begin{pmatrix}Q_1\\Q_2\\Q_3\end{pmatrix}.
            \]
            \label{part:Qis}
            \item Given these quadratic forms \(Q_1,Q_2,Q_3\), consider the family \(\calQ\to \PP^1\) whose fiber above \([t_0:t_1]\) is given by \(V(t_0^2Q_1 + 2t_0t_1Q_2 + t_1^2Q_3 - (t_0r + t_1s)^2)\subset\PP^4\), which is a quadric of rank at most \(4\). The Stein factorization of the relative Fano variety of projective \(2\)-planes \(\calF_2(\calQ/\PP^1)\to \PP^1\) produces a genus \(2\) curve \(\Gamma\to \PP^1\) given by the equation
            \[
                y^2 = -\det\left(t^2M_1 + 2tM_2 + M_3\right),
            \]
            where \(M_i\) is the symmetric \(3\times3\) matrix associated to \(Q_i\).\label{part:Stein} 
            \item The morphism \(\phi\colon\calF_2(\calQ/\PP^1)\to \Sym^4\Deltatilde\) given by \(\Lambda\mapsto \Lambda\cap \Deltatilde\) 
            induces an isomorphism \(\calF_2(\calQ/\PP^1)\xrightarrow{\sim}\Sm{1}\).
            \label{part:isoF2toS1}
            \item {\cite{bruin}*{Lemma 4.1}} The morphism \(\phi\) induces a map \(\overline{\phi}\) such that the diagram 
            \[
                \xymatrix{
                    \calF_2(\calQ/\PP^1) \ar[r]^-{\phi}\ar[d] & \Sm{1} \ar[d]\\
                    \Gamma \ar[r]^-{\phibar} & \Pm{1}
                }
            \]
            commutes, where the left vertical map comes from~\eqref{part:Stein}.\label{part:Steincommutes}
            \item {\cite{bruin}*{Sec. 5, Case 4}} The morphism \(\overline{\phi}\) induces isomorphisms \(\bPic^0_{\Gamma/k} \xrightarrow{\sim} \bPrym_{\Deltatilde/\Delta}\) and \(\bPic^1_{\Gamma/k} \xrightarrow{\sim} \Pm{1}\)\label{part:isoJacGammaToPrym}\label{part:isoPic1ToP1}.
            \item The Abel--Jacobi map \(\Stildem{1}\to \Ptildem{1}\) is an isomorphism.\label{part:Ptilde1}\label{part:AJiso}
        \end{enumerate}
    \end{theorem}
    \begin{defn}
    \label{def:prym-curve}
         The \defi{Prym curve} of an \'etale double cover \(\Deltatilde\to \Delta\) of a smooth plane quartic curve is the genus \(2\) curve in part~\eqref{part:Stein} of the above theorem, and we denote it \(\Gamma_{\Deltatilde/\Delta}\).
    \end{defn}
    \begin{proof}
        The first claim in part~\eqref{part:Qis} follows from~\cite{bruin}*{Section 3}.  It remains to prove uniqueness up to the given \(\PGL_2\)-action.  From the equations, one can verify that \(\kk(\Deltatilde) = \kk(\Delta)(\sqrt{Q_1/u^2})\). We claim that if \(Q\) is any other quadric such that \(Q/Q_1\in\kk(\Delta)^{\times2}\), then \(Q = t_0^2Q_1 + 2t_0t_1Q_2 + t_1^2Q_3\), for some choice of \(t_0,t_1\).  In particular, these functions are parametrized by a \(\PP^1\), and the choice of \(Q_1,Q_3\), and \(Q_1 + 2Q_2 + Q_3\) corresponds to choosing \(3\) points on the \(\PP^1\).  Since any choice of \(3\) points differs from these by an element of \(\PGL_2\) (acting by automorphisms of \(\PP^1\)), this gives the desired statement.  To prove the claim, observe that \(Q/Q_1\in\kk(\Delta)^{\times2}\) implies that \(\frac12(V(Q) - V(Q_1)) = 0 \in \Pic^0\Delta\).  By the Riemann--Roch theorem, \(|\frac12V(Q_1)|\isom \PP^1\), and its image in \(|2K_{\Delta}|=|\OO_\Delta(2)|\) under the multiplication by two map gives a quadratic curve that contains the vanishing of \(Q_1,Q_3, Q_1 + 2Q_2 + Q_3\).  Thus, \(V(Q)\) has the claimed form, and since \(t_0^2Q_1 + 2t_0t_1Q_2 + t_1^2Q_3\) has the same square class as \(Q_1\), we also have the correct defining quadratic equation.
        
        Part~\eqref{part:Stein} follows from structural results about quadrics; see, e.g.,~\cite{EKM}*{Section 85}. 
        
        Now we consider~\eqref{part:isoF2toS1}.  If \(\Lambda\in \calF_2(\calQ/\PP^1)\), then \(\Lambda\) is completely contained in some quadric in the \(k\)-span of \(Q_1 - r^2, Q_2 - rs, Q_3 - s^2\), and thus the intersection \(\Lambda\cap \Deltatilde\) is given by the vanishing of two quadrics in the plane \(\Lambda\).  Since \(\Deltatilde\) is geometrically irreducible and spans \(\PP^4\), this intersection must be \(0\)-dimensional, and so \(\Deltatilde\cap \Lambda\) is a degree \(4\) \(0\)-dimensional scheme that spans \(\Lambda\).  In particular, \(\Lambda\mapsto\Lambda\cap \Deltatilde\) defines an injective morphism \(\phi\colon \calF_2(\calQ/\PP^1)\to \Sym^4\Deltatilde\). 
        
        In order to prove that the image of \(\phi\) is contained in \(\Sm{1}\) we must show that the divisor class of \(\phi(\Lambda) = \Lambda\cap \Deltatilde\) lands in \(\Pm{1}\), i.e., that 
        \[
            \piDelta_*(\phi(\Lambda)) \text{{\(\in |K_{\Delta}|\)}} \quad\textup{ and }\quad \hh^0([\phi(\Lambda)]) \equiv 0 \bmod 2.
        \]
        Since \(\Lambda\) is a \(2\)-plane contained in a \(3\)-dimensional rank \(3\) or \(4\) quadric \(\mathcal Q_{[t_0:t_1]}\), \(\Lambda\) must contain the singular locus of this quadric, in particular contains the point \(V(u, v, w, t_0r+t_1s)\).
        (Note that each quadric threefold in the pencil has rank \(3\) or \(4\), and the fiber over a point \([t_0:t_1]\) has rank \(3\) if and only if \(\det(t_0^2 M_1 + 2 t_0 t_1 M_2 + t_1^2 M_3)=0\).  In the rank \(4\) case, the point \(V(u, v, w, t_0r+t_1s)\) is the singular locus, while in the rank \(3\) case, the singular locus is a line containing this point.)
        
        Therefore, \(\Lambda\) must be contained in a hyperplane of the form \(au+bv+cw=0\), and so \(\piDelta_*(\phi(\Lambda))\) must be contained in the line \(au + bv + cw=0\).  Hence, recalling that \(\Delta\) is canonically embedded, we have  \(\piDelta_*(\phi(\Lambda))\) \(\in |K_{\Delta}|\).  Now let us consider the dimension \(\hh^0([\phi(\Lambda)])\).  If \(\Lambda, \Lambda'\) are 2-planes in the same ruling, then \([\phi(\Lambda)] = [\phi(\Lambda')]\); hence \(\hh^0([\phi(\Lambda)]) \geq 2\).  By Clifford's theorem on special divisors, \(\hh^0([\phi(\Lambda)]) \leq 2\).
        Thus, we have the desired parity, and so \(\im\phi\subset \Sm{1}\).  
        
        To prove that \(\phi\) surjects onto \(\Sm{1}\) we appeal to~\cite{bruin}*{Lemma 4.1(1) and 4.1(2)}.  This result of Bruin implies that any \(\calD\in \Sm{1}\) is \(\Lambda\cap \Deltatilde\), for some \(2\)-plane \(\Lambda\) contained in a rank \(3\) or \(4\) quadric contained in the net of quadrics spanned by \(\{Q_1 - r^2, Q_2 - rs, Q_3 - s^2\}\).  From the description in \cite{bruin}*{Section 4} of the singular locus of the net, all the rank \(3\) quadrics in this net are in \(\calQ\) (in Bruin's notation, \(\calQ\) is parametrized by \(\Gamma^- \simeq\PP^1\)). So we may assume the rank is \(4\). Then, since, by definition, \(\piDelta_*\Sm{1}\subset |K_{\Delta}|\subset \Sym^4\Delta\), the rank \(4\) quadric must have its singular point on the line \(V(u,v,w)\).  Thus, the rank \(4\) quadric must be a member of \(\calQ\).

        Finally, we will show that \(\Sm{1}\) is smooth to conclude that \(\phi\) is an isomorphism.  Since \(\Sm{1} \to \Pm{1}\) is a Severi--Brauer fibration over its image, it suffices to show that the image is smooth.  The image is the Prym--Brill--Noether locus with \(h^0 = 2\), and so it is smooth of the expected dimension at a point \(L \in \Pic \Deltatilde\) if the Prym--Petri map is injective \cite{welters}*{(1.9)}.  To recall the definition of the Prym--Petri map \cite{welters}*{(1.8)}, first consider the map
        \[H^0(L) \otimes H^0(L) \to H^0(L) \otimes H^0(K_{\Deltatilde} \otimes L^{-1}) \xrightarrow{\alpha} H^0(K_{\Deltatilde}) \to H^0(K_{\Deltatilde})/\varpi^*H^0(K_{\Delta}),\]
        where the first map sends \(s \otimes t \mapsto s \otimes \iota^* t\), the map \(\alpha\) is the usual Petri map, and the final projection maps \(\lambda \mapsto (1/2)(\lambda - \iota^* \lambda)\).  Since this map is skew-symmetric, it factors as \(\bigwedge^2 H^0(L) \to H^0(K_{\Deltatilde})/\varpi^*H^0(K_{\Delta})\); this induced map is the \defi{Prym--Petri map}.  Given two independent sections \(s\) and \(t\) in \(H^0(L)\), the divisors \(V(s)\) and \(V(\iota^* t)\) correspond to two projective \(2\)-planes in some quadric in the family \(\calQ\) that meet along a line (which contains the singular locus of the quadric).  Their span is a hyperplane, and the Prym--Petri map is injective if this hyperplane is not pulled back from \(H^0(K_\Delta)\).  To prove that the span of \(V(s), V(\iota^*t)\) is not pulled back from \(H^0(K_\Delta)\), first observe that for any fixed quadric \(Q\) from \(\calQ\), projection from the line \(u=v=w=0\) yields a map \(\calF_2(Q) \to \check{\PP}^2\) that is \(2\)-to-\(1\) onto its image.  It follows that, given any \(2\)-plane \(\Lambda \subset Q\), the only other \(2\)-plane with the same image under projection from the line \(u=v=w=0\) is \(\iota\Lambda\).  If \(s\wedge t \neq 0\), then \(V(s)\neq V(t) =\iota V(\iota^* t)\), and the map is injective as desired.
        
        Part~{\eqref{part:Steincommutes}} follows from~\cite{bruin}*{Lemma 4.1}, and the first claim of part~\eqref{part:isoJacGammaToPrym} is proved in~\cite{bruin}*{Section 5, Case 4}.  The second claim of part~{\eqref{part:isoPic1ToP1}} follows from the first together with the universal property of \(\Gamma\to \bPic^1_{\Gamma/k}\).        Finally, for part~\eqref{part:Ptilde1}, recall that \(\Ptildem{1}\) is the connected component on which \(\hh^0\) is odd. Since, as explained above, \(\hh^0\) is at most 2, the only possibility is that \(\hh^0 = 1\) identically on \(\Ptildem{1}\).  Since the fibers of the map \(\Stildem{1}\to \Ptildem{1}\) are projective spaces of dimension \(\hh^0 - 1\), it is an isomorphism.
\end{proof}

\begin{remark}\label{rem:Bruin-characteristic}
    The results in \cite{bruin}*{Sections 3--4 and Section 5 Case 4} are stated in characteristic 0 but are valid in any characteristic different from \(2\). \cite{bruin}*{Lemma 3.1} does not require characteristic 0: a genus 5 curve always has at most one \(g^1_3\) and no \(g^2_3\)'s. Petri's theorem also holds in any characteristic, so the only part of the proof of \cite{bruin}*{Lemma 3.2} that uses characteristic \(0\) is in invoking Bertini's theorem to argue that a general quadric in \(\Lambda\) is smooth away from \(\Deltatilde\). Smoothness of a general member also holds in odd characteristic because \cite{bruin} shows that every quadric in \(\Lambda\) is smooth along \(\Deltatilde\), so the generic fiber of \(\Bl_{\Deltatilde}\PP^4\to\Lambda\) is a regular quadric threefold in \(\PP^4_{\kk(\Lambda)}\), and a regular quadric can fail to be smooth only in characteristic 2 \cite{EKM}*{Remark 7.20}.
\end{remark}

\section{Intermediate Jacobian torsors}\label{sec:IJ-torsors}

We now transition to the study of smooth geometrically ordinary and geometrically standard conic bundles \(\pi\colon X\to W\), for \(W\) a smooth geometrically rational surface.
The goal of this section is to relate the codimension \(2\) Chow group \(\CH^2\Xbar\) to the \(\kbar\)-points of the \((\bPic_{W/k})\)-polarized Prym scheme \(\PPrym^{\bPic_{W/k}}_{\Deltatilde/\Delta}\) as Galois modules.

When \(X\) is geometrically rational, we can consider the intermediate Jacobian torsors of \(X\), i.e., the split connected components of Benoist--Wittenberg's codimension 2 Chow scheme \(\bCH2_{X/k}\).  We prove that the connected components of  Benoist--Wittenberg's codimension 2 Chow scheme \(\bCH2_{X/k}\) 
 are each isomorphic to a connected component of  \(\PPrym^{\bPic_{W/k}}_{\Deltatilde/\Delta}\).  
The component group \(\bCH2_{X/k}/(\bCH2_{X/k})^0\) is isomorphic (as a Galois module) to the group of curve classes \(\NS^2 \Xbar\)~\cite{bw-ij}*{Theorem~3.1}, and so each connected component is identified with an algebraic curve class.
Our precise result is as follows.
\begin{theorem}
\label{thm:IJTtorsorsOfConicBundles}
    Let \(k\) be a field of characteristic different from 2, and let \(\piX \colon X\to W\) be a conic bundle over a smooth geometrically rational surface, all defined over \(k\), with smooth and geometrically irreducible discriminant cover \(\piDelta\colon \Deltatilde\to \Delta\).
    \begin{enumerate}
    \item There is a Galois-equivariant surjective group homomorphism 
    \[ \CH^2\Xbar \to \PPrym^{\bPic_{W/k}}_{\Deltatilde/\Delta}(\kbar), \]
    yielding, for any \(\gamma \in (\NS^2X_{\kbar})^{G_k}\), a Galois-equivariant isomorphism between \((\CH^2X_{\kbar})^\gamma\) 
    and the \(\kbar\)-points of one of the connected components of \(\PPrym_{\Deltatilde/\Delta}^{\bPic_{W/k}}\),  that is compatible with the actions of the identity components.
    \label{part:IsoOnkbarPts}
    
    \item There is a unique Galois-invariant algebraic curve class \(\gammatilde_0\) given by a \(\kbar\)-line 
    contained in a 
    singular fiber of \(\pi\), and for this class the homomorphism from~\eqref{part:IsoOnkbarPts} induces the \(G_k\)-equivariant isomorphisms
    \[(\CH^2\Xbar)^{n \gammatilde_0} \simeq \begin{cases} P(\kbar) &: \text{ \(n\) even,} \\ \Ptilde(\kbar) &: \text{ \(n\) odd.} \end{cases}\]\label{part:gamma0tilde}
    
    \item If \(X\) is geometrically rational, the map in \eqref{part:IsoOnkbarPts} can be upgraded to a surjective morphism of group schemes over \(k\) 
    \[\bCH2_{X/k} \to \PPrym_{\Deltatilde/\Delta}^{\bPic_{W/k}}\]
    that is an isomorphism when restricted to each connected component.
    In particular, for any \(\gamma \in (\NS^2X_{\kbar})^{G_k}\), the torsor \((\bCH2_{X/k})^\gamma\) is isomorphic to one of the connected components of \(\PPrym_{\Deltatilde/\Delta}^{\bPic_{W/k}}\), and \((\bCH2_{X/k})^{n \gammatilde_0}\) is isomorphic to \(P\) for \(n\) even and to \(\Ptilde\) for \(n\) odd.\label{part:IsoAsTorsors}
    \end{enumerate}
\end{theorem}

\begin{remark}\label{rem:maps-are-pje}
The homomorphism in Theorem~\ref{thm:IJTtorsorsOfConicBundles}\eqref{part:IsoOnkbarPts} is given by \((\pje, \pi_*) \colon \CH^2 X_{\kbar} \to \Pic \Deltatilde_{\kbar} \times \Pic W_{\kbar}\) (defined in Section~\ref{sec:beauville}).  Our results from Section~\ref{sec:BWreview} will be used to show that this can be upgraded to the morphism of group schemes of Theorem~\ref{thm:IJTtorsorsOfConicBundles}\eqref{part:IsoAsTorsors}.
\end{remark}

    In Section~\ref{sec:NS2}, we give a description for the group of algebraic curve classes on conic bundle threefolds over geometrically rational surfaces. In Section~\ref{sec:beauville}, we recall some intersection-theoretic calculations of Beauville used in his proof that the intermediate Jacobian of a conic bundle over \(\PP^2\) over an algebraically closed field is the Prym variety of \(\piDelta \colon \Deltatilde \to \Delta\).  In Section~\ref{sec:scheme_map},
    we upgrade this identification to a morphism of group schemes \(\bPic_{\Deltatilde/k} \to \bCH2_{X/k}\) for \(X\) geometrically rational. We show that this morphism induces an isomorphism of group schemes from the identity component of the \((\bPic_{W/k})\)-polarized Prym scheme to \((\bCH2_{X/k})^0\), recovering Beauville's result for \(W=\PP^2\) after base change to \(\kbar\).
    In Section~\ref{sec:proof_IJT}, we explain how these results combine to prove Theorem~\ref{thm:IJTtorsorsOfConicBundles}.  Finally, in Section~\ref{subsec:Explicitpje}, we explicitly determine the images of certain curves classes under the map in Theorem~\ref{thm:IJTtorsorsOfConicBundles}\eqref{part:IsoOnkbarPts}; this will be used in our rationality constructions in Section~\ref{sec:IJTvanishingConsequences}.

\subsection{Algebraic curve classes on \texorpdfstring{\(X\)}{X}}\label{sec:NS2}

\begin{prop}\label{prop:NS2}
    Let \(k\) be a field of characteristic different from \(2\). Let \(\piX\colon X\to W\) be a conic bundle threefold over a smooth geometrically rational surface with smooth and geometrically irreducible discriminant cover \(\Deltatilde\to \Delta\), all defined over \(k\). Let \(C\subset W\) be a curve whose support does not contain \(\Delta\) and let \(C_0\) be a (geometric) irreducible component of \(C_{\kbar}\).
    \begin{enumerate}
        \item There exists a surface \(\widetilde{W}_C\subset X\), defined over \(k\), such that \(\pi|_{\widetilde{W}_C}\) is generically \(2\)-to-\(1\) and \(\widetilde{W}_C\) meets each singular fiber of \(X_C\) transversely and each \(\kbar\)-line in a singular fiber with multiplicity \(1\).\label{part:Wtilde}
        \item  There exists a \(\kbar\)-curve \(\frakS_{C_0}\subset (X_\kbar)_{C_0}\) that maps birationally onto \(C_0\) under \(\piX\). In particular, the map \(\piX_*\colon \NS^2\Xbar \to \NS^1\Wbar\) is surjective.\label{part:surjectivity}
        \item The subgroup of \(\NS^2\Xbar\) generated by fibral curves is free of rank \(1\) and is generated by the class \(\gammatilde_0\) of any \(\kbar\)-line contained in
        a singular fiber.\label{part:fibral}
        \item  We have \label{part:Relation}
        \(
        (\widetilde{W}_C)_\kbar \cap (X_\kbar)_{C_0} - 2\frakS_{C_0} \algequiv m\gammatilde_0\) for some integer \(m\).
        
        \item  As subgroups of \(\NS^2\Xbar\), we have \(2(\ker \piX_*) \subset \Z\gammatilde_0\). 
        \label{part:kernel}
        \item \label{part:rationalNS2} If \(X\) is geometrically rational, then \(\NS^2\Xbar\) is free, so in particular we have an exact sequence
        \[
        0 \to \Z\gammatilde_0 \to \NS^2 \Xbar \xrightarrow{\piX_*} \NS^1 \Wbar\to 0.
        \]
    \end{enumerate}
\end{prop}
\begin{proof}

\eqref{part:Wtilde}: For a point \(\pt \in (C \cap \Delta)_{\reduc}\), the fiber \(X_{\pt}\) is a conic over \(\kk(\pt)\) that is geometrically reducible. The antidualizing sheaf of a conic \(X_{\pt}\) has a section whose support does not contain the singular point of \(X_{\pt}\) and the section is reduced.
Thus, if there exists an invertible sheaf \(\scrF\) on \(X\) such that
\begin{enumerate}
    \item \(\scrF_{X_{\pt}}\simeq \omega_{X/W}^\vee|_{X_{\pt}}\) for all \(\pt\), and
    \item \(\HH^0(X, \scrF) \to \oplus_{\pt\in (C\cap \Delta)_{\reduc}}\HH^0(X_{\pt}, \scrF|_{X_{\pt}})\) is surjective,
\end{enumerate}
then there is a global section of \(\scrF\) that gives the desired \(\widetilde{W}_C\).

We claim that for any ample \(H\), \(\scrF :=\omega_{X/W}^\vee\otimes \piX^* H^{\otimes m}\) has the desired properties. Indeed, \((\omega_{X/W}^\vee\otimes \piX^* H^{\otimes m})|_{X_{\pt}} \simeq \omega_{X/W}^\vee|_{X_{\pt}} \simeq \omega_{X_{\pt}}^\vee\) by the adjunction formula, so \(\HH^0(X_{\pt}, (\omega_{X/W}^\vee\otimes \piX^* H^{\otimes m})|_{X_{\pt}})\simeq \HH^0(X_{\pt}, \omega_{X_{\pt}}^\vee)\).

It remains to show that for sufficiently large \(m\), the restriction map on global sections
\begin{equation}\label{eq:global_sections}
    H^0(X, \omega_{X/W}^\vee\otimes \piX^* H^{\otimes m}) \to \bigoplus_{\pt\in (C\cap\Delta)_{\red}} H^0(X_{\pt}, (\omega_{X/W}^\vee\otimes \piX^* H^{\otimes m})|_{X_{\pt}})
\end{equation}
is surjective; we do so by showing that
\(
    \HH^1(X, \omega_{X/W}^\vee\otimes \pi^* H^{\otimes m} \otimes \piX^*\scrI) = 0, 
\)
where \(\scrI\) is the ideal sheaf of \((C\cap\Delta)_{\reduc}\subset W\).  Using the Leray spectral sequence, it suffices to show
\begin{align*}
    \HH^1(W, \pi_*(\omega_{X/W}^\vee\otimes \piX^* H^{\otimes m} \otimes \piX^*\scrI)) = H^1(W, \piX_*(\omega_{X/W}^\vee)\otimes H^{\otimes m} \otimes \scrI) &= 0, \\
    \HH^0(W, R^1\piX_*(\omega_{X/W}^\vee\otimes \pi^* H^{\otimes m} \otimes \piX^*\scrI)) = \HH^0(W, R^1\piX_*(\omega_{X/W}^\vee)\otimes H^{\otimes m} \otimes \scrI) &= 0,
\end{align*}
where the left equalities hold by the push-pull formula.  The first vanishing holds by Serre's criterion for ampleness, and the second holds because \(R^1\piX_*(\omega_{X/W}^\vee)=0\) \cite{Sarkisov-conic-bundles}*{Section 1.5}.

\eqref{part:surjectivity}--\eqref{part:rationalNS2}: As the rest of the parts of the proposition are geometric, henceforth we will assume \(k=\kbar\) and drop \(\kbar\) from the notation.

\eqref{part:surjectivity}: The first statement follows from Tsen's theorem, and the second follows from the first together with Chow's moving lemma.

\eqref{part:fibral}:  The assumption that the discriminant cover is smooth and geometrically irreducible implies that all fibers are rank at least \(2\).  Furthermore, since any two points on \(W\) are algebraically (in fact, rationally) equivalent, any class in the subgroup of fibral curves can be expressed as a sum of lines contained in singular fibers.  Thus, it suffices to show that any two lines contained in singular fibers are algebraically equivalent.  The normalization of \(X_{\Delta}\) is a \(\PP^1\)-bundle \(S\to \Deltatilde\) such that for any closed \(\widetilde{\pt}\in \Deltatilde\) the curve \(S_{\widetilde{\pt}}\simeq\PP^1\) maps isomorphically to a line in \(X_{\piDelta(\widetilde{\pt})}\).  Hence, the statement follows from the fact that \(S_{\widetilde{\pt}}\sim_{\alg} S_{\widetilde{\pt}'}\) for any two points \({\widetilde{\pt}}, {\widetilde{\pt}}'\in \Deltatilde\).

\eqref{part:Relation}: The conic bundle surface \(X_{C_0}\) is birational to a regular conic bundle surface over a smooth model \(\tilde{C}_0\) of \(C_0\). Recall that the N\'eron--Severi group of a regular conic bundle surface \(X_{\tilde{C}_0}\) over an algebraically closed field is freely generated by the class of a section, a smooth fiber \(F\), and one component of each singular fiber. On \(X_{C_0}\) the difference of the strict transforms of \(\widetilde{W}_C\cap X_{C_0}\) and \(2\frakS_{C_0}\) is algebraically equivalent to a combination of fibral curves, and hence on \(X_{C_0}\) we have \(\widetilde{W}_C\cap X_{C_0}-2\frakS_{C_0}\algequiv m \gammatilde_0\) for some integer \(m\).
   
    \eqref{part:kernel}:  Let \(\gamma\in \ker \pi_*\) and let \(Z_i\subset X\) be a collection of integral curves such that \(\gamma = \sum_i n_i Z_i\). By Chow's moving lemma, we may assume \(\piX(Z_i)\not\subset\Delta\) for all \(i\).  Observe that if \(\pi(Z_i)\) is a point, then by~\eqref{part:fibral} we have \(Z_i \in \Z\gammatilde_0\) as desired.  Therefore, we may reduce to the case that \(C_i := \pi(Z_i)\) is a curve for all \(i\).  Let  \(C:= \cup C_i\). For each \(i\), let \(d_i\) be the degree of the map \(Z_i\to C_i\).  Note that \(\pi_*\gamma \algequiv \sum n_id_iC_i\algequiv 0\) implies
    \begin{equation}\label{eq:algtrivial}
    \sum_i (n_id_i)(\widetilde{W}_{C}\cap X_{C_i})\algequiv 0.    
    \end{equation}
     Using the structure of N\'eron--Severi groups of conic bundle surfaces as in the proof of~\eqref{part:Relation}, we conclude that \(Z_i \algequiv d_i \frakS_{C_i} + a_i \gammatilde_0\) for some \(a_i\in\mathbb{Z}\), and so
    \begin{align*}
        2\gamma = \sum_i2n_iZ_i &\algequiv \sum_i \left(2n_id_i \frakS_{C_i} + 2n_ia_i\gammatilde_0\right)\\
        & \algequiv \sum_i \left[(n_id_i)(\widetilde{W}_{C}\cap X_{C_i}) + (2n_ia_i - n_id_im_{C_i})\gammatilde_0\right] &\textup{by}~\eqref{part:Relation}\\
        & \algequiv \left(\sum_i (2n_ia_i - n_id_im_{C_i})\right)\gammatilde_0&\textup{by}~\eqref{eq:algtrivial}.
    \end{align*} 
     
     \eqref{part:rationalNS2}: When \(X\) is a smooth projective \emph{rational} threefold over an algebraically closed field, \(\NS^2 X\) is free by \cite{bw-ij}*{Theorem 3.1(v)}.  For \(\widetilde{W}_C\) as in part~\eqref{part:Wtilde} we have \(\widetilde{W}_C\cdot\gammatilde_0=1\), which since \(X\) is smooth implies that \(\gammatilde_0\) is not divisible in \(\NS^2 X\). Part~\eqref{part:kernel} then implies that the kernel of \(\piX_*\) must be freely generated by \(\widetilde{\gamma}_0\) and part~\eqref{part:surjectivity} gives surjectivity.
\end{proof}

\subsection{Relating cycles on \texorpdfstring{\(X\)}{X} to cycles on 
\texorpdfstring{\(\Delta\)}{Delta} and \texorpdfstring{\(W\)}{P2}}\label{sec:beauville}

Observe that the map \(\piX\) fails to be smooth along the image of a section \(\delta \colon \Delta \to X\).  We write \(\eps \colon X' \to X\) for the blow-up along \(\delta(\Delta)\).  The proper transform \(S\) of \(X_\Delta\) is a \(\PP^1\)-bundle over \(\Deltatilde\).  The exceptional divisor \(E\) of \(\eps\) meets the surface \(S \) transversely along a section of \(S \xrightarrow{p} \Deltatilde\).
We summarize this setup in the following commutative diagram.
\begin{equation}\label{diagram:beauville-maps}
    \begin{tikzcd}
        S \cap E \arrow[rr, hook] \arrow[dr, hook] \arrow[ddr, equal]  && E \arrow[rd, hook] 
        \arrow[dddl, swap, near start] \\
        & S \arrow[rr, hook, "j", crossing over] \arrow[dr, "\nu", swap, crossing over, near start] \arrow[d, "p", swap] && X' \arrow[d, "\eps"] \\
        & \Deltatilde \arrow[d, "\piDelta", swap] & X_\Delta \arrow[dl, "\pi|_{X_\Delta}", near start, shift left=1] \arrow[r, hook, "r'"] & X \arrow[d, "\piX"] \\
        & \Delta \arrow[ur, hook, "\delta", shift left] \arrow[rr, hook, "r", near end] && W
    \end{tikzcd}
\end{equation}
\noindent Of particular interest will be the maps
\begin{equation}\label{eq:pje-ejp}
p_*j^*\eps^* \colon \CH^2\Xbar \to \CH^1\Deltatilde_{\kbar} \quad\text{and}\quad \eps_*j_*p^* \colon  \CH^1\Deltatilde_{\kbar} \to \CH^2\Xbar
\end{equation}
where \(\eps^*\colon\CH^2\Xbar\to\CH^2 X'_{\kbar}\) and \(j^*\colon\CH^2 X'_{\kbar}\to\CH^2 S_{\kbar}\) are the refined Gysin homomorphisms~\cite{Fulton-IsectBook}*{Section 6.6}, and \(p_*\colon\CH^2 S_{\kbar}\to\CH^1\Deltatilde_{\kbar}\) is the pushforward.

\begin{lemma}\label{lem:pullback_pushforward}\hfill
\begin{enumerate}
    \item\label{X} For any class \(\alpha \in \CH^2\Xbar\), we have
        \(\piDelta_* (p_* j^* \eps^* \alpha) = r^*\piX_* \alpha.\)
    \item\label{delta} For any class \(\beta \in \CH^1\Delta_{\kbar}\), we have
        \(\eps_* j_* p^* (\piDelta^* \beta) = \piX^*r_* \beta.\)
\end{enumerate}
\end{lemma}
\begin{proof}
    We will use that the lower right square in \eqref{diagram:beauville-maps} is Cartesian.
    Using that \(r\) is a regular embedding and \(\nu\) is monoidal, we have
    \begin{align*}
        r^*\piX_* \alpha &= (\piX|_{X_\Delta})_*r'^*\alpha &\text{\cite{Fulton-IsectBook}*{Theorem 6.2(a)}}\\
        &= (\piX|_{X_\Delta})_*\nu_*\nu^*r'^*\alpha &\text{\cite{Fulton-IsectBook}*{Proposition 6.7(b)}}\\
        &= (\piDelta \circ p)_*( \eps \circ j)^*\alpha. 
    \end{align*}
    
    \noindent In the other direction, using that \(r\) is proper, \(\pi\) is flat and \(\nu\) is monoidal, we have
    \begin{align*}
        \piX^*r_* \beta &= r'_*(\piX|_{X_\Delta})^*\beta &\text{\cite{Fulton-IsectBook}*{Proposition 1.7}}\\
        &= r'_*\nu_*\nu^*(\piX|_{X_\Delta})^*\beta &\text{\cite{Fulton-IsectBook}*{Proposition 6.7(b)}} \\
        &= (\eps \circ j)_*(\piDelta \circ p)^* \beta. \qedhere
    \end{align*}
\end{proof}

Recall that \(\iota \colon \Deltatilde \to \Deltatilde\) is the (geometrically) fixed-point-free involution whose quotient is \(\Delta\).  A key input to Beauville's computation of the intermediate Jacobian of \(X\) is the following.

\begin{lemma}[\citelist{\cite{beauville-ij}*{Chap. III}\cite{BeltramettiFrancia83}*{Claim 2}}]\label{lem:BeauvilleGroupLem}
Given a class \(x \in \CH^1\Deltatilde_{\kbar}\), we have
    \[
        p_*j^*\eps^*(\eps_*j_*p^*(x)) = (\iota_* - 1) x. \hfill \qed
    \]
\end{lemma}

\begin{prop}\label{prop:CH2-to-PPrym}
The map
\[\CH^2\Xbar \xrightarrow{\;(\pje, \pi_*)\;} \bPic_{\Deltatilde/k}(\kbar) \times \Pic W_{\kbar}\]
is a \(G_k\)-equivariant group homomorphism with image contained in \(\PPrym_{\Deltatilde/\Delta}^{\bPic_{W/k}}(\kbar)\).  
\end{prop}

\begin{proof}
Given a class \(\alpha \in \CH^2\Xbar\), by Lemma~\ref{lem:pullback_pushforward}, we have \(\piDelta_*p_*j^*\eps^* \alpha = r^*\piX_* \alpha\). Hence \((\pje\alpha, \pi_*\alpha)\) is a point of the \((\bPic_{W/k})\)-polarized Prym scheme.  Since the maps \(p\), \(j\), \(\eps\), and \(\pi\) are defined over \(k\), this map is \(G_k\)-equivariant.
\end{proof}

\subsection{Identifying the intermediate Jacobian with the Prym variety}\label{sec:scheme_map}

Using our results from Section~\ref{sec:BWreview}, we can upgrade the maps of the previous section to natural transformations of functors.

\begin{lemma}\label{lem:ejp-pje}
    There exist natural transformations of functors
\[
    \pje \colon \CH^2_{X/k, \fppf} \to \Pic_{\Deltatilde/k, \fppf}, \qquad
    \ejp \colon \Pic_{\Deltatilde/k, \fppf} \to \CH^2_{X/k, \fppf}
\]
    extending the homomorphisms in \eqref{eq:pje-ejp} on \(\kbar\)-points.  When \(X\) is geometrically rational, these give rise to morphisms of \(k\)-group schemes 
\[    
\pje  \colon \bCH2_{X/k} \to \bPic_{\Deltatilde/k}, \qquad 
    \ejp \colon \bPic_{\Deltatilde/k} \to \bCH2_{X/k}.
\]
    In this case, there is also a morphism \(\pi_* \colon \bCH2_{X/k} \to \bPic_{W/k}\) and the product
    \[\bCH2_{X/k} \xrightarrow{\;(\pje, \pi_*)\;} \bPic_{\Deltatilde/k} \times \bPic_{W/k}\]
    has image contained in \(\PPrym_{\Deltatilde/\Delta}^{\bPic_{W/k}}\).
\end{lemma}
\begin{proof}
    The existence of the natural transformations/morphisms \(\ejp\) and \(\pje\) follows by composing the natural transformations defined in Propositions~\ref{prop:epsilon}, \ref{prop:jp}, and \ref{prop:pj}, which is possible since \(\eps\) is birational, \(p\) is smooth, and \(j\) is a regular embedding.  Existence of the morphism \(\pi_*\) follows from Proposition~\ref{prop:pi}.

    Finally, since \(\bCH2_{X/k}\) is a smooth group scheme, it suffices to check the image of \((\pje, \pi_*)\) on \(\kbar\)-points, which is Proposition~\ref{prop:CH2-to-PPrym}.
\end{proof}

Beauville showed that, over algebraically closed fields of characteristic not \(2\), the Prym variety \(\bPrym_{\Deltatilde/\Delta}\) with the homomorphism \((\iota_*-1)\circ(\ejp)^{-1}\) of Section~\ref{sec:beauville} is the algebraic representative for \((\CH^2 X_\kbar)^0\) (i.e., Murre's definition of the intermediate Jacobian \cite{Murre-applications-K-theory}) \cite{beauville-ij}*{Proposition 3.3}.  We build on Beauville's work to extend the result to over any field of characteristic different from \(2\) (not necessarily algebraically closed).  In addition, when \(X\) is geometrically rational, we show that \((\bCH2_{X/k})^0\) is isomorphic to the Prym variety.
\begin{thm}\label{thm:jac0}
    Let \(k\) be a field of characteristic different from \(2\). The homomorphism \(\ejp\) of Lemma~\ref{lem:ejp-pje}  induces a \(G_k\)-equivariant group isomorphism
    \[\bPrym_{\Deltatilde/\Delta}(\kbar)\xrightarrow{\;\ejp\circ(\iota_* - 1)^{-1}} (\CH^2\Xbar)^0.\] 
    Here \((\iota_*-1)^{-1}\) denotes the inverse of the isomorphism \(\bPic^0_{\Deltatilde/k}/\piDelta^*\bPic^0_{\Delta/k}\simeq\bPrym_{\Deltatilde/\Delta}\) induced by \(\iota_*-1\) (Lemma~\ref{lem:P-Ptilde-characterization}).
    
    Moreover, if \(X\) is geometrically rational, 
    this induces  an isomorphism    
    \[\bPrym_{\Deltatilde/\Delta}\xrightarrow{\;\ejp\circ(\iota_* - 1)^{-1}} (\bCH2_{X/k})^0\] 
    of principally polarized abelian varieties.
\end{thm}
\begin{proof}
    First we show that the kernel of \(\ejp \) has the same \(\kbar\)-points as \(\piDelta^*\bPic^0_{\Delta/k}\).
    By Lemma~\ref{lem:BeauvilleGroupLem}, the kernel of \(\eps_*j_*p^* \colon \Pic\Deltatilde_{\kbar} \to \CH^2\Xbar\) is contained in that of \(\iota_* - 1\colon \Pic\Deltatilde_{\kbar} \to \Pic\Deltatilde_{\kbar}\), which, by Lemma~\ref{lem:P-Ptilde-characterization} is equal to \(\piDelta^*\Pic^0\Delta_{\kbar}\).
    For the opposite containment, we need to show that \(\ejp(\piDelta^*\Pic^0\Delta_\kbar) = 0\) in \(\CH^2X_\kbar\).  Note that \((\ejp)\circ\piDelta^* = \piX^*r_*\) by Lemma \ref{lem:pullback_pushforward}, but, since \((\CH^2\Wbar)^0 = 0\), the map \(r_*\) is identically \(0\) on \(\Pic^0\Delta_\kbar\).
    
    Hence \(\piDelta^*\bPic^0_{\Delta/k}\) is the maximal reduced subscheme of the kernel and \(\ejp\) induces a homomorphism \(\bPic^0_{\Deltatilde/k}/\piDelta^*\bPic^0_{\Delta/k}(\kbar)\to(\CH^2\Xbar)^0\), and, if \(X\) is geometrically rational, a morphism  \(\bPic^0_{\Deltatilde/k}/\piDelta^*\bPic^0_{\Delta/k}\to(\bCH2_{X/k})^0\). Thus, precomposing with the isomorphism \((\iota_*-1)^{-1}\) yields a group homomorphism \(\bPrym_{\Deltatilde/\Delta}(\kbar)\to(\CH^2\Xbar)^0\), and, if \(X\) is geometrically rational, a morphism of abelian varieties \(\bPrym_{\Deltatilde/\Delta}\to(\bCH2_{X/k})^0\). We claim these are both isomorphisms.
    For the group homomorphism, it suffices to show surjectivity. For this, the proof of surjectivity in \cite{beauville-ij}*{Th\'eor\`eme 3.1} holds when \(\PP^2\) is replaced by a smooth surface \(W_\kbar\) with \((\CH^q W_\kbar)^0=0\) for all \(q\).

To check that \(\bPrym_{\Deltatilde/\Delta}\to(\bCH2_{X/k})^0\) is an isomorphism of abelian varieties, it suffices to base change to \(\kbar\). Since both 
\[
    \begin{array}{rccrcll}
        \big( & \bPrym_{\Deltatilde/\Delta} &, \;&
(\iota_*-1)\circ(\ejp)^{-1} & \colon & (\CH^2X_\kbar)^0\to\bPrym_{\Deltatilde/\Delta}(\kbar)&\big), \quad \textup{and}\\
\big( &(\bCH2_{X_\kbar/\kbar})^0 &, \;& \psi^2_X &\colon & (\CH^2X_\kbar)^0\to (\bCH2_{X_\kbar/\kbar})^0(\kbar) & \big)
    \end{array}
\]
are the algebraic representative for algebraically trivial codimension two cycles on \(X_\kbar\) by \cite{beauville-ij}*{Proof of Proposition 3.3} (which only requires that \((\CH^q W_{\kbar})^0=0\) for all \(q\)) and \cite{bw-ij}*{Theorem 3.1(vi)}, the claim follows from the universal property \cite{beauville-ij}*{Definition 3.2.3} of the algebraic representative. The isomorphism \(\bPrym_{\Deltatilde/\Delta}\simeq(\bCH2_{X/k})^0\) respects the principal polarizations because the principal polarization on \((\bCH2_{X/k})^0\) is induced by the one on the algebraic representative over \(\kperf\) \cite{bw-ij}*{Theorem 3.1(vi)}.
\end{proof}

    \subsection{Proof of Theorem~\ref{thm:IJTtorsorsOfConicBundles}}\label{sec:proof_IJT}
    
    By Proposition~\ref{prop:CH2-to-PPrym}, the map \((\pje, \pi_*)\) gives a Galois-equivariant group homomorphism from \(\CH^2 \Xbar\) to \(\PPrym^{\bPic_{W/k}}_{\Deltatilde/\Delta}(\kbar)\).  We will prove that this is the desired map in Theorem~\ref{thm:IJTtorsorsOfConicBundles}\eqref{part:IsoOnkbarPts}.
    In addition, combining Theorem \ref{thm:jac0}, Lemma \ref{lem:P-Ptilde-characterization} and Lemma \ref{lem:BeauvilleGroupLem}, we have the diagram
    \begin{center}
        \begin{tikzcd}[column sep = huge]
            P(\kbar) \arrow{r}{(\iota_*-1)^{-1}}[below]{\sim} \arrow[bend right = 15, equal]{rrr} & \frac{\bPic^0\Deltatilde}{\piDelta^*\bPic^0\Delta}(\kbar) \arrow{r}{\ejp}[below]{\sim} & (\CH^2 \Xbar )^0  \arrow[r, "{(\pje,\, \pi_*)}"] & P(\kbar),
        \end{tikzcd}
    \end{center}
    which
    shows that the map \((\pje, \pi_*)\) is an isomorphism on \(\kbar\)-points from \((\CH^2 \Xbar )^0\) onto the identity component \(\bPrym_{\Deltatilde/\Delta}\subset \PPrym_{\Deltatilde/\Delta}^{\bPic_{W/k}}\). 
        Hence, for arbitrary \(\gamma\in (\NS^2X_{\kbar})^{G_k}\), the map \((\pje, \pi_*)\) yields a \(G_k\)-equivariant isomorphism from \((\CH^2 \Xbar)^\gamma\) to the \(\kbar\)-points of one of the two geometric irreducible components of \(V_{\piX_*\gamma}\) by Lemma~\ref{lem:pullback_pushforward}\eqref{X}. Note that since \(\gamma\in (\NS^2X_{\kbar})^{G_k}\) and this isomorphism is \(G_k\)-equivariant, we deduce that a (and hence both) geometric irreducible component(s) of \(V_{\piX_*\gamma}\) must be defined over \(k\).
    Since \((\pje, \pi_*)\) is a group homomorphism, this isomorphism is compatible with actions of the identity components.  
    
    To conclude~\eqref{part:IsoOnkbarPts} it remains to prove surjectivity of \((\pje, \pi_*)\).  Since each split connected component of \(\PPrym_{\Deltatilde/\Delta}^{\bPic_{W/k}}\) is a torsor under \(\bPrym_{\Deltatilde/\Delta}\) and we have already shown that \(\bPrym_{\Deltatilde/\Delta}(\kbar)\) is in the image of \(\pje\), it suffices to show that for each connected component of \(\PPrym_{\Deltatilde/\Delta}^{\bPic_{W/k}}\), there is a \(\gamma\in \CH^2\Xbar\) such that \((\pje\gamma, \pi_*\gamma)\) lies in that component.
    Since \(W\) is geometrically rational, \(\NS{}^1 \Wbar = \Pic \Wbar\) and so Proposition~\ref{prop:NS2}\eqref{part:surjectivity}  implies that for all \(D\in \Pic \Wbar\), there exists \(\gamma \in \CH^2 \Xbar\) such that  \(\pi_* \gamma =D\).  Further, since \(\HH^1(G_L, \Z\gammatilde_0) = 0\) for any subgroup \(G_L\subset G_k\), if \(D\) is fixed by a subgroup \(G_L\) then \(\gamma\) can be chosen to be fixed by that subgroup as well.  Hence, by Lemma~\ref{lem:pullback_pushforward}\eqref{X}, we have \((\pje \gamma, \pi_*\gamma) \in V_D\), and at least one of the two connected components of \(V_{D}\) is contained in the image of \(\pje\). Further, since \(\pje\) is a group homomorphism, if both \(P\) and \(\Ptilde\) are contained in the image of \(\pje\), then the same is true for all \(D \in \Pic \Wbar\).
    
    Let \(\tilde{D}\in \Pic \Deltatilde_{\kbar}\).  By definition of \(p,j,\eps\), we have that \(\ejp \tilde{D}\) is a fibral curve for \(\piX\), so by Proposition~\ref{prop:NS2}\eqref{part:fibral}, \(\ejp \tilde{D}\algequiv (\deg \tilde{D})\gammatilde_0\).  By Lemma~\ref{lem:BeauvilleGroupLem}, \(\pje(\ejp \tilde{D})=(\iota_*-1)\tilde{D}\)  and so by  Lemma~\ref{lem:P-Ptilde-characterization}, \(\pje(\ejp \tilde{D})\) lands in \(P\) or \(\Ptilde\) depending on the parity of \(\deg\tilde{D}\). This completes the proof of~\eqref{part:IsoOnkbarPts} and also proves~\eqref{part:gamma0tilde}.
    
    Now we assume that \(X\) is geometrically rational and prove~\eqref{part:IsoAsTorsors}.  The existence of the morphism of group schemes \((\pje, \pi_*)\) follows from Lemma~\ref{lem:ejp-pje}.  Surjectivity and the fact that it is an isomorphism on connected components follows from the Galois equivariance together with the properties on \(\kbar\)-points in \eqref{part:IsoOnkbarPts}.  The explicit result about \((\bCH2_{X/k})^{n \gammatilde_0}\) follows from ~\eqref{part:gamma0tilde}.\qed

\subsection{Explicit computation of  \texorpdfstring{\(\pje\)}{pje} for certain curves in \texorpdfstring{\(X\)}{X}}\label{subsec:Explicitpje}

\begin{lemma}\label{lem:RelativeSym}
    The \(m\)-fold relative symmetric product \(\Sym_{\Delta}^m\Deltatilde := \left(\Deltatilde\times_{\Delta}\Deltatilde \times_{\Delta}\cdots \times_{\Delta}\Deltatilde\right)/S_m\) has \(\lfloor\frac{m}{2}\rfloor + 1\) irreducible components that are given by the images of the morphisms \(\phi_i\):
    \begin{align*}
        (\underbrace{\id, \cdots, \id}_{m-i}, \underbrace{\iota, \cdots, \iota}_{i})&\colon \Deltatilde \to \Sym_{\Delta}^m\Deltatilde, \quad 0\leq i < \frac{m}{2},\\ 
        (\underbrace{\piDelta^*, \cdots, \piDelta^*}_{\frac{m}{2}})&\colon \Delta \to \Sym_{\Delta}^m\Deltatilde,\quad i = \frac{m}{2} \textup{ and }m\textup{ even}.
    \end{align*}
\end{lemma}
\begin{proof}
    The relative fiber product \(\Deltatilde\times_{\Delta}\Deltatilde\) is isomorphic to two copies of \(\Delta\) given by the maps \((\id,\id)\) and \((\id,\iota)\).  Thus, by induction, the \(m\)-fold relative fiber product \(\Deltatilde\times_{\Delta}\Deltatilde \times_{\Delta}\cdots \times_{\Delta}\Deltatilde\) is isomorphic to \(2^{m-1}\) copies of \(\Deltatilde\) labeled by the length \(m-1\) words in \(\iota\) and \(\id\).  Quotienting by the symmetric group gives the desired statement.
\end{proof}

\begin{prop}\label{prop:pjeExplicit}
    Let \(C\subset W\) be a smooth geometrically integral curve. Then the following properties hold:
    \begin{enumerate}
        \item The conic bundle surface \(X_{C}\to C\) is regular away from \(\delta(\Delta\cap C)\) and for a point \(\pt\in\Delta\cap C\) with \(m_{\pt} := \mult_{\pt}(\Delta\cap C)\), we have that \(\delta(\pt)\subset X_C\) is a \(A_{m_{\pt}-1}\)-singularity.\label{part:Singularities}
        \item Let \(\hat{\pi}_C\colon \hat{X}_C\to C\) be the minimal resolution of \(X_C\to C\), which is obtained by a series of \(\lfloor\frac{m_{\pt}}{2}\rfloor\) blow-ups above each \(\pt \in \Delta \cap C\).  For any point \(\pt\in \Delta\cap C\), there is a unique bijection
        \[\begin{array}{rcl}
        \left\{\begin{array}{l}
            \textup{irreducible components}  \\
            \textup{of the fiber } \hat{\pi}_C^{-1}(\pt) 
        \end{array}\right\} &\longleftrightarrow &
        \left\{\begin{array}{l}
            \textup{closed points}  \\
            \textup{of } (\Sym^{m_{\pt}}_{\Delta}\Deltatilde)_{\pt}
        \end{array}\right\} \\
        \end{array}
        \]
        with the following properties:\label{part:IrredComponents}
        \begin{enumerate}
            \item  for \(i>0\), the exceptional curves from the \(i\)-th blow-up map correspond to closed points in the image of \(\phi_i\), \label{cond:ithblow-up}
            \item for any \(\widetilde{\pt}\in \Deltatilde_{\pt}\), the strict transform of \(\nu_*p^*\widetilde{\pt}\) corresponds to \(\phi_0(\widetilde{\pt})\), and \label{cond:phi0}
            \item if \(m_{\pt}>1\) then for any \(\widetilde{\pt}\in \Deltatilde_{\pt}\), the component corresponding to \(\phi_i(\widetilde{\pt})\) meets only the components corresponding to \(\phi_{i-1}(\widetilde{\pt})\) or \(\phi_{i+1}(\widetilde{\pt})\).\footnote{If \(m_{\pt}=2M\) for some integer \(M\), then \(\phi_{M}(\widetilde{\pt}) = \phi_{M}(\iota(\widetilde{\pt}))\) so this characterization should be taken to mean that the component corresponding to \(\phi_{M}(\widetilde{\pt})\) meets \(\phi_{M - 1}(\widetilde{\pt})\) and \(\phi_{M - 1}(\iota(\widetilde{\pt}))\).} \label{cond:Intersections}
        \end{enumerate}

        \item Let \(\hat{\frakS}\subset \hat{X}_{C}\) be a section of \(\hat{X}_C \to C\) defined over \(k\).
        Since \(\hat{\frakS}\) intersects a unique irreducible component of \(\hat{\pi}_C^{-1}(\pt)\) which must be split over \(\kk(\pt)\), by~\eqref{part:IrredComponents} we obtain, for each \(\pt\in \Delta\cap C\), a point \(D_{\hat{\frakS},\pt}\in \Sym^{m_{\pt}}(\Deltatilde_{\pt})\) of relative degree \(1\), and these points satisfy\label{part:pjeSection}
        \[
            \pje[\im\hat{\frakS}] = \left[\sum_{\pt \in \Delta\cap C} D_{\hat{\frakS},\pt}\right].
        \]
    \end{enumerate}
\end{prop}
\begin{proof}
    \textbf{\eqref{part:Singularities}:} The morphism \(\piX\) is smooth away from \(\delta(\Delta)\), so the same holds when restricting to \(C\).  To understand the structure of \(X_C\) locally around a point \(\delta(\pt)\) for \(\pt\in\Delta\cap C\), we appeal to~\cite{Sarkisov-conic-bundles}*{Proof of Proposition 1.8(5.b)}, which shows that there is an open neighborhood \(\pt\in U\subset W\) such that \(X_U\to U\) is locally given by
    \[
    V(ax^2 + by^2 + cz^2 ) \subset U\times\PP^2_{[x:y:z]}
    \]
    where \(a,b,c\) are functions on \(U\) such that \(V(a) = \Delta\cap U\) and \(b,c\) do not vanish at \(\pt\).  So if \(\pi_{\pt}\) is a local parameter of \(\pt\) in \(C\), then we have that locally around \(\delta(\pt)\), \(X_C\) is defined by \(\pi_{\pt}^{m_{\pt}} = ug_{\pt}\) for some rank \(2\) quadratic form \(g_{\pt}\) and some unit \(u\).  Hence \(\delta(\pt)\subset X_C\) is a \(A_{m_{\pt}-1}\)-singularity (which means that \(\delta(\pt)\) is smooth if \(m_{\pt}=1\)).

    \textbf{\eqref{part:IrredComponents}:} Let \(\pt\in \Delta\cap C\).  If \(m_{\pt}=1\), then the existence of a bijection between these two sets follows from the definition of \(\Deltatilde\), and the uniqueness follows from Condition~\eqref{cond:phi0}.  The other two conditions vacuously hold in this case. Henceforth we assume that \(m_{\pt}>1\) so in particular \(\delta(\pt)\) is singular in \(X_C\).
    
    Over \(\kbar\), any \(A_{n}\) singularity is resolved by \(\lfloor\frac{n+1}{2}\rfloor\) iterated blow-ups with two exceptional curves introduced in the \(i\)-th blow-up for all \(i<\frac{n+1}{2}\), and a unique exceptional curve introduced when \(i = \frac{n+1}{2}\) (which necessarily implies that \(n\) is odd).  Further, at the \(i\)-th stage, the new exceptional curves are introduced in the middle of the chain. Over a non-closed field, the number of irreducible components in the exceptional divisor at each blow-up depends on whether the defining equation modulo \(\mm_{\pt}^3\) factors or not.  In our case, this is equivalent to whether \(X_{\pt}\) has one or two irreducible components or, equivalently, whether \(\Deltatilde_{\pt}\) has one or two irreducible components. Precisely, if \(\Deltatilde_{\pt}\) is irreducible, then the exceptional curve at each stage is irreducible. Since there is a unique exceptional curve from each of the \(\lfloor\frac{m_{\pt}}{2}\rfloor\) blow-ups (and a unique component of \(X_{\pt}\) which we can think of as coming from the \(0\)-th blow-up),  Lemma~\ref{lem:RelativeSym}, and conditions~\eqref{cond:ithblow-up} and~\eqref{cond:phi0} characterize a unique bijection.  Condition~\eqref{cond:Intersections} holds because, as remarked above, the exceptional curves at the \(i\)-th stage appear in the middle of the chain.

    It remains to consider the case when \(m_{\pt}>1\) and \(\Deltatilde_{\pt} = \{\widetilde{\pt}, \iota(\widetilde{\pt})\}\) is reducible.  Then the above discussion implies that there are
    \[
        2\cdot \#\left\{0\leq i < \frac{m_{\pt}}{2}\right\} + \begin{cases} 1 & m_{\pt}\equiv 0 \bmod 2\\
        0 & m_{\pt}\equiv 1 \bmod 2\end{cases}
    \]
    irreducible components of the fiber \(\hat{\pi}_C^{-1}(\pt)\).  Lemma~\ref{lem:RelativeSym} implies that there are the same number of closed points of \((\Sym^{m_{\pt}}_{\Delta}\Deltatilde)_{\pt}\).  Further, the irreducible components of the fiber \(\hat{\pi}_C^{-1}(\pt)\) form a chain, and the ends of the chain correspond to irreducible components of \(X_{\pt}\), which are equal to the strict transforms under \(\hat{X}_C \to X_C\) of the lines \(\nu_*p^*\widetilde{\pt}\) and \(\nu_*p^*\iota(\widetilde{\pt})\).
    Thus, condition~\eqref{cond:phi0} characterizes where the ends of the chain go and condition~\eqref{cond:Intersections} characterizes where the other irreducible components go.  Since the exceptional curves at the \(i\)-th stage appear in the middle of the chain, condition~\eqref{cond:Intersections} implies~\eqref{cond:ithblow-up}.
    
    \textbf{\eqref{part:pjeSection}:} We will use the notation from~\eqref{diagram:beauville-maps}. For \(w\in\Delta\cap C\), let \(E_w\) denote the fiber of \(X'=\Bl_{\delta(\Delta)}X\to X\) over \(\delta(w)\). Let \(\frakS'\) be the proper transform of \(\im\hat\frakS\) in \(X'\). Then 
    \[
    \eps^*[\im\hat{\frakS}] = \frakS' + \sum_{\pt\in\Delta\cap C} n_\pt E_\pt
    \] 
    where each \(n_\pt\) is the intersection multiplicity of \(\delta(\Delta)\cdot\im\hat{\frakS}\) at \(\delta(\pt)\). Note that since \(E\) meets \(S\) transversely, \(p_*j^* E_\pt = \Deltatilde_\pt\) for each \(\pt\in\Delta\cap C\).
    Since the generic point of \(\frakS'\) is not contained in \(S\), \(j^*\frakS' = S\cdot \frakS'\) is an effective \(0\)-cycle on \(S\). Further, since \(\hat{\frakS}\) is a section, \(S\cdot \frakS' = \sum_{\pt\in \Delta\cap C} r_{\pt} s_{\pt}\) for some nonnegative integer \(r_{\pt}\) and some closed \(s_{\pt}\in S_{\pt}\) that has relative degree \(1\) over \(\pt\). (If \(r_{\pt}=0\) then we do not require the existence of such an \(s_{\pt}\).) All together, we may write
    \[
    \pje[\im\hat{\frakS}] = p_*j^*\frakS' + \sum_{\pt\in\Delta\cap C} n_\pt p_*j^*E_\pt = \sum_{\pt\in\Delta\cap C}\left(r_{\pt} p_*s_{\pt} + n_{\pt}\Deltatilde_{\pt}\right).
    \]
    The projection formula together with the equalities \(\eps^*\pi^*\Delta = 2E+S\) and \(E=\eps^*\delta(\Delta)\) yields, 
    \begin{align*}
        \sum_{\pt\in \Delta\cap C} m_\pt \pt= \Delta \cdot \pi_*\eps_*\frakS' 
        & = \pi_*\eps_*((2E+S)\cdot\frakS') \\
       & =  2\pi_*\eps_*(E\cdot\frakS') + \pi_*\eps_*\left(\sum_{\pt\in \Delta\cap C} r_{\pt} s_{\pt}\right)\\
       & =  2\pi_*(\delta(\Delta)\cdot\im\hat{\frakS})+\sum_{\pt\in \Delta\cap C} r_{\pt}\pi_*\eps_* s_{\pt} \\
              & = \sum_{\pt\in \Delta\cap C}\left(2\pi_*n_{\pt}\delta(\pt) + r_{\pt}\pi_*\eps_* s_{\pt} \right) = \sum_{\pt\in \Delta\cap C}\left(2n_{\pt} + r_{\pt}\right)\pt.
    \end{align*} 
    In particular, 
    for each \(\pt\in \Delta\cap C\), we have \(m_{\pt} = 2n_{\pt} + r_{\pt}\), and so
    \[
    \pje[\im\hat{\frakS}] = p_*j^*\frakS' + \sum_{\pt\in\Delta\cap C} n_\pt p_*j^*E_\pt = \sum_{\pt\in\Delta\cap C}\left((m_{\pt} - 2n_{\pt}) p_*s_{\pt} + n_{\pt}\Deltatilde_{\pt}\right).
    \]
    Thus, we have reduced to showing that for each \(\pt\in \Delta\cap C\), we have \(D_{\hat\frakS, \pt} = (m_{\pt} - 2n_{\pt}) p_*s_{\pt} + n_{\pt}\Deltatilde_{\pt}\).
    
    Henceforth, we work with a fixed \(\pt\in \Delta\cap C\). Let \(i\leq\lfloor \frac{m_{\pt}}{2}\rfloor\) be such that \(\hat\frakS\) intersects an exceptional curve from the \(i\)-th blow-up (where \(i=0\) means that \(\hat\frakS\) intersects the strict transform of an irreducible component of \(X_{\pt}\)) in \(\hat{\pi}^{-1}(\pt)\). By properties of blow-ups, \(\im\hat\frakS\subset X\) passes through \(\delta(\pt)\) with multiplicity \(i\), or, in other words, \(n_{\pt}=i\).
    
    If \(\Deltatilde_{\pt}\) is irreducible, then \(S_{\pt}\) contains no closed points of relative degree \(1\) over \(\pt\) and so \(r_{\pt} = m_{\pt} - 2n_{\pt} = m_{\pt} - 2i=0\). Thus we wish to prove that \( D_{\hat\frakS, \pt} = \frac{m_{\pt}}{2}\Deltatilde_{\pt}\).  However, under this irreducibility assumption, \(\frac{m_{\pt}}{2}\Deltatilde_{\pt}\) is the unique point in \((\Sym^{m_{\pt}}_{\Delta}\Deltatilde)_{\pt}\) that has relative degree \(1\) over \(\pt\).  Hence, we have the desired equality.
    
    Henceforth, we assume that \(\Deltatilde_{\pt} = \{\widetilde{\pt}, \iota(\widetilde{\pt})\}\) is reducible.  After possibly interchanging \(\widetilde{\pt}\) and \(\iota(\widetilde{\pt})\), we may assume that \(D_{\hat\frakS, \pt}=\phi_i(\widetilde{\pt}) = (m_{\pt}-i)\widetilde{\pt} + i\iota(\widetilde{\pt})\).  Thus, to prove
    \[
        (m_{\pt}-i)\widetilde{\pt} + i\iota(\widetilde{\pt}) = (m_{\pt} - 2i) p_*s_{\pt} + i\Deltatilde_{\pt},
    \]
    it suffices to show that if \(2i < m_{\pt}\) then \(p_*s_{\pt} = \widetilde{\pt}\), where \(s_{\pt}\) is the unique point in \(\frakS'\cap S\) lying over \(\pt\) (if \(2i = m_{\pt}\) then the above equality holds already).  
    
    Since \(D_{\hat\frakS, \pt}=\phi_i(\widetilde{\pt})\), by~\eqref{cond:Intersections} and~\eqref{cond:phi0}, the component that meets \(\hat\frakS\) is closer to the strict transform of \(\nu_*p^*\widetilde{\pt}\) than it is to the strict transform of \(\nu_*p^*\iota(\widetilde{\pt})\).  Thus, \(\frakS'\) meets \(S_{\widetilde{\pt}}\) and so \(p_*s_{\pt}=\widetilde{\pt}\).
    \end{proof}

\section{Conic bundle threefolds obtained as double covers of \texorpdfstring{\(\PP^1\times \PP^2\)}{P1xP2}}\label{sec:DoubleCover}
    
    A double cover \(\piDoubleCover\colon Y \to \PP^1_{[t_0:t_1]}\times \PP^2_{[u:v:w]}\) branched over a \((2,2)\) divisor is given by 
    \begin{equation}\label{eqn:DoubleCoverConst}  
        z^2 = t_0^2Q_1 + 2t_0t_1Q_2 + t_1^2Q_3,
    \end{equation}
    for some choice of quadrics \(Q_1,Q_2,Q_3\in k[u,v,w]\). Recall from Theorem~\ref{thm:EqnsForCurves} that a triple of quadrics for which \(\Delta=V(Q_1Q_3-Q_2^2)\) is smooth determines an \'etale double cover \(\piDelta\colon \Deltatilde\to \Delta\).  Furthermore, given \(M\in \PGL_2= \Aut(\PP^1)\), changing coordinates on \(\PP^1\) by \(M\) corresponds to replacing the quadratic forms with \(M\cdot(Q_i)\) as explained in Theorem~\ref{thm:EqnsForCurves}\eqref{part:Qis}.  Thus the isomorphism class of this double cover depends only on \(\Deltatilde\to \Delta\) and not the specific choice of the quadratic forms \(Q_i\), and so we denote the double cover given by~\eqref{eqn:DoubleCoverConst} by \(\DoubleCover\).  When the double cover \(\Deltatilde \to \Delta\) is unambiguous, we write \(Y = \DoubleCover\).  This family of Fano threefolds is \textnumero2.18 in the Mori--Mukai classification \cite{mori-mukai}.

    In Section~\ref{sec:DoubleCoverProps}, we prove some basic properties about this double cover, and in Section~\ref{sec:DoubleCoverAndCalQ} we show how these double covers arise naturally from the geometry of \(\Deltatilde\to \Delta\).  In Section~\ref{sec:DoubleCoverIJT}, we determine the N\'eron--Severi group of algebraic curve classes on \(Y\) and combine the results from Sections~\ref{sec:PrymPlaneQuartics} and~\ref{sec:IJ-torsors} to give an extended version (Theorem~\ref{thm:IJ-torsors-double-cover}) of Theorem~\ref{thm:IJTtorsorsOfConicBundles} for this particular family of conic bundles. 
    Lastly, in Section~\ref{sec:IJTvanishingConsequences}, we deduce some consequences of a vanishing IJT obstruction and use these to prove that, over fields \(k\) with \((\Br k)[2] = 0\), the IJT obstruction characterizes rationality for degree \(4\) conic bundles.

\subsection{Properties of the double cover \texorpdfstring{\(\DoubleCover\)}{Y}}\label{sec:DoubleCoverProps}

\begin{prop}\label{prop:DoubleCover-structures}
\hfill
    \begin{enumerate}
        \item The threefold \(Y\) is smooth.
        \item The second projection \(\piX = \pi_2 \colon Y\to\PP^2_{[u:v:w]}\) gives \(Y\) the structure of a conic bundle with discriminant cover \(\piDelta\colon\Deltatilde\to\Delta\), and so \(Y\) is geometrically rational. 
        \item The projection \(\piQuadricSurface\colon Y\to\PP^1_{[t_0:t_1]}\) gives \(Y\) the structure of a quadric surface bundle. In particular, if \(\pi_1\) has a multisection of odd degree, then \(Y\) is rational.\label{part:OddDegreeMultisection}
        \item The subgroup of \(\NS^2 \Ybar\) generated by curves in fibers of \(\pi_1\) is free of rank \(1\) and is generated by the class of a \(\kbar\)-line in a fiber, which we denote \(\gamma_1\).  In particular, \(\NS^2 \Ybar\) is generated by \(\gamma_1\) and \(\tilde{\gamma_0}\), the class of a \(\kbar\)-line contained in a singular fiber of \(\piX\).\label{part:FibralCurves}
        \item If \(\Deltatilde(k)\neq\emptyset\), then \(\pi_1\) has a section and so \(Y\) is \(k\)-rational.
        \label{part:PointOnDeltatilde}
    \end{enumerate}
\end{prop}
\begin{remark}
    For \(\pi_1\) to have a section, it is necessary that \(\pi_1\) is surjective on \(k\)-points.  Over the real numbers, this necessary condition is also sufficient \cite{Witt-quadratic-forms}*{Satz 22}.
\end{remark}
\begin{proof}
    The smoothness of \(Y\) is verified by the Jacobian criterion, using the smoothness of \(\Delta\) and \(\Deltatilde\). The conic bundle and quadric surface bundle structures on \(Y\) can be verified by checking fibers; geometric rationality of \(Y\) then follows from \cite{prokhorov-rationality-conic-bundles}*{Corollary 5.6.1}. Furthermore, quadric surfaces are rational as soon as they have a point (projection from the point gives a parametrization), which, by Springer's Theorem occurs if and only if the surface has a point of odd degree~\cite{springer}. So~\eqref{part:OddDegreeMultisection} follows from Springer's Theorem applied to the generic fiber of \(\piQuadricSurface\).
    
    To prove~\eqref{part:FibralCurves} we work over \(\kbar\).  The fibers of \(\piQuadricSurface\) are all quadric surfaces of rank \(3 \) or \(4\), and any curve class on such a quadric is a linear combination of classes of lines on the quadric. The parameter space \(\calF_1(Y/\PP^1)\) of such lines is connected, and so all lines contained in a fiber of \(Y \to \PP^1\) are algebraically equivalent.  The second statement of~\eqref{part:FibralCurves} follows from the first statement and Proposition~\ref{prop:NS2}\eqref{part:rationalNS2}.

    To deduce~\eqref{part:PointOnDeltatilde}, we observe that for any \(\tilde{p}\in \Deltatilde\), the preimage \(\piDoubleCover^{-1}(\PP^1\times \piDelta(\tilde{p}))\) has two irreducible components and they each map with degree \(\deg(\tilde{p})\) under \(\pi_1\).  In particular, if  \(\tilde{p}\in \Deltatilde(k)\), then  each irreducible component of \(\piDoubleCover^{-1}(\PP^1\times \piDelta(\tilde{p}))\) is a section of \(\pi_1\).
\end{proof}

\subsection{The relationship between \texorpdfstring{\(Y\)}{Y} and \texorpdfstring{\(\calQ\)}{mathcal Q}}\label{sec:DoubleCoverAndCalQ}

Recall that \(\calQ \to \PP^1\) is the family defined in Theorem \ref{thm:EqnsForCurves}, which consists of rank at most 4 quadric threefolds and \(\Gamma_{\Deltatilde/\Delta}\) denotes the Prym curve (Definition \ref{def:prym-curve}).

\begin{prop}\label{prop:BlowupQuadricFibration}\hfill
\begin{enumerate}
    \item\label{part:YSteinFactorization} The rational map induced by projection \(\rho\colon \calQ \dashrightarrow \PP^1\times \PP^2, \; ([t_0:t_1],[u:v:w:r:s])\mapsto ([t_0:t_1],[u:v:w])\) is regular away from \(V(u,v,w,t_0r + t_1s)\).  Resolving the indeterminacy gives a 
    fibration of rank at most \(2\) conics \(\tilde{\rho}\colon \Bl_{V(u,v,w,t_0r + t_1s)}\calQ \to \PP^1\times\PP^2\) with the Stein factorization
    \[
        \Bl_{V(u,v,w,t_0r + t_1s)}\calQ \longrightarrow Y \xrightarrow{\piDoubleCover} \PP^1\times \PP^2.
    \]

    \item The surface \(\PP^1\times \Deltatilde\subset \Bl_{V(u,v,w,t_0r + t_1s)}\calQ\) surjects onto \(Y_{\Delta}=Y\times_{\PP^2}\Delta\) via the map in \eqref{part:YSteinFactorization}, thereby giving an explicit realization of \(S\xrightarrow{\nu}Y_{\Delta}\) from~\eqref{diagram:beauville-maps}.
    \label{part:normalizationYDeltatilde}
    \item The morphism \(\tilde\rho\) induces an isomorphism \(\calF_2(\calQ/\PP^1) \xrightarrow{\sim}\calF_1(Y/\PP^1)\), which fits into the following commutative diagram
            \begin{center}
        \begin{tikzcd}
        \calF_1(Y/\PP^1)(\kbar) \arrow[dr] & \calF_2(\calQ/\PP^1)(\kbar) \arrow[r, "\phi"] \arrow[l, swap, "\tilde{\rho}"]& \Sm{1}(\kbar) \arrow[d] \\
        &(\CH^2Y_{\bar{k}})^{\gamma_1} \arrow[r, "\pje"] & \Pm{1}(\kbar),
        \end{tikzcd} 
        \end{center}
    where \(\phi\) is as in Theorem~\ref{thm:EqnsForCurves}\eqref{part:isoF2toS1} and \(p_*j^*\eps^*\) is as in~\eqref{diagram:beauville-maps}. In particular, the Prym curve \(\Gamma_{\Deltatilde/\Delta}\) is the discriminant cover of the quadric surface fibration \(\piQuadricSurface\).\label{part:CurvesinP1}
\end{enumerate}
\end{prop}
\begin{proof}
Recall that the total space of the family \(\calQ\) is defined by the equation \[t_0^2Q_1 + 2t_0t_1Q_2 + t_1^2Q_3 - (t_0r + t_1s)^2=t_0^2(Q_1-r^2) + 2t_0t_1(Q_2-rs) + t_1^2(Q_3-s^2)=0\] in \(\PP^1_{[t_0:t_1]}\times\PP^4_{[u:v:w:r:s]}\). Define a rational map \(\calQ\dashrightarrow Y\) by \(([t_0:t_1],[u:v:w:r:s])\mapsto([t_0:t_1],[u:v:w:t_0r+t_1s])\). For each \([t_0:t_1]\), the fiber \(\calQ_{[t_0:t_1]}\) is the cone over the quadric surface \(Y_{[t_0:t_1]}\). The induced fiberwise rational map is projection from the cone point \(V(u,v,w,t_0r+t_1s)\) of \(\calQ_{[t_0:t_1]}\). Hence the rational map \(\calQ\dashrightarrow Y\) is defined away from \(V(u,v,w,t_0r+t_1s)\), the closure of each fiber of \(\calQ\dashrightarrow Y\) is a line through the cone point, and composing with the double cover \(\piDoubleCover\) yields the rational map \(\rho\). Blowing up the line \(V(u,v,w,t_0r+t_1s)\) yields a morphism \(\tilde{\rho}\). This shows \eqref{part:YSteinFactorization}.

Consider the restriction of \(\rho\) to \(\PP^1\times\Deltatilde\); note that \(\PP^1\times\Deltatilde\) does not meet the indeterminacy locus of \(\rho\). For each fixed \([t_0:t_1]\), the map \([t_0:t_1]\times\Deltatilde\to[t_0:t_1]\times\Delta\) is the double cover \(\piDelta\). So \(\rho|_{\PP^1\times\Deltatilde}\) factors through and surjects onto \(Y_\Delta\) by definition, and this morphism is an isomorphism away from the ramification divisor of \(Y_\Delta\to\PP^1\times\Delta\). This shows~\eqref{part:normalizationYDeltatilde}.

It remains to show~\eqref{part:CurvesinP1}. From the description of the rational map \(\calQ\dashrightarrow Y\) given above, we see that \(\tilde\rho\) sends projective \(2\)-planes in the fibers of \(\calQ \to \PP^1\) to lines in the fibers of \(Y \to \PP^1\). It follows that \(\tilde\rho\) induces an isomorphism \(\calF_2(\calQ/\PP^1) \xrightarrow{\sim}\calF_1(Y/\PP^1)\). By~\eqref{part:normalizationYDeltatilde} the diagram below commutes, where the maps are as (\ref{sec:beauville}).
    \[   
    \begin{tikzcd}
        & S \arrow[rr, hook, "j"] \arrow[dr, "\nu", swap] \arrow[d, "p", swap] && Y' \arrow[d, "\eps"] & \PP^1 \times \Deltatilde \arrow[d, hook, "\theta"] \arrow[l, hook'] \arrow[bend right = 25, equal, yshift=2]{lll} \\
        & \Deltatilde & Y_\Delta \arrow[r, hook, "r'"] & Y \arrow[dr, "\piQuadricSurface"] & \calQ \arrow[d] \arrow[l, dashed] \\
        &&&& \PP^1
    \end{tikzcd}
    \]

Fix \([t_0:t_1]\), let \(\ell\) be a line in \(Y_{[t_0:t_1]}\), and let \(\Lambda\) be the 2-plane in \(\calQ_{[t_0:t_1]}\) obtained as the cone over \(\ell\). Identifying \(\PP^1\times\Deltatilde\) with \(S\) we have that \(\phi(\Lambda)=p_*\theta^*(\Lambda)\), and \(\theta^*\Lambda=j^*\eps^*\ell\) by commutativity of the above diagram. Therefore \(\phi(\Lambda)=\pje(\tilde{\rho}(\Lambda))\), so the commutativity of the diagram in~\eqref{part:CurvesinP1} holds as claimed.
\end{proof}

\subsection{Intermediate Jacobian torsors of \texorpdfstring{\(Y\)}{Y}}\label{sec:DoubleCoverIJT}

Recall that if \(\Delta\) is a smooth plane quartic curve, then \(\PPrym_{\Deltatilde/\Delta}^{\Pic_{\PP^2/k}}=\PPrym_{\Deltatilde/\Delta}^{\langle K_\Delta\rangle}\).

\begin{thm}[Corollary of Theorem~\ref{thm:IJTtorsorsOfConicBundles}]
\label{thm:IJ-torsors-double-cover}
Let \(k\) be a field of characteristic different from \(2\), let \(\piDelta\colon\Deltatilde\to\Delta\) be a geometrically irreducible \'etale double cover of a smooth plane quartic curve, and let \(Y=Y_{\Deltatilde/\Delta}\) be the threefold constructed in~\eqref{eqn:DoubleCoverConst}.
Let \(\gammatilde_0\) be the algebraic class of a \(\kbar\)-line contained in a singular fiber of the conic bundle \(\piX\) and \(\gamma_1\) the algebraic class of a \(\kbar\)-line in the fiber of the quadric surface fibration \(\piQuadricSurface\). Then \(\NS^2\Ybar=(\NS^2\Ybar)^{G_k}=\Z\gammatilde_0\oplus\Z\gamma_1\), and the morphism from Theorem~\ref{thm:IJTtorsorsOfConicBundles} induces the following isomorphisms of connected components
    \[
        (\bCH2_{Y/k})^{n \gammatilde_0+m\gamma_1} \simeq 
        \begin{cases} 
            P &: \text{ \(n, m\) even,} \\ 
            \Ptilde &: \text{ \(n\) odd, \(m\) even,} \\ 
            \Pm{1} &: \text{ \(n\) even, \(m\) odd,} \\ 
            \Ptildem{1} &: \text{ \(n, m\) odd,} 
        \end{cases}
    \]
    where \(P, \Ptilde, \Pm{1}, \Ptildem{1}\) are as defined in~\eqref{eq:DefnOfPAndPtilde} and Definition~\ref{def:Ptildem}.
\end{thm}
\begin{proof}
The first statement is Proposition~\ref{prop:DoubleCover-structures}\eqref{part:FibralCurves}.
We already have \((\bCH2_{Y/k})^{\gammatilde_0} \simeq \Ptilde\) by Theorem~\ref{thm:IJTtorsorsOfConicBundles}\eqref{part:IsoAsTorsors}.  
Proposition~\ref{prop:BlowupQuadricFibration}\eqref{part:CurvesinP1} implies that \(\pje \gamma_1 \in \Pm{1}\) so 
by Theorem~\ref{thm:IJTtorsorsOfConicBundles}\eqref{part:IsoAsTorsors} and Remark~\ref{rem:maps-are-pje}, we have \((\bCH2_{Y/k})^{\gamma_1} \simeq \Pm{1}\).
The result follows from \eqref{eq:torsors_add} and that \([\Ptildem{1}] = [\Ptilde]+ [\Pm{1}]\) and \([\Ptilde]\) and \([\Pm{1}]\) are \(2\)-torsion (Definition~\ref{def:Ptildem}).   
\end{proof}

    \subsection{Consequences of vanishing IJT obstruction}\label{sec:IJTvanishingConsequences}
Since \(\Pm{1} \simeq \bPic^1_{\Gamma/k}\) by Theorem \ref{thm:EqnsForCurves}\eqref{part:isoPic1ToP1} and the genus of \(\Gamma\) is \(2\), the vanishing of the IJT obstruction for a degree \(4\) conic bundle (Section~\ref{sec:Notation}) is equivalent to each of \(\Ptilde\) and \(\Ptildem{1}\) being isomorphic to either \(\bPic^0_{\Gamma/k}\) or \(\bPic^1_{\Gamma/k}\).  Since \(\Ptildem{1} = \Pm{1} + \Ptilde\), this is equivalent to (at least) one of \(\Ptilde\) or \(\Ptildem{1}\) being the trivial \(P\)-torsor.  When \(\Ptildem{1}\) is trivial, we extract the following information.

\begin{prop}\label{prop:Ptilde1_trivial}
Let \(k\) be a perfect field of characteristic different from \(2\).
Let \(\piX\colon X\to \PP^2\) be a degree \(4\) conic bundle. Given a point \(D \in \Ptildem{1}(k)\),  there exists a line \(L = L_D\) in \(\PP^2\), defined over \(k\), such that \(X_L\to L\) has a regular model \(\hat{X}_L\to L\) with a \(\kbar\)-section \(\hat{\mathfrak{S}}\) whose class in \(\Pic((\hat{X}_L)_{\kbar})\) is Galois-invariant and \(\pje [\im \hat{\mathfrak{S}}] = [D]\in\Ptildem{1}(k)\). Furthermore, if \({X}_L(k)\neq\emptyset\), then there is a \(k\)-rational section with these properties.

In particular, if \(X = Y_{\Deltatilde/\Delta}\) as constructed in Section~\ref{sec:DoubleCover}, the aforementioned section maps with odd degree to \(\PP^1\) under \(\piQuadricSurface\).
\end{prop}

Proposition \ref{prop:Ptilde1_trivial} underpins why the intermediate Jacobian obstruction is insufficient to determine rationality: this Galois-invariant rational equivalence class \([\mathfrak S]\) need not descend to a \(k\)-rational section.  We exploit this idea in the construction used in Theorem~\ref{thm:NoRatlityObsSuff}\eqref{part:IJTnotenough} (see Theorem~\ref{thm:IJTorsorObsNotEnough}).

The obstruction to descending a Galois-invariant rational equivalence class of a section of a conic bundle surface is a Brauer class of order \(2\).  If \(k\) has no nontrivial \(2\)-torsion in its Brauer group (which implies that this section does descend), then we can strengthen Proposition~\ref{prop:Ptilde1_trivial} to Theorem~\ref{thm:SufficiencyBrk2trivial} below, which allows us to deduce that the IJT obstruction characterizes rationality for these double covers (Theorem~\ref{thm:Br2trivial-IJT}).

\begin{lemma}\label{lem:Gk-invar-section}
    Let \(\pi\colon X\to\PP^1\) be a smooth morphism such that every fiber is a conic. Given any finite set \(S\subset X(\kbar)\), there exists a geometric section \(\frakS\) of \(\pi\) that is disjoint from the Galois orbit of \(S\) and such that the rational equivalence class of \(\frakS\) in \(\Pic \Xbar\) is Galois invariant.
\end{lemma}
\begin{proof}
    The assumptions imply that \(X\) is either a Hirzebruch surface \(\F_e\) with distinguished section \(\sigma_e\), or \(C'\times\PP^1\) for a conic \(C'\) \cite{Iskovskih-rational-pencil}*{Theorem 4.1(4)}. We have \(\Pic\Xbar = (\Pic\Xbar)^{G_k}\simeq\Z\oplus\Z\) with the first factor generated by the class of \(\sigma_e\) (in the \(\F_e\) case) or a section given by the same choice of \(\kbar\)-point on each fiber (in the \(C'\times\PP^1\) case), and the second factor generated by the class of a fiber. The Galois orbit of \(S\) is finite, so in either case we may find a geometric section \(\frakS\) of \(\pi\) avoiding this finite set, and its class in \(\Pic\Xbar\) is Galois invariant.
\end{proof}

\begin{proof}[{Proof of Proposition~\ref{prop:Ptilde1_trivial}}]
Since \(\Ptildem{1}\) is isomorphic to \(\Stildem{1}\) by Theorem~\ref{thm:EqnsForCurves}\eqref{part:AJiso}, if \(\Ptildem{1}(k)\neq\emptyset\), then there exists a \(k\)-rational divisor \(D\in\Stildem{1}(k)\).  Since \(\piDelta_*D \in |K_{\Delta}| = |\OO_{\Delta}(1)|\), there exists a unique \(k\)-rational line \(L \subset \PP^2\) such that \(L \cap \Delta = \piDelta_*D\). In particular, \(D = \sum_{\pt\in \Delta\cap L}D_{\pt}\) where each \(D_{\pt}\in (\Sym^{\mult_{\pt}(\Delta\cap L)}_{\Delta}(\Deltatilde))_{\pt}\) has relative degree \(1\) over \(\pt\).

Thus, by Proposition~\ref{prop:pjeExplicit}\eqref{part:IrredComponents}, \(D\) selects a single irreducible component in each fiber of \(\hat{X}_L\) over \(\pt\in\Delta\cap L\) that is split over \(\kk(\pt)\). By iteratively blowing down all other irreducible components in the singular fibers, we obtain a \(k\)-morphism \(\hat{X}_L\to \hat{X}_L^0\), where \(\hat{X}_L^0\to L\) is a conic bundle where every fiber is smooth. Now we apply Lemma~\ref{lem:Gk-invar-section} to \(\hat{X}_L^0\) to obtain a geometric section \(\frakS^0\) whose class in \(\Pic((\hat{X}_L^0)_\kbar)\) is Galois invariant and that doesn't contain any of the (finitely many) points where the map \(\hat{X}_L\to \hat{X}_L^0\) has no inverse. Note that if \(X_L(k)\neq\emptyset\) then the Lang--Nishimura lemma implies that \(\hat{X}_L(k)\neq\emptyset\) and \(\hat{X}^0_L(k)\neq\emptyset\) and so \(\frakS^0\) may be taken to be defined over \(k\). 

By the properties of \(\frakS^0\), the proper transform of \(\frakS^0\) in \(\hat{X}_L\) is equal to its complete preimage and so this curve, which we denote by \(\hat{\frakS}\), gives a Galois invariant class in \(\Pic (\hat{X}_L)_{\kbar}\).  Furthermore, \(\hat{\frakS}\) (and all of its Galois conjugates) intersects the irreducible component of the fiber \((\hat{X}_L)_\pt\) corresponding to \(D_{\pt}\) for all \(\pt\in\Delta\cap L\) (using the correspondence in Proposition~\ref{prop:pjeExplicit}\eqref{part:IrredComponents}).  Therefore, letting \(\frakS:=\im\hat{\frakS}\), by Proposition~\ref{prop:pjeExplicit}\eqref{part:pjeSection}, \(\pje[{\frakS}] = D\), as desired.

If \(X = Y_{\Deltatilde/\Delta}\), then since the coefficient of \(\gammatilde_0\) gives the intersection number with a fiber of \(\piQuadricSurface\), Theorem~\ref{thm:IJ-torsors-double-cover} implies that \(\frakS\) is an odd degree geometric multisection of \(\piQuadricSurface \colon Y \to \PP^1\).
\end{proof}

\begin{thm}\label{thm:SufficiencyBrk2trivial}
Let \(k\) be a perfect field of characteristic not \(2\) with \((\Br k)[2] = 0\) and let \(\piX \colon X \to \PP^2\) be a degree \(4\) conic bundle.  If the IJT obstruction vanishes (Definition \ref{defn:IJT-obstruction}), then there exists a smooth genus \(0\) plane curve \(C \subset \PP^2\) and a section of \(X_C \to C\), all defined over \(k\).
Furthermore, if \(X = Y_{\Deltatilde/\Delta}\), then this curve maps with odd degree to \(\PP^1\) under \(\piQuadricSurface\), and hence \(Y_{\Deltatilde/\Delta}\) is rational.
\end{thm}

Before proving Theorem~\ref{thm:SufficiencyBrk2trivial}, we show how it implies Theorem~\ref{thm:Br2trivial-IJT} and give an immediate corollary in the case of finite fields (Corollary~\ref{cor:rationality-over-finite-fields}).

    \begin{proof}[Proof of Theorem~\ref{thm:Br2trivial-IJT}]
        By the purity exact sequence for Brauer groups, any conic bundle with the same discriminant cover \(\Deltatilde\to\Delta\) differs from one of the form \(Y_{\Deltatilde/\Delta}\) by a class in \((\Br k)[2]\). Since we have assumed that \((\Br k)[2] = 0\), every degree \(4\) conic bundle over \(k\) is birational to one of the form \(Y_{\Deltatilde/\Delta}\).
    
        If \(k\) is perfect, applying Theorem~\ref{thm:SufficiencyBrk2trivial} to \(Y_{\Deltatilde/\Delta}\) yields an odd degree multisection of \(\piX_1 \colon Y_{\Deltatilde/\Delta} \to \PP^1\), and hence \(Y_{\Deltatilde/\Delta}\) is \(k\)-rational by Proposition~\ref{prop:DoubleCover-structures}\eqref{part:OddDegreeMultisection}.  If \(k\) is imperfect, then applying Theorem~\ref{thm:SufficiencyBrk2trivial} to \(Y_{\Deltatilde/\Delta}\) over the perfect closure yields a purely inseparable rational multisection of \(\piQuadricSurface\) over \(k\). This purely inseparable rational section has odd degree since the characteristic of \(k\) is not \(2\), so, again by Proposition~\ref{prop:DoubleCover-structures}\eqref{part:OddDegreeMultisection}, \(Y_{\Deltatilde/\Delta}\) is rational. Hence, the IJT obstruction characterizes rationality for degree \(4\) conic bundles over \(k\).
    \end{proof}

    \begin{cor}\label{cor:rationality-over-finite-fields}
    Let \(\mathbb F_q\) be a finite field of odd characteristic. Then any degree \(4\) conic bundle \(\piX \colon X \to \PP^2\) is \(\mathbb F_q\)-rational.
    \end{cor}
    \begin{proof}
        Lang's theorem implies that the \(P\)-torsors \(\Ptilde,\Pm{1}\), and \(\Ptildem{1}\) are all trivial, and so by Theorem~\ref{thm:Br2trivial-IJT}, \(X\) is rational.
    \end{proof}

\begin{proof}[Proof of Theorem~\ref{thm:SufficiencyBrk2trivial}]
By assumption, \((\Br k)[2] = 0\), so every smooth \(k\)-fiber of \(\piX\) has a \(k\)-point.  Further, since there is a section \(\delta\) of \(X_{\Delta} \to \Delta\), \(\piX((X(k)) \supset \Delta(k)\) and so \(\piX((X(k)) = \PP^2(k)\).

The vanishing of the IJT obstruction is equivalent to either \(\Ptilde(k) \neq \emptyset\) or \(\Ptildem{1}(k) \neq \emptyset\).  In the second case, we apply Proposition~\ref{prop:Ptilde1_trivial} to obtain a line \(L\subset \PP^2\), defined over \(k\), such that \(X_L\to L\) has a \(k\)-section \(\frakS\) such that \(\pje\frakS\in \Ptildem{1}(k)\).   (We obtain the stronger conclusion that \(\frakS\) is defined over \(k\) because \(L(k)\neq\emptyset\) so by the beginning remark \(X_L(k)\neq\emptyset\).) In particular, if \(X = Y_{\Deltatilde/\Delta}\), then \(\frakS\) maps with odd degree to \(\PP^1\), as desired.

It remains to consider the case that \(\Ptilde(k) \neq \emptyset\).  Translating by \(K_{\Deltatilde}\), we see that \(\Ptildem{2}(k)\neq\emptyset\).  By the Riemann--Roch Theorem (noting that \(K_{\Deltatilde}\) itself belongs to \(\Pm{2}\)), the fibers of \(\Stildem{2} \to \Ptildem{2}\) are Severi--Brauer threefolds.  Since \((\Br k)[2]=0\), each fiber over a \(k\)-point is in fact isomorphic to \(\PP^3_k\).  For any degree \(8\) divisor \(D \in \Stildem{2}(k)\), we have \(\piDelta_*D \in |2K_\Delta| = |\OO_\Delta(2)|\).  Since \(|\OO_{\PP^2}(2)| \simeq |\OO_\Delta(2)|\), there exists a unique \(k\)-rational plane conic \(C_D\) such that \(C_D \cap \Delta = \piDelta_*D\). Note that by the first paragraph of the proof \(X_{C_D}(k)\neq\emptyset\). If \(C_D\) is smooth, we argue using Proposition~\ref{prop:pjeExplicit} similiar to the proof of Proposition~\ref{prop:Ptilde1_trivial} as follows.

Let \(\hat{X}_{C_D}\to X_{C_D}\) be the minimal resolution.  By Proposition~\ref{prop:pjeExplicit}\eqref{part:IrredComponents}, \(D\) determines, for each \(\pt\in\Delta\cap C_D\), an irreducible component of the fiber \((\hat{X}_{C_D})_\pt\) that is split over \(\kk(\pt)\).  Iteratively blowing down all other irreducible components, we obtain a conic bundle \(\hat{X}^0_{C_D}\to C_D\) where every fiber is smooth.  Note that by Lang--Nishimura, \(\hat{X}_{C_D}(k)\) and \(\hat{X}^0_{C_D}(k)\) are both nonempty.  Thus by Lemma~\ref{lem:Gk-invar-section}, \(\hat{X}^0_{C_D}\to C_D\) has a \(k\)-rational section \(\frakS^0\) that is disjoint from the finitely many points where \(\hat{X}_{C_D}\to \hat{X}^0_{C_D}\) has no inverse.  Thus the proper transform \(\hat{\frakS}\) of \(\frakS^0\) intersects all the irreducible components corresponding to \(D\) and so, setting \(\frakS := \im\hat{\frakS}\), Proposition~\ref{prop:pjeExplicit}\eqref{part:pjeSection} implies that \(\pje[{\frakS}]=[D]\). 

If \(X  = Y_{\Deltatilde/\Delta}\), then by 
Theorem~\ref{thm:IJ-torsors-double-cover}, we have \([\frakS] = 2\gamma_1 + (2n + 1)\gammatilde_0\) for some integer \(n\).  Therefore \(\frakS\) maps with odd degree onto \(\PP^1\) under \(\piX_1 \colon Y_{\Deltatilde/\Delta} \to \PP^1\).

Finally, if \(C_D\) is not geometrically integral then the divisor \(D\) is (geometrically) the sum of two degree \(4\) divisors from \(\Sm{1} \cup \Stildem{1}\).
Since the sum of two degree \(4\) divisors which are either both from \(\Sm{1}\) or both from \(\Stildem{1}\) necessarily gives rise to a point on \(\Sm{2}\) by Definition~\ref{def:Ptildem}, we must have that \(D\) is geometrically the sum of a divisor from \(\Sm{1}\) and a divisor from \(\Stildem{1}\).
Furthermore, the Galois action must preserve the parity of \(h^0\), and hence a divisor on \(\Sm{1}\) can never be Galois-conjugate to a divisor on \(\Stildem{1}\).
Therefore, in this case we must have \(D = D_1 + \tilde{D}_1\) where \(D_1 \in \Sm{1}(k)\) and \(\tilde{D}_1 \in \Stildem{1}(k)\), and so \(\Ptildem{1}(k) \neq \emptyset\). Hence, the result follows again from Proposition~\ref{prop:Ptilde1_trivial}.  
\end{proof}

\section{Constructions of Theorem~\ref{thm:NoRatlityObsSuff} and Example~\ref{ex:NoRatlityObsOrConst}}
\label{sec:ProofsofMainThms}

In this section, we construct the conic bundles in Theorem~\ref{thm:NoRatlityObsSuff} and Example~\ref{ex:NoRatlityObsOrConst}.  In all three cases, we exhibit conic bundles constructed as double covers of \(\PP^1\times \PP^2\) as in Section~\ref{sec:DoubleCover} that have the desired properties.
The connected components \(\Pm{1}\) and \(\Ptildem{1}\) of the Prym scheme  (defined in Section~\ref{sec:CanonicalPrymScheme}), and their respective preimages \(\Sm{1},\Stildem{1}\subset \Sym^4\Deltatilde\) under the Abel--Jacobi map will play a key role in the proofs.

Part~\eqref{part:IJTnotenough} of Theorem~\ref{thm:NoRatlityObsSuff} follows from Theorem~\ref{thm:IJTorsorObsNotEnough} and part~\eqref{part:RRtopnotenough} follows from Theorem~\ref{thm:IrratlConicBundle}. We construct the threefold verifying Example \ref{ex:NoRatlityObsOrConst} in Theorem~\ref{thm:NoKnownRationality}, which appears directly after the proof of Theorem~\ref{thm:IJTorsorObsNotEnough} since the methods are similar.
All numerical and algebraic claims in the proofs can be verified with \texttt{Magma} code available on \texttt{Github}~\cite{FJSVV_code}.

    \begin{thm}[Theorem~\ref{thm:NoRatlityObsSuff}\eqref{part:IJTnotenough}]\label{thm:IJTorsorObsNotEnough}
        There exists a degree \(4\) conic bundle \(Y\to \PP^2_{\Q}\) such that:
        \begin{enumerate}
        \item \label{part:ex131Qunirational} \(Y\) is \(\Q\)-unirational, hence 
        \(Y(\Q)\neq \emptyset\);
        \item \label{part:IJTvanishes}\( \tilde{P}^{(1)}(\Q) \neq \emptyset \) , which implies that \(\Ptildem{1}\isom \bPic^0_{\Gamma/\Q}\) and \(\Ptilde \isom \Pm{1} \isom \bPic^1_{\Gamma/\Q}\), so the intermediate Jacobian torsor obstruction vanishes; and
        \item \label{part:irrational} \(Y(\RR)\) is disconnected, and hence \(Y\) is  irrational over any subfield of \(\RR\).
        \end{enumerate}
    \end{thm}

    \begin{proof}
    Let
    \[
    \begin{array}{lcrcrcrcrcrcr}
        Q_1 &:=& -31u^2 &+& 12uv &-& 6v^2 &+& 9uw &+& 531vw &+& 25w^2, \\
        Q_2 &:=& -25u^2 &+& 120uv &+ &30v^2 &-& 31uw &+& 37vw,&\\
        Q_3 &:=& -8047u^2 &+& 1092uv & - &1446v^2 &-& 423uw &-& 375vw &-& 25w^2,
    \end{array}
    \]
    and let \(\Deltatilde\to \Delta\) be the curves as defined in Theorem~\ref{thm:EqnsForCurves}\eqref{part:Qis}.  One can check, using the Jacobian criterion, that these curves are smooth and so we may define \(Y = Y_{\Deltatilde/\Delta}\) to be the conic bundle threefold as defined in \eqref{eqn:DoubleCoverConst}.  
    We will prove that \(Y\) has the properties claimed above.

The threefold \(Y\) has a \(\Q\)-point over \(([1:1],[0:0:1])\in\PP^1\times\PP^2\), so part \eqref{part:ex131Qunirational} follows by Proposition \ref{prop:lowdegrationality}(iii). To prove \eqref{part:irrational}, we compute that the quadric surface bundle \(Y\to \PP^1\) geometrically has 6 degenerate fibers, lying over the vanishing of
\begin{gather*}
8813625t_0^6 + 16982610t_0^5t_1 + 2262441955t_0^4t_1^2 
+ 464971196t_0^3t_1^3 - 
     2293725941t_0^2t_1^4 \\- 291034182t_0t_1^5 + 429774609t_1^6.
\end{gather*}
This locus has exactly four real points \([\alpha_i:1]\) which fall in the following intervals: \(\alpha_1 < -0.75 <\alpha_2 < 0 <\alpha_3 < 0.5 <\alpha_4<1\).  One can compute that the quadric surfaces \(Y_t\) for \(t\in \{-0.75,0,0.5,1\}\) have signatures as described in the following table.
\begin{center}
    \begin{tabular}{c||c|c|c|c}
         \(t\) & \(-0.75\) & \(0\) & \(0.5\) &\(1\) \\
         \hline
         Signature & \((0,4)\) & \((1,3)\) & \((0,4)\) & \((1,3)\) 
    \end{tabular}
\end{center}
In particular, \(\pi_1(Y(\RR))\) has two connected components, and so \(Y(\RR)\) must be disconnected.

It remains to prove \eqref{part:IJTvanishes}. We will now show that \(\tilde{P}^{(1)}(\Q) \neq \emptyset\) by exhibiting a \(\Q\)-point of \(\Stildem{1}\) that maps to \(\Delta\cap V(w)\).
One can verify that \(\Delta\cap V(w)\) consists of the four complex points
\[
    \left\{[i: 2: 0],\;
    [-i: 2: 0],\;
    [1-i: 4: 0],\;
    [1+i: 4: 0]\right\},
\]
and that the following points form a \(G_{\Q}\)-invariant set, lie on \(\Deltatilde\), and map onto \(\Delta\cap V(w)\):
\begin{align*}
  [ -i: 2: 0: 4-3i: 52-21i ],\quad
& [ 1-i: 4: 0: 1 + 7i: -41 -143i ],\\
  [ \phantom{-}i: 2: 0: 4+3i: 52+21i ],\quad
& [ 1+i: 4: 0: 1 - 7i: -41 +143i ].
\end{align*}
Let \(y\in \Sym^4\Deltatilde\) denote the collection of these \(4\) points.  Since \(y\) is fixed by \(\Gal(\Qbar/\Q)\) and maps to \(\Delta\cap V(w)\), we have \(y\in \Sm{1}(\Q)\cup \Stildem{1}(\Q)\).  It remains to show that \(y \notin \Sm{1}(\Q)\).

By Theorem~\ref{thm:EqnsForCurves}\eqref{part:isoF2toS1}, every point of \(\Sm{1}\) is obtained by intersecting \(\Deltatilde\) with a \(2\)-plane.  However, one can check that 
\[
\det \left(\begin{bmatrix}
   -i & 2 & 4-3i& 52-21i \\
  \phantom{-}i & 2 & 4+3i & 52+21i \\
  1-i& 4& 1 + 7i& -41 -143i \\
 1+i & 4& 1 - 7i& -41 +143i 
\end{bmatrix}\right) = -2^93^25\neq 0,
\]
and so the support of \(y\) spans a \(3\)-plane.  Hence, \(y\in \Stildem{1}(\Q) = \Ptildem{1}(\Q)\) as desired.
\end{proof}


    \begin{thm}[Example~\ref{ex:NoRatlityObsOrConst}]\label{thm:NoKnownRationality}
        There exists a degree \(4\) conic bundle \(Y\to \PP^2_{\Q}\) such that:
        \begin{enumerate}
        \item \label{part:ex15Qunirational} \(Y\) is \(\Q\)-unirational;
        \item \label{part:IJTvanishes2} \( \tilde{P}^{(1)}(\Q) \neq \emptyset \),
        so the intermediate Jacobian torsor obstruction vanishes (as in Theorem~\ref{thm:NoRatlityObsSuff}\eqref{part:IJTnotenough});
        \item \label{part:3sphere} \(Y(\RR)\) is diffeomorphic to a \(3\)-sphere;
        \item \label{part:unramcohom} \(Y\) has trivial unramified cohomology groups over \(\RR\);  and finally
        \item \label{part:noratconstr} \(\pi_1(Y(\RR))\) is a proper closed subset of \(\PP^1(\RR)\), and so \(\pi_1\) has no section over any subfield \(k\subseteq\RR\). In particular, there is no known rationality construction for \(Y\).
        \end{enumerate}
    \end{thm}
 \begin{proof}
Let
\[
    \begin{array}{lcrcrcrcrcrcr}
        Q_1 &:=& -31u^2 &+& 12uv &-& 6v^2 &+& 4uw &+& 8vw &+& 25w^2, \\
        Q_2 &:=& -25u^2 &+& 120uv &+ &30v^2 &+& 9uw &-& vw,&\\
        Q_3 &:=& -8047u^2 &+& 1092uv & - &1446v^2 &+& 4uw &+& 7vw &-& 25w^2,
    \end{array}
    \]
and let \(\Deltatilde\to \Delta\) be the curves as defined in Theorem~\ref{thm:EqnsForCurves}\eqref{part:Qis}.  One can check, using the Jacobian criterion, that these curves are smooth and so we may define \(Y = Y_{\Deltatilde/\Delta}\) to be the conic bundle threefold as defined in \eqref{eqn:DoubleCoverConst}.    
We will prove that \(Y\) has the properties claimed above.

As in the proof of Theorem~\ref{thm:IJTorsorObsNotEnough}, the fiber of \(Y_{[0:0:1]}(\Q)\neq\emptyset\) so \eqref{part:ex15Qunirational} holds by Proposition \ref{prop:lowdegrationality}\eqref{part:unirational}. Note that, over \(V(w)\), the example we consider here is equal to the example constructed in the proof of Theorem~\ref{thm:IJTorsorObsNotEnough}.  
Since we showed in the proof of Theorem~\ref{thm:IJTorsorObsNotEnough} that there was a \(\Q\)-point of \(\Stildem{1}\simeq\Ptildem{1}\) that maps to \(\Delta\cap V(w)\), the same is true for this example.  This proves \eqref{part:IJTvanishes2}.

\textbf{\eqref{part:3sphere}:} Let \(M_i\) be the quadratic form corresponding to \(Q_i\), set 
\[
M_t := M_1 + 2t M_2 + t^2 M_3 = 
\begin{pmatrix}
    -8047t^2 - 50t - 31 &  546t^2 + 120t + 6  &     2t^2 + 9t + 2 \\
    546t^2 + 120t + 6 & -1446t^2 + 60t - 6     &  \frac{7}{2}t^2 - t + 4 \\
    2t^2 + 9t + 2     &  \frac{7}{2}t^2 - t + 4     &      -25t^2 + 25 
\end{pmatrix}, 
\]
and let \(M_t^{(i)}\) denote the minor of the top left \(i\times i\) block for \(i=1,2,3\). Observe that
\begin{align*}
    M_t^{(1)} & = -8047t^2 - 50t - 31 = -8047 \left(t + \frac{25}{8047}\right)^2 - \frac{248832}{8047},\\
    M_t^{(2)} & = 11337846t^4 - 541560t^3 + 69156t^2 - 3000t + 150\\
    & = 6 \left(1061 \left(t + \frac{11}{1061}\right)^2 + \frac{5184}{1061} \right) \left(1781 \left(t - \frac{61}{1781}\right)^2 + \frac{5184}{1781}\right),
\end{align*}
so in particular \(M_t^{(1)}\) is strictly negative and \(M_t^{(2)}\) is strictly positive.  In particular, the reciprocals of \(M_t^{(1)}\) and \(M_t^{(2)}\) are \(C^\infty\) functions.  Since \(\textup{Diag}([M_t^{(1)},M_t^{(2)}/M_t^{(1)},M_t^{(3)}/M_t^{(2)}])\) is similar to \(M_t\), and the diagonalization change of basis matrix requires inverting only \(M_t^{(1)}\) and \(M_t^{(2)}\), we see that \(Y(\RR)\) is diffeomorphic to the real \(3\)-manifold \(S\) defined by \(M_t^{(1)} u^2 + M_t^{(2)}/M_t^{(1)} v^2 + M_t^{(3)}/M_t^{(2)} w^2 -z^2\).   

One can compute that \(M_t^{(3)}\) is separable, has exactly two real roots \(\alpha_i\) which fall in the intervals \(\alpha_1<0<\alpha_2<2\), and that \(M_t^{(3)}(0)\) is positive. In particular, \((t-\alpha_1)(t-\alpha_2)/M_t^{(3)}\) is strictly negative. So after scaling \(u,v,w\) by \(\sqrt{-1/M_t^{(1)}}, \sqrt{-M_t^{(1)}/M_t^{(2)}}\) and \(\sqrt{-M_t^{(2)}(t-\alpha_1)(t-\alpha_2)/M_t^{(3)}}\), respectively, \(S\) is diffeomorphic to the real \(3\)-manifold defined by
\[
    \left(\frac{u}{w}\right)^2 + \left(\frac{v}{w}\right)^2 + \left(\frac{z}{w}\right)^2 = -(t-\alpha_1)(t - \alpha_2) = -\left(t - \frac{\alpha_1+\alpha_2}{2}\right)^2 + \frac{(\alpha_1-\alpha_2)^2}{4}.
\]
Moving \((t - \frac{\alpha_1+\alpha_2}{2})^2\) to the left hand side makes evident that this \(3\)-manifold (and hence \(Y(\RR)\)) is diffeomorphic to the \(3\)-sphere.  Note that this last equation also shows that \(\pi_1(Y(\RR)) = \{[1:t]\mid \alpha_1 \leq 
t \leq \alpha_2\} \subset\PP^1(\RR)\), which proves~\eqref{part:noratconstr}.

Finally, for \eqref{part:unramcohom}, since \(\Br Y \simeq\Br \RR\) by the Artin--Mumford exact sequence \cite{Poonen-Qpoints}*{Theorem 6.8.3, Proof of Proposition 6.9.15}, the unramified cohomology groups of \(Y\) are trivial over \(\RR\) \cite{bw-cg}*{Theorem 1.4 and following paragraphs}.
 \end{proof}


    \begin{thm}[Theorem~\ref{thm:NoRatlityObsSuff}\eqref{part:RRtopnotenough}]\label{thm:IrratlConicBundle}
        There exists a degree \(4\) conic bundle \(Y\to \PP^2_{\Q}\) such that:
        \begin{enumerate}
        \item \label{part:ex132Qunirational} \(Y\) is \(\Q\)-unirational;
        \item \label{part:3sphere2} \(Y(\RR)\) is connected; and
        \item \label{part:RealPointsTorsors} \(\bPic^m_{\Gamma/\mathbb R}(\RR)\neq \emptyset\) for all \(m\in \mathbb N\) and
            \(\Ptildem{1}(\RR) = \emptyset\); in particular, there is an intermediate Jacobian torsor obstruction over \(\RR\), and hence \(Y\) is irrational over any subfield \(k \subset \RR\).
        \end{enumerate}
    \end{thm}

    \begin{proof}
        Define
        \[
            Q_1 := -u^2 - v^2 - 3w^2, \quad
            Q_2 := 3u^2 + 5v^2, \quad \textup{and}\; 
            Q_3 := -7u^2 - 23v^2 - 12w^2,
        \]
        and let \(\Deltatilde\to \Delta\) be the curves as defined in Theorem~\ref{thm:EqnsForCurves}\eqref{part:Qis}.  One can check, using the Jacobian criterion, that these curves are smooth and so we may define \(Y = Y_{\Deltatilde/\Delta}\) to be the conic bundle threefold as defined in \eqref{eqn:DoubleCoverConst}.  
        We will prove that \(Y\) has the desired properties.

        Note that \(Y\) has a \(\Q\)-point over \(([2:1],[1:0:0])\in\PP^1\times\PP^2\), so again ~\eqref{part:ex132Qunirational} follows by Proposition \ref{prop:lowdegrationality}(iii). 
        A signature computation shows that the fiber \(Y_{[t_0:t_1]}\) has real points exactly for \(\frac{t_0}{t_1}\in [3-\sqrt{2}, 5+\sqrt{2}]\) and real lines exactly for \(\frac{t_0}{t_1}\in [5-\sqrt{2},3+\sqrt{2}] \subset [3-\sqrt{2}, 5+\sqrt{2}]\).  In particular, \(Y(\mathbb R)\) is nonempty and connected, proving~\eqref{part:3sphere2}.

        It remains to show \eqref{part:RealPointsTorsors}. Multiples of the four real Weierstrass points \(\{3 \pm\sqrt{2}, 5 \pm \sqrt2\}\) give real points  on \(\bPic^m_{\Gamma/\mathbb R}\) for all \(m\), which gives the first part of \eqref{part:RealPointsTorsors}. 
        It remains to show that \(\Ptildem{1}(\RR) = \emptyset\), which, by Theorem~\ref{thm:EqnsForCurves}\eqref{part:AJiso}, is equivalent to showing that \(\Stildem{1}(\RR) = \emptyset.\)  We use the action of complex conjugation to deduce the following constraints on the image of \(\Stildem{1}(\RR)\).
        
        \begin{lemma}\label{lem:RealConditions}
            Let \(\piX\colon X\to\PP^2\) be a degree \(4\) conic bundle over \(\mathbb R\).
            Define \(\Sigma \subset \check{\PP}^2(\RR)\) to be the locus of lines that meet every point of \(\Delta(\RR)\) with even multiplicity.  
            If \(\Deltatilde(\RR) = \emptyset\), then 
            \[
            \piDelta_*(\Sm{1}(\RR))\cap \piDelta_*(\Stildem{1}(\RR)) = \emptyset\quad\textup{ and }\quad 
            \piDelta_*(\Sm{1}(\RR))\cup \piDelta_*(\Stildem{1}(\RR)) \subset \Sigma.
            \]
        \end{lemma}
        \begin{proof}
            Let \(s\in \Sm{1}(\RR)\) and \(\tilde{s}\in\Stildem{1}(\RR))\) be such that \(\piDelta_*s = \piDelta_*\tilde{s}\).  Then Corollary~\ref{cor:SandStildeDifferByOdd} implies that \(\deg(s\cap\tilde{s}) \equiv 1\bmod 2\).  However, \(s\cap \tilde{s}\) is a \(G_{\RR}\) invariant effective divisor, which, since \(\Deltatilde(\RR) = \emptyset\), must have even degree.    
            Thus, \(\piDelta_*(\Sm{1}(\RR))\cap \piDelta_*(\Stildem{1}(\RR)) = \emptyset\).  
            
            Now let \(\ell\in\check{\PP}^2(\RR)-\Sigma\); since \(\Delta\) has degree \(4\), there exists a real point \(x\in \ell\cap \Delta\) such that the intersection is transverse at \(x\).  Since \(\Deltatilde(\RR) = \emptyset\),  \(\Delta_x\) must be an irreducible degree \(2\) point, i.e., there is no Galois-invariant choice of a single point of \(\Deltatilde\) lying above \(x\).   Thus \(\ell\notin \piDelta_*(\Sm{1}(\RR))\cup \piDelta_*(\Stildem{1}(\RR))\).
        \end{proof}
        \begin{remark}
        Note that there is always a real line that does not meet \(\Delta(\RR)\) \cite{R19}, so \(\Sigma\neq\emptyset\).
        \end{remark}
        Observe that \(Q_1\) is negative definite.  Since \(\Deltatilde \subset V(Q_1 - r^2)\subset \PP^4\) (see Section~\ref{sec:UnramDoubleCoverDeg4}), this implies that \(\Deltatilde(\RR) = \emptyset\).  Therefore, by Lemma~\ref{lem:RealConditions}, to prove that \(\Stildem{1}(\RR) = \emptyset\) it suffices to show that \(\piDelta_*(\Sm{1}(\RR)) = \Sigma\).
        
        By Proposition~\ref{prop:BlowupQuadricFibration}\eqref{part:CurvesinP1}, real lines in the fibers of \(\pi_1\) give points in \(\Sm{1}\).  We will show (with assistance of \texttt{Magma}) that the real lines on the quadric surfaces \(Y_{[t_0:t_1]}\), which only exist for \(\frac{t_0}{t_1}\in [5-\sqrt{2},3+\sqrt{2}]\), surject onto the real points of \(\Sigma\).
        
        Our verification requires a description of the set \(\Sigma\), but first we describe the real points of \(\Delta\).  Observe that the quadratic forms \(Q_i\) are linear in \(u^2, v^2, w^2\), so we may take the quotient of \(\Delta\) by \(\mu_2\times\mu_2\).  This gives a model of \(\Delta\) as a \(4\)-to-\(1\) cover of the conic \(C \colonequals V(-2a^2 -2b^2 + c^2) \) where the map sends \([u:v:w]\mapsto [4u^2 - 33w^2 : 4v^2 - 81w^2 : 126w^2].\)
        Since \(C\) has a rational point, it is isomorphic to \(\PP^1\), and one can verify that the image of \(\Delta(\RR)\to C(\RR)\) is a single connected component and that its preimage is also connected.  In particular, \(\Delta(\RR)\) consists of a single topological oval that is symmetric about \([0:0:1]\).

        \begin{lemma}\label{lem:Sigmalines}
            Let \(\Sigma^0\subset \Sigma\) be the open subset consisting of real lines that don't meet \(\Delta(\RR)\) (which is nonempty by~\cite{R19}). Assume that \(\Delta(\RR)\) has a single connected component.      
            \begin{enumerate}
                \item  We have \(\Sigma = \overline{\Sigma^0}\).  Furthermore, if \(L\in \Sigma^0\), then \(\Sigma = \cup_{x\in L} \overline{\calL_x\cap \Sigma^0}\), where \(\calL_x\) is the pencil of lines through \(x\) and the closure is taking place in \(\calL_x\). \label{part:ClosureSigma0}
                \item If \(L\in \Sigma^0\)  and \(x\in L(\RR)\), the intersection of \(\Sigma^0\) with the pencil \(\calL_x\) of lines through \(x\) is nonempty and connected.\label{part:connectedSigma0}
            \end{enumerate}
        \end{lemma}
        \begin{proof}
            \textbf{\eqref{part:ClosureSigma0}:}             From the definition of \(\Sigma\) and \(\Sigma^0\), it is clear that we have \(\overline{\Sigma^0}\subset \Sigma\); it remains to prove the reverse containment. By~\cite{R19}, \(\Sigma^0\neq\emptyset\), so fix an \(L\in \Sigma^0\), and consider the affine plane \(U := \PP^2 - L\).  Since \(L\in \Sigma^0\), \(\Delta(\RR)\) is contained completely within \(U\). 
            
            Let \(L'\in \Sigma-\Sigma^0\) and let \(x'\in \Delta(\RR)\cap L'\).  Note that \(U - L'\) has two connected components, so there are two ``sides'' of \(L'\) (in U).  Since \(\mult_x(\Delta\cap L')\) is even (and positive), locally near \(x'\), \(\Delta(\RR)\) must be on one side of \(L'\).  Since this holds for all \(x'\in \Delta(\RR)\cap L'\), \(\Delta(\RR)\) must be on one side of \(L'\).  Thus, we may move \(L'\) away from \(\Delta\) and land inside \(\Sigma^0\).  Hence, \(\overline{\Sigma^0} = \Sigma\).  Furthermore, if \(x = L\cap L'\), then (in the affine chart \(\PP^2-L\)) we may move \(L'\) to another line through \(x\) in the direction away from \(\Delta\), which shows that \(L'\in \overline{\calL_x\cap \Sigma^0}\).
            
            \textbf{\eqref{part:connectedSigma0}:} Consider the morphism \(f\colon \Delta \to \calL_x\) induced by projection from \(x\) (this is a morphism since \(x\notin\Delta\)). By definition, \(\calL_x\cap\Sigma^0 = \calL_x(\RR)-f(\Delta(\RR))\). By assumption, \(L\in \calL_x\cap \Sigma^0\) so we have \(L\notin f(\Delta(\RR))\).
            Since \(\Delta(\RR)\) is connected, its image under \(f\) is also connected, and so the nonempty complement, which is \(\Sigma^0\cap \calL_x\), is also connected.
        \end{proof}
        
    Note that \(L_{\infty}:= V(w)\in \Sigma\) does not meet \(\Delta\) at any \(\RR\)-points.  We will show that for all \(x\in L_{\infty}(\RR)\) and all \(L'\in \calL_x\cap \Sigma\), the line \(L'\) is the image of a line on \(Y_{[t_0:t_1]}\) for some \(\frac{t_0}{t_1}\in [5-\sqrt2, 3 + \sqrt2]\).  (For example, the image of the line \(V(w, z - 2\sqrt{\sqrt{2} - 1}v)\) on
    \[
    Y_{[3 + \sqrt2:1]} : \;4(\sqrt2-1)v^2 + (-18\sqrt2 - 45)w^2 = z^2.
    \] is \(L_{\infty}\).) By Lemma~\ref{lem:RealConditions}, this will imply that \(\Ptildem{1}(\RR)=\emptyset\).
    
    Since \(Y_{[t_0:t_1]}\simeq \PP^1_{\RR}\times \PP^1_{\RR}\) for \([t_0:t_1]\) in this range, a line \(L'\in \calL_x\cap\Sigma\) is the image of a line on \(Y_{[t_0:t_1]}\) if and only if \(L'\) is tangent to the branch conic of the double cover \(Y_{[t_0:t_1]}\to \PP^2\) given by the restriction of \(\pi\).  This tangency condition is encoded by a quadratic form\footnote{The quadratic form is called \texttt{tEqn} in the code \cite{FJSVV_code}.} in \(t_0,t_1\) (whose coefficients depend quadratically on the pencil \(\calL_x\)) having roots in the interval \( I:=[5-\sqrt2, 3 + \sqrt2]\).  Since \(\Delta\) and all of the quadrics \(Q_i\) are symmetric about the origin, it suffices to prove this for the lines that intersect the positive parts of the \(u\) and \(v\) axis, which all have negative slope.

    Let \(x\in L_\infty(\RR)\) be a point corresponding to a line with negative slope.  By Lemma~\ref{lem:Sigmalines}, there are two boundary points of \(\calL_x\cap\Sigma^0\), and exactly one of these corresponds to a line \(L'\) that intersects the positive parts of the \(u\) and \(v\) axis. For this \(L'\), the quadratic form  has a double real root \(r_0\in I\). 
Furthermore, for the lines in \(\calL_x\cap \Sigma^0\), we verify that the quadratic form is nonnegative at the point \(r_0\) and nonpositive at the end points of \(I\).  Thus, by the intermediate value theorem, the quadratic form must have a real root somewhere in \(I\), and hence every line in \(\calL_x\cap \Sigma^0\) is the image of a real line on \(Y_{[t_0:t_1]}\).

Since the above argument extends by symmetry to all \(x\in L_\infty(\RR)\), we have shown that every \(L\in\Sigma\) is the image of a real line on \(Y_{[t_0:t_1]}\), i.e. \(\varpi_*(\Sm{1}(\RR))=\Sigma\).
Thus, we have shown that  \(\Ptildem{1}(\RR)=\emptyset\), so \(\Ptildem{1}\) is not isomorphic over \(\RR\) to \(\bPic^m_{\Gamma/\mathbb R}\) for any \(m\). By~\cite{bw-ij}*{Thm. 3.11(iii)}, \(Y\) must be irrational over any subfield of \(\RR\).
\end{proof}

\section{Contextual results}\label{sec:Context}

 \subsection{Rationality and unirationality of low degree conic bundle threefolds}\label{sec:low_degree}

 Geometrically standard conic bundle threefolds \(\pi \colon X \to \PP^2\) with \(\pi\) smooth away from a curve \(\Delta\subset \PP^2\) are geometrically rational if \(\deg\Delta\leq 4\) \cite{prokhorov-rationality-conic-bundles}*{Corollary 5.6.1} and, in characteristic 0, if \(\deg\Delta=5\) and the discriminant cover \(\Deltatilde \to \Delta\)  corresponds to an even theta characteristic \cite{Panin80}. They are geometrically irrational if \(\deg\Delta\geq 6\) \cite{beauville-ij}*{Th\'eor\`eme 4.9} or if \(\deg\Delta=5\) and \(\Deltatilde \to \Delta\) is given by an odd theta characteristic \cite{Shokurov-Prym}*{Main Theorem}.
 
 When \(X(k) \neq \emptyset\) and \(\deg \Delta\leq 3\),  
  the proof of rationality in \cite{prokhorov-rationality-conic-bundles}*{Corollary 5.6.1} can be modified to hold over any field \(k\) by taking a pencil of lines through the image of a \(k\)-point on \(X\).
  Similarly, the proof in the case when   \(\deg \Delta = 4\) can be modified to hold when \(\Deltatilde(k) \neq \emptyset\) by taking a pencil of lines through the image of a \(k\)-point on \(\Deltatilde\).
 
 When \(X\) has a \(k\)-point on a smooth fiber of \(\piX\) and \(\deg\Delta\leq 7\), then \(X\) is unirational by unirationality results for conic bundle surfaces. Specifically, if \(\pt\in\PP^2(k)\) is the image of the given \(k\)-point, then the generic fiber of \(X\times_{\PP^2}(\Bl_{\pt}\PP^2)\to\PP^1\) is a conic bundle surface over the generic fiber of \(\Bl_{\pt}\PP^2\to\PP^1\) and admits a \(\kk(\PP^1)\)-point, and so is unirational by \cite{Koll'arMella}*{Corollary 8}. In summary:
 
     \begin{prop}\label{prop:lowdegrationality}
        Let \(\piX\colon X\to\PP^2\) be a geometrically standard and geometrically ordinary conic bundle over a field \(k\) of characteristic different from \(2\). Then:
        \begin{enumerate}
            \item If \(X(k)\neq\emptyset\) and \(\deg\Delta\leq 3\), then \(X\) is \(k\)-rational.
            \item If \(\Deltatilde(k)\neq\emptyset\) and \(\deg\Delta=4\), then \(X\) is \(k\)-rational.
            \item\label{part:unirational} If \(\piX\) has a smooth fiber with a \(k\)-point and \(\deg\Delta\leq 7\), then \(X\) is \(k\)-unirational.
        \end{enumerate}  
    \end{prop}

\subsection{Conic bundles and the intermediate Jacobian obstruction.}\label{sec:IJconicbundlescontext}
As mentioned in the introduction, over \(\C\), the intermediate Jacobian of a minimal conic bundle \(X\to W\) over a rational surface is isomorphic to the Prym variety of its discriminant cover \(\Deltatilde\to\Delta\) \cites{beauville-ij,BeltramettiFrancia83}. Shokurov proved that, over \(\C\), the intermediate Jacobian obstruction characterizes rationality for standard conic bundles over minimal rational surfaces \cite{Shokurov-Prym}*{Theorem 10.1}.

Over nonclosed fields, the rationality problem has also been studied for conic bundle threefolds that are not geometrically standard. Benoist and Wittenberg have previously constructed an example of a conic bundle \(X\to W\) with geometrically \emph{reducible} discriminant cover such that \(X\) is \(\C\)-rational but not \(\RR\)-rational and for which \(\bCH2_{X/\RR}=0\); in particular, the IJT obstruction vanishes for \(X\) \cite{bw-cg}*{Theorem 5.7}. In their example \(W(\RR)\) is disconnected, and so \(X\) has a Brauer group obstruction to rationality.

\subsection{Relation to complete intersections of two quadrics}\label{sec:TwoQuadrics}

    As mentioned in the introduction, work of Hassett--Tschinkel, Benoist--Wittenberg, and Kuznetsov--Prokhorov combine to show that in characteristic \(0\) the intermediate Jacobian obstruction characterizes rationality for all geometrically rational Fano threefolds of geometric Picard rank \(1\)~\cites{HT-intersection-quadrics, bw-ij, KP-Fano-3folds-rank1}. One such class of threefolds, namely that of smooth complete intersection of two quadrics \(Z\subset \PP^5\) containing a conic \(C\), is known to give rise to conic bundle threefolds; indeed, the projection from this conic \(\Bl_C Z\to\PP^2\) has the structure of a conic bundle ramified over a quartic curve \cite{HT-intersection-quadrics}*{Remark 13}.
    
    As we describe in Proposition~\ref{prop:BW3.10} and Corollary~\ref{cor:BlowupIJT} below, the results of Benoist and Wittenberg imply that since the IJT obstruction characterizes rationality for \(Z\), it also characterizes rationality for \(\Bl_C Z\). In particular, the conic bundle from Theorem~\ref{thm:NoRatlityObsSuff}\eqref{part:IJTnotenough} is not isomorphic to \(\Bl_CZ\) for an intersection of quadrics \(Z\subset \PP^5\) containing a conic \(C\). However, these results do not imply, a priori, that the conic bundles from Theorem~\ref{thm:NoRatlityObsSuff}\eqref{part:RRtopnotenough} and Example~\ref{ex:NoRatlityObsOrConst} are not of this form.  To demonstrate this, we show in Proposition~\ref{prop:NecConds2Quadrics} that \(\Bl_CZ\) has a Prym curve with a rational point, and one can verify that the Prym curve from Theorem~\ref{thm:NoRatlityObsSuff}\eqref{part:RRtopnotenough} and Example~\ref{ex:NoRatlityObsOrConst} have no rational points over \(\Q\) (in fact, no \(\Q_3\)-points; see the \texttt{Magma} code in the \texttt{Github} repository for verification \cite{FJSVV_code}).
    
    \begin{prop}[Special case of~\cite{bw-ij}*{Proposition 3.10}]\label{prop:BW3.10}
        Let \(Z \subset \PP^5\) be a smooth intersection of two quadrics that contains a conic \(C\) over a field \(k\). Then there is an isomorphism of group schemes
        \(\bCH2_{Z/k}\times\bPic_{C/k}\xrightarrow{\sim}\bCH2_{\Bl_C Z/k}\)
        respecting the principal polarizations.\hfill \qed
    \end{prop}

    \begin{cor}[{Corollary of  \cite{HT-intersection-quadrics}*{Theorem 14}, \cite{bw-ij}*{Theorem 4.7} and the above}]\label{cor:BlowupIJT}
        Let \(Z \subset \PP^5\) be a smooth intersection of two quadrics  containing a conic \(C\). Then the IJT obstruction characterizes rationality for \(\Bl_C Z\).
    \end{cor}

\begin{proof}
If \(\Bl_C Z\) is rational, then the intermediate Jacobian torsor obstruction vanishes by \cite{bw-ij}*{Theorem 3.11}. For the reverse implication, if the IJT obstruction vanishes for \(\Bl_C Z\) then, by Proposition~\ref{prop:BW3.10}, it also vanishes for \(Z\) and so \(k\)-rationality of \(Z\) follows from \cite{HT-intersection-quadrics}*{Theorem 14} \cite{bw-ij}*{Theorem 4.7}.
\end{proof}

    \begin{prop}\label{prop:NecConds2Quadrics}
        Let \(Z\subset \PP^5\) be a smooth complete intersection of two quadrics that contains a conic \(C\).  Then \((\bCH2_{Z/k})^0 \simeq (\bCH2_{\Bl_C Z/k})^0 \isom \bPic^0_{\Gamma_Z}\) where \(\GammaZ\) is the genus two curve produced by the Stein factorization of the Fano variety of \(2\)-planes in the pencil of quadrics \(\calF_2(\Bl_Z\PP^5/\PP^1) \to\GammaZ\to\PP^1\).  Moreover,  \(\GammaZ(k)\neq\emptyset\).
    \end{prop}
    \begin{proof}
        The isomorphism \((\bCH2_{Z/k})^0\simeq\bPic^0_{\Gamma_Z}\) is shown in \cite{HT-intersection-quadrics}*{Section 11.6}
        and \cite{bw-ij}*{Theorem 4.5}, and by Proposition~\ref{prop:BW3.10}, we have the isomorphism \((\bCH2_{Z/k})^0\simeq(\bCH2_{\Bl_C Z/k})^0\). The conic \(C\) is the intersection of \(Z\) with a 2-plane in a member of the associated pencil, so \(C\) corresponds to a \(k\)-point on \(\GammaZ\) \cite{HT-intersection-quadrics}*{Section 2.3}.
    \end{proof}
    
    \begin{rmk}
    The condition of Proposition~\ref{prop:NecConds2Quadrics} is necessary but not sufficient for a degree \(4\) conic bundle to arise from an intersection of two quadrics. If \(Y\) is the conic bundle of Theorem~\ref{thm:NoRatlityObsSuff}\eqref{part:IJTnotenough} then \(\Gamma(\RR)\neq\emptyset\), but by Corollary~\ref{cor:BlowupIJT}, \(Y_\RR\) cannot be obtained from a complete intersection of quadrics by projection from a conic over \(\RR\).
    \end{rmk}
    
\subsubsection{Complete intersections of two quadrics over fields with trivial Brauer group} In Theorem~\ref{thm:Br2trivial-IJT} we showed that over a field with \((\Br k)[2]=0\), the intermediate Jacobian torsor obstruction characterizes \(k\)-rationality for degree \(4\) conic bundles \(X\to\PP^2\). We end this section with an example, suggested by the referee, which shows that, over such fields, \(X\) can still have an IJT obstruction to rationality. Using the intersections of two quadrics in \cite{bw-ij}*{Theorem 4.14}, we construct a conic bundle \(\pi\colon X\to\PP^2\) with degree \(4\) discriminant over a field with \(\Br k = 0\) and such that \(X\) has an IJT obstruction to rationality.
\begin{example}\label{exmp:Brk=0-IJT-example} 
    Let \(\kappa\) be an algebraically closed field of characteristic different from \(2\), and let \(k=\kappa((t))\). For \(a_0,\ldots,a_5\in\kappa\) pairwise distinct elements, define the diagonal matrices as in~\cite{bw-ij}*{Theorem 4.14} 
    \[
    M_1 := \textup{Diag}([t,t,1,1,1,1]), \quad
    M_2 := \textup{Diag}([a_0t,a_1t,a_2,a_3,a_4,a_5]).
    \] 
    Let \(Q_i\subset\mathbb P^5\) be the quadric corresponding to \(M_i\), and let \(Z=Q_1\cap Q_2\). Benoist and Wittenberg prove that \(Z\) has an intermediate Jacobian torsor obstruction to rationality~\cite{bw-ij}*{Theorem 4.14}. 
    Furthermore, the quadric \(Q_1\) has square discriminant and since \(k\) is a \(C_1\)-field~\cite{CTS-Brauer-book}*{Theorem 1.2.13}, \(Q_1\) contains a plane defined over \(k\) and hence \(Z\) contains a conic \(C\) defined over \(k\) \cite{HT-intersection-quadrics}*{Remark 13}.  Thus, \(X:= \Bl_C Z\) is an irrational degree \(4\) conic bundle.
\end{example}


\bibliographystyle{alpha}
\bibliography{references}

\end{document}